\documentclass[12pt]{amsart}
\usepackage{amssymb,amscd}
\usepackage[colorlinks=true,allcolors = blue]{hyperref} % my new

\let\frak\mathfrak

\def\>{\relax\ifmmode\mskip.666667\thinmuskip\relax\else\kern.111111em\fi}
\def\<{\relax\ifmmode\mskip-.333333\thinmuskip\relax\else\kern-.0555556em\fi}
\def\vsk#1>{\vskip#1\baselineskip}
\def\vv#1>{\vadjust{\vsk#1>}\ignorespaces}
\def\vvn#1>{\vadjust{\nobreak\vsk#1>\nobreak}\ignorespaces}
 \let\alb\allowbreak
\def\fratop{\genfrac{}{}{0pt}1}
\def\satop#1#2{\fratop{\scriptstyle#1}{\scriptstyle#2}}
\let\dsize\displaystyle  \let\ssize\scriptstyle
\let\sssize\scriptscriptstyle 
 \let\vp\vphantom \let\hp\hphantom

\let\Medskip\medskip
\def\medskip{\par\Medskip}
\let\Bigskip\bigskip
\def\bigskip{\par\Bigskip}

\let\Maketitle\maketitle
\def\maketitle{\Maketitle\thispagestyle{empty}\let\maketitle\empty}

\newtheorem{thm}{Theorem}[section]
\newtheorem{cor}[thm]{Corollary}
\newtheorem{lem}[thm]{Lemma}
\newtheorem{prop}[thm]{Proposition}

\numberwithin{equation}{section}

\theoremstyle{definition}
\newtheorem*{rem}{Remark}
\newtheorem*{example}{Example}

\let\mc\mathcal
\let\nc\newcommand

\let\al\alpha
\let\bt\beta
\let\dl\delta
\let\Dl\Delta

\let\eps\varepsilon
\let\gm\gamma
\let\Gm\Gamma
\let\ka\kappa
\let\la\lambda

\let\pho\phi
\let\phi\varphi
\let\si\sigma
\let\Si\Sigma

\let\Tht\Theta
\let\tht\theta

\let\Om\Omega

\let\der\partial
\let\Hat\widehat

\let\ox\otimes
\let\Tilde\widetilde
\let\bra\langle
\let\ket\rangle

\let\ge\geqslant
\let\geq\geqslant
\let\le\leqslant
\let\leq\leqslant

\let\on\operatorname
\let\bi\bibitem
\let\bs\boldsymbol

\def\C{{\mathbb C}}
\def\Z{{\mathbb Z}}
\def\R{{\mathbb R}}

\def\Cc{{\mc C}}
\def\F{{\mc F}}
\def\Hc{{\mc H}}
\def\Ic{{\mc I}}

\def\Mc{{\mc M}}

\def\+#1{^{\{#1\}}}
\def\lsym#1{#1\alb\dots\relax#1\alb}
\def\lc{\lsym,}

\def\Hom{\on{Hom}}
\def\id{\on{id}}

\def\Res{\on{Res}}

\def\tr{\on{tr}}

\def\cirs{{\raise.2ex\hbox{$\sssize\circ$}}}

\def\albt{\al,\>\bt}
\def\btal{\bt\<,\>\al}
\def\blai{\bla,\<\>i}
\def\ii{i,\<\>i}
\def\ij{i,\<\>j}
\def\ik{i,\<\>k}

\def\ji{j,\<\>i}

\def\jk{j,\<\>k}
\def\ki{k,\<\>i}

\def\mm{m,\<\>m}
\def\Ii{I\<\<,\<\>i}

\def\IJ{I\<\<,\<\>J}

\def\JJ{J,\<\>J}
\def\JK{J,\<\>K}

\def\pp{p,\<\>p}

\def\gl{\mathfrak{gl}}

\def\glnn{{\gl_n}}

\def\beq{\begin{equation}}
\def\eeq{\end{equation}}
\def\be{\begin{equation*}}
\def\ee{\end{equation*}}

\nc{\bea}{\begin{eqnarray*}}
\nc{\eea}{\end{eqnarray*}}
\nc{\bean}{\begin{eqnarray}}
\nc{\eean}{\end{eqnarray}}
\nc{\Ref}[1]{{\rm(\ref{#1})}}

\let\ga\gamma
\let\Ga\Gamma

\def\Cxs{\C^\times}
\def\Cxx{{\Cxs}}
\def\mb{\bs m}
\def\bss{\bs s}
\def\Thb{\bs\Tht}
\nc{\bla}{{\bs\la}}
\nc{\Il}{{\Ic_{\bla}}}
\nc{\Fla}{\F_\bla}
\nc{\tfl}{{T^*\<\Fla}}
\nc{\GL}{{GL_n(\C)}}
\nc{\GLC}{{GL_n(\C)\times\C^*}}
\def\tfls{{T^*\<\<\Fla}}

\def\dti{\tilde d}
\def\et{\tilde e}
\def\ellt{\tilde\ell}

\def\blat{{\tilde{\bs\la}}}
\def\qti{\tilde q}
\def\kat{\tilde\ka}
\def\dch{\check d}
\def\tch{\check t}
\def\gac{\check\gm}

\def\gmd{\,\acute{\!\gm}}
\def\GGd{\acute\GG}
\def\Sid{\acute\Si}
\def\thd{\acute\tht}

\def\Thd{\acute\Thb}
\def\tdd{\acute t}
\def\zdd{\acute z}
\def\ttd{\>\acute{\<\<\TT}}
\def\zzd{\>\acute{\<\<\zz}}

\def\pii{\pi\sqrt{\<-1}}

\let\Dx D

\def\xxx{x_1\lc x_n}

\def\zzz{z_1\lc z_n}

\def\Ima{I^{\<\>\max}}
\def\Imi{I^{\<\>\min}}

\def\Sla{S_{\la_1}\!\<\lsym\times S_{\la_N}}

\def\It{\tilde I}

\let\sd s %% \def\sd{\dot s}

\def\Wt{{\<\>\Tilde{\<W\<}\<\>}}

\def\Uh{\Hat U}

\def\Wh{{\<\>\Hat{\<W\<}\<\>}}
\def\Who{\Wh{\vp W}^\cirs}
\def\Omh{\Hat\Om}
\def\Psh{\Hat\Psi}
\def\Gmh{\Hat\Gm}
\def\Yh{\Hat Y}

\def\gaq{\<\>\tilde{\<\gm\<\>}\<}
\def\GGq{\Tilde\GG}
\def\thq{\tilde\tht}
\def\Thq{\Tilde\Thb}
\def\Pst{\Tilde\Psi}
\def\Pshb{\>\overline{\<\<\smash{\Psh}\vp\Psi\<\<}\>}

\def\zb{\bs z}

\def\ddk_#1{q_{#1}\<\>\frac\der{\der\<\>q_{#1}}}

\def\Hqtl{\Hc^{\<\>\qti}_T(\tfl)}
\def\Wqt{W^{\qti}}
\def\hgrtv{h_{\<\>\mathrm{GRTV}}}

\def\bul{\mathbin{\raise.2ex\hbox{$\sssize\bullet$}}}
\def\intt{\mathchoice
{\mathop{\raise.2ex\rlap{$\,\,\ssize\backslash$}{\intop}}\nolimits}
{\mathop{\raise.3ex\rlap{$\,\sssize\backslash$}{\intop}}\nolimits}
{\mathop{\raise.1ex\rlap{$\sssize\>\backslash$}{\intop}}\nolimits}
{\mathop{\rlap{$\sssize\<\>\backslash$}{\intop}}\nolimits}}

\let\kp\kappa %% p

\def\GZ/{Gelfand-Zetlin}
\def\KZ/{{\slshape KZ\/}}
\def\qKZ/{{\slshape qKZ\/}}
\def\qKZB/{{\slshape qKZB\/}}
\def\XXX/{{\slshape XXX\/}}
\def\XXZ/{{\slshape XXZ\/}}

\def\new{{\mathrm{new}}}

\def\FF{{\bs F}}
\def\rr{{\bs r}}
\def\ww{{\bs w}}
\def\zz{{\bs z}}
\def\pp{{\bs p}}
\def\qq{{\bs q}}
\def\TT{{\bs t}}
\def\TTT{\tilde\TT}
\def\TTc{\check\TT}
\def\glN{{\frak{gl}_N}}
\def\Sym{\on{Sym}}

\def\GG{{\bs\Ga}}
\def\XX{{\mc X}}

\def\St{{\on{Stab}}}
\def\xx{{\bs x}}
\def\yy{{\bs y}}

\def\Cnn{(\C^N)^{\otimes\<\>n}}
\def\Cnnl{\Cnn_{\<\>\bla}}

\def\UU{{\check U}}
\def\WW{{\check W}}
\def\qqt{{\bs{\tilde q}}}
\def\naqla{\nabla^{\>\vp|\mathrm{\sssize quant}}}

\begin{document}

\hrule width0pt
\vsk->

\title[Quantum differential equations, Pieri rules, and Gamma-theorem]
{$\qq\<\>$-Hypergeometric solutions of quantum differential\\[3pt]
equations, quantum Pieri rules, and Gamma theorem}

\author[Vitaly Tarasov and Alexander Varchenko]
{Vitaly Tarasov$\>^\circ$ and Alexander Varchenko$\>^\star$}

\maketitle

\begin{center}
{\it $^{\star}\<$Department of Mathematics, University
of North Carolina at Chapel Hill\\ Chapel Hill, NC 27599-3250, USA\/}

\vsk.5>
{\it $^{\star}\<$Faculty of Mathematics and Mechanics, Lomonosov Moscow State
University\\ Leninskiye Gory 1, 119991 Moscow GSP-1, Russia\/}

\vsk.5>
{\it $\kern-.4em^\circ\<$Department of Mathematical Sciences,
Indiana University\,--\>Purdue University Indianapolis\kern-.4em\\
402 North Blackford St, Indianapolis, IN 46202-3216, USA\/}

\vsk.5>
{\it $^\circ\<$St.\,Petersburg Branch of Steklov Mathematical Institute\\
Fontanka 27, St.\,Petersburg, 191023, Russia\/}
\end{center}

{\let\thefootnote\relax
\footnotetext{\vsk-.8>\noindent
$^\circ\<${\sl E\>-mail}:\enspace vt@math.iupui.edu\>, vt@pdmi.ras.ru\>
\\
$^\star\<${\sl E\>-mail}:\enspace anv@email.unc.edu\>,
supported in part by NSF grants DMS-1362924, DMS-1665239}}

\vsk>
{\leftskip3pc \rightskip\leftskip \parindent0pt \Small
{\it Key words\/}: Flag varieties, quantum differential equation,
dynamical connection, $q$-hyper\-geometric solutions
\vsk.6>
{\it 2010 Mathematics Subject Classification\/}: 82B23, 17B80, 14N15, 14N35
\par}

\begin{abstract}
We describe \,$q$-hypergeometric solutions of the equivariant quantum
differential equations and associated \qKZ/ difference equations for
the cotangent bundle $\tfl$ of a partial flag variety \,$\Fla$\,.
These \,$q$-hypergeometric solutions manifest a Landau-Ginzburg mirror symmetry
for the cotangent bundle. We formulate and prove Pieri rules for quantum
equivariant cohomology of the cotangent bundle. Our Gamma theorem for \,$\tfl$
\,says that the leading term of the asymptotics of the \,$q$-hypergeometric
solutions can be written as the equivariant Gamma class of the tangent bundle
of $\tfl$ multiplied by the exponentials of the equivariant first Chern classes
of the associated vector bundles. That statement is analogous to the statement
of the gamma conjecture by B.\,Dubrovin and by S.\,Galkin, V.\,Golyshev,
and H.\,Iritani, see also the Gamma theorem for \,$\F_\bla$
\,in Appendix~\ref{app A}.
\end{abstract}

\vsk.8>
\rightline{\it In memory of Victor Lomonosov (1946\,--\,2018)}

{\small\tableofcontents\par}

\setcounter{footnote}{0}
\renewcommand{\thefootnote}{\arabic{footnote}}

\section{Introduction}

In \cite{MO}, D.\,Maulik and A.\,Okounkov develop a general theory connecting
quantum groups and equivariant quantum cohomology of Nakajima quiver varieties,
see \cite{N1, N2}. In particular, in \cite{MO} the operators of quantum
multiplication by divisors are described. As it is well-known, these operators
determine the equivariant quantum differential equations of a quiver variety.
In this paper we apply this description to the cotangent bundles \,$\tfl$
\,of the \,$\glnn$ \,$N$-step partial flag varieties and construct
\,$q$-hypergeometric solutions of the associated equivariant quantum
differential equations and \qKZ/ difference equations. The \,$q$-hypergeometric
solutions are constructed in the form of Jackson integrals.

Studying solutions of the equivariant quantum differential equations may lead
to better understanding Gromov-Witten invariants of the cotangent bundle,
cf.~Givental's study of the $J$-function in \cite{G1,G2,G3}.

The presentation of solutions of the equivariant quantum differential equations
as \,$q$-hypergeometric integrals manifests a version of the Landau-Ginzburg
mirror symmetry for the cotangent bundle.

In \cite{MO} the equivariant quantum differential equations come together
with a compatible system of difference equations called the \qKZ/ equations.
In \cite{GRTV, RTV} the equivariant quantum differential equations and \qKZ/
difference equations were identified with the dynamical differential equations
and \qKZ/ difference equations with values in the tensor product
$\Cnn$ of vector representations of $\frak{gl}_N$.
The $q$-hypergeometric solutions of the $\Cnn$-valued \qKZ/
difference equations were constructed long time ago in \cite{TV1}, see also
\cite{TV2}-\cite{TV4}. It was expected that those \,$q$-hypergeometric
solutions are also solutions of the compatible dynamical differential
equations. That fact is proved in this paper and is the first main result of
the paper. The proof is based on some new rather nontrivial identities for the
integrand of the Jackson integral. The integrand is the product of the scalar
master function and a vector-valued function, whose coordinates are called
weight functions. In \cite{RTV} it was shown that the weight functions are
nothing else but the stable envelopes of \cite{MO} for the cotangent bundle of
the partial flag varieties. Our new identities can be interpreted as new
identities for stable envelopes. We interpret these new identities as Pieri
rules in quantum equivariant cohomology of the cotangent bundle of the partial
flag variety. That is our second main result.

Our Gamma theorem for $\tfl$ (Theorem \ref{thm gamma}) says that the leading
term of the asymptotics of the \,$q$-hypergeometric solutions for $\tfl$ is
the product of the equivariant gamma class of the tangent bundle of $\tfl$
and the exponentials of the equivariant first Chern classes of the associated
vector bundles. That statement is analogous to the statement of the gamma
conjecture by B.\,Dubrovin and by S.\,Galkin, V.\,Golyshev, and H.\,Iritani,
see Appendix~\ref{app A}. See also the Gamma theorem for $\F_\bla$ (Theorem
\ref{thm gamma fl}).

\vsk.2>
The paper is organized as follows. In Section \ref{DQKZ} we introduce
the $\Cnn$-valued dynamical and \qKZ/ equations. In Section
\ref{Wfmf} we define the weight functions and list their basic properties.
In Section \ref{sMFDiDi} we introduce the master function and describe the
discrete differentials --- the quantities with zero Jackson integrals.
We also formulate there two key identities for the weight functions ---
Theorems \ref{thm1} and \ref{2thm}. We prove Theorem \ref{thm1} in
Section \ref{prthm1} and Theorem \ref{2thm} in Section \ref{prthm2}.
In Section \ref{Corr}, we summarize Theorems \ref{thm1} and \ref{2thm}
as a statement about the integrand of the main Jackson integral.
In Section \ref{sirfs} we construct integral representations for solutions
of the $\Cnn$-valued dynamical equations. In Section \ref{sQde}
we introduce the equivariant quantum differential equations and explain
how their \,$q$-hypergeometric solutions are obtained from solutions of
the $\Cnn$-valued dynamical equations. In Section \ref{QuPr} we formulate and
prove Pieri rules. In Section~\ref{seckth} we show that the space of solutions
of the quantum differential equation can be identified with the vector space of
the equivariant K-theory algebra. We also discuss two limiting cases of
the quantum differential equation. In Appendix~\ref{Sch} we discuss the basic
properties of Schubert polynomials, and in Appendix~\ref{app A} we formulate
our Gamma theorems.

\vsk.2>
The authors thank G.\,Cotti, V.\,Golyshev, and R.\,Rimanyi for useful
discussions. The second author thanks the Hausdorff Institute for Mathematics
in Bonn for hospitality in March 2018, when the Gamma theorem was discovered.
The second author also thanks the Max Planck Institute for Mathematics in Bonn
for hospitality in May--June 2018.

\section{Dynamical and \qKZ/ equations}
\label{DQKZ}

\subsection{Notations}
\label{Nts}
Fix $N, n\in \Z_{>0}$ and $h,\ka\in\Cxs$.
Let \,$\bla\in\Z^N_{\geq 0}$, \,$|\bla|=\la_1+\ldots+\la_N =n$.
Let $I=(I_1\lc I_N)$ be a partition of $\{1\lc n\}$ into disjoint subsets
$I_1\lc I_N$. Denote $\Il$ the set of all partitions $I$ with
$|I_j|=\la_j$, \;$j=1\lc N$.

Consider $\C^N$ with basis
$v_i=(0\lc 0,1_i,0\lc 0)$, $i=1\lc N$, and the tensor product
$\Cnn$ with basis
\vvn-.5>
\be
v_I\>=\,v_{i_1}\otimes\dots\otimes v_{i_n},
\ee
where the index $I$ is a partition $(I_1\lc I_N)$ of $\{1\lc n\}$ into disjoint subsets
$I_1\lc I_N$ and $i_j =m$ if $j\in I_m$.

The space $\Cnn$ is a module over the Lie algebra $\glN$ with basis $e_{\ij}$, $i,j=1\lc N$.
The $\glN$-module $\Cnn$ has weight decomposition $\Cnn=\sum_{|\bla|=n} \Cnnl$,
where $\Cnnl$ is the subspace with basis $(v_I)_{I\in \Il}$.

\subsection{Dynamical differential equations}
\label{Dde}
Define the linear operators \,$X_1\lc X_n$ \,acting on \,$\Cnn$-valued
functions of \,$\zz=(z_1\lc z_n)$\,, \,$\qq=(q_1\lc q_N)$ \,and called
the dynamical Hamiltonians:
\vvn-.5>
\begin{align}
\label{Xi}
X_i(\zz;h;\qq)\,=\,\sum_{a=1}^n\,z_ae^{(a)}_{\ii}-\>
h\>\Bigl(\>\frac{\et_{i,i}(1-\et_{\ii})}2\,&{}+\!
\sum_{1\leq a<b\leq n}\,\sum_{k=1}^N\,e_{\ik}^{(a)}\>e_{\ki}^{(b)}\>+{}
\\[3pt]
&{}+\>\sum_{\satop{j=1}{j\ne i}}^N\,\frac{q_j}{q_i\<-q_j}\,
(\et_{\ij}\>\et_{\ji}\<-\et_{\ii})\<\Bigr)\>,\kern-1.4em
\notag
\\[-22pt]
\notag
\end{align}
where \,$\et_{s,\<\>t} = \sum_{a=1}^n e_{s,\<\>t}^{(a)}$ and a superscript
means that the corresponding operator acts on the corresponding tensor factor.
The differential operators
\beq
\label{nady}
\nabla_{\qq,\ka,i}\,=\,\ka\>q_i\frac{\der}{\der q_i} - X_i(\zz;h;\qq)\,,
\qquad i=1\lc N\>,
\eeq
preserve the weight decomposition of $\Cnn$ and
pairwise commute, see \cite{TV2}, also \cite[Section 3.4]{GRTV}, \cite[Section 7.1]{RTV}, \cite{MTV1}.
The operators $\nabla_{\qq,\ka,i}$ define the $\Cnn$-valued {\it dynamical connection}.
The system of differential equations
\beq
\label{DEQ}
\ka\>q_i\>\frac{\der f}{\der q_i}\,=\,X_i(\zz;h;\qq)\>f\,, \qquad i=1\lc N\>,
\eeq
on a \,$\Cnn$-valued function $f(\zz;h;\qq)$ is called the
{\it dynamical equations}.

\subsection{Difference \qKZ/ equations}
\label{sec qkz}

Define the $R$-matrices acting on $\Cnn$,
\be
R^{(\ij)}(u)\,=\,\frac{u-h\<\>P^{(\ij)}}{u-h}\;,\qquad
i,j=1\lc n\,,\quad i\ne j\,. \kern-3em
\vv.3>
\ee
Define the \qKZ/ operators \,$K_1\lc K_n$ acting on $\Cnn$:
\vvn.4>
\begin{align*}
K_i(\zz;h;\qq;\kp)\>&{}=\,
R^{(\ii-1)}(z_i\<-z_{i-1}+\ka)\,\dots\,R^{(i,1)}(z_i\<-z_1+\ka)\,\times{}
\\[2pt]
& {}\>\times\<\;q_1^{e_{1,1}^{(i)}}\!\dots\,q_N^{e_{N,N}^{(i)}}\,
R^{(i,n)}(z_i\<-z_n\<\>)\,\dots\,R^{(\ii+1)}(z_i\<-z_{i+1}
\<\>)\,.
\end{align*}
The \qKZ/ operators preserve the weight decomposition of $\Cnn$
and form a discrete flat connection,
\be
K_i(z_1\lc z_j+\ka\lc z_n;\qq;\ka)\,K_j(\zz;h;\qq;\kp)\,=\,
K_j(z_1\lc z_i+\ka\lc z_n;\qq;\ka)\,K_i(\zz;h;\qq;\kp)
\ee
for all $i,j$, see \cite{FR}.
The system of difference equations with step $\ka$,
\beq
\label{K_i}
f(z_1\lc z_i+\ka\lc z_n;\qq)\,=\,K_i(\zz;h;\qq;\kp)\,f(z_1\lc z_n;\qq),\qquad
i=1\lc N\>,
\vv.3>
\eeq
on a \,$ (\C^N)^{\otimes n}\<$-valued
function $f(\zz,\qq)$ is called the \qKZ/ equations.

\begin{thm} [{\cite{TV2}}]
\label{thm qkz}
The systems of dynamical and \qKZ/ equations are compatible.
\qed
\end{thm}

\section{Weight functions}
\label{Wfmf}

\subsection{Weight functions $\WW_I$}
\label{dwf}
For $I\in \Il$, we define the weight functions $\WW_I(\TT;\zz)$,
cf.~\cite{TV1,TV4,RTV}. The functions \,$\WW_I(\TT;\zz)$ \,here coincide with
the functions \,$W_I(\TT;\zz;h)$ defined \,in \cite[Section~3.1]{RTV}.

\vsk.2>
Recall \,$\bla=(\la_1\lc\la_N)$.
Denote \,$\la^{(i)}\<=\la_1\lsym+\la_i$, $i=1\lc N-1$,
$\la^{(N)}=n$, and
\,$\la^{\{1\}}\<=\sum_{i=1}^{N-1}\la^{(i)}=$
$ \sum_{i=1}^{N-1}(N\<\<-i)\>\la_i$\>.
Recall $I=(I_1\lc I_N)$. Set
\;$\bigcup_{\>k=1}^{\,j}I_k=\>\{\>i^{(j)}_1\!\lsym<i^{(j)}_{\la^{(j)}}\}$\>.
Consider the variables \,$h$ \,and
\,$t^{(j)}_a$, \,$j=1\lc N-1$, \,$a=1\lc\la^{(j)}$.
\,Set \,$t^{(N)}_a\!=z_a$, \,$a=1\lc n$\>.
Denote $\TT^{(j)}=(t^{(j)}_k)_{k\leq\la^{(j)}}$ and
\,$\TT=(\TT^{(1)}\lc\TT^{(N-1)})$.

\vsk.2>
The weight functions are
\vvn.4>
\beq
\label{hWI-}
\WW_I(\TT;\zb)\,=\,(-h)^{\>\la^{\{1\}}}\,
\Sym_{\>t^{(1)}_1\!\lc\,t^{(1)}_{\la^{(1)}}}\,\ldots\;
\Sym_{\>t^{(N-1)}_1\!\lc\,t^{(N-1)}_{\la^{(N-1)}}}\UU_I(\TT;\zb)\,,
\vv.3>
\eeq
\be
\UU_I(\TT;\zb)\,=\,\prod_{j=1}^{N-1}\,\prod_{a=1}^{\la^{(j)}}\,\biggl(
\prod_{\satop{c=1}{i^{(j+1)}_c\<<\>i^{(j)}_a}}^{\la^{(j+1)}}
\!\!(t^{(j)}_a\<\<-t^{(j+1)}_c-h)
\prod_{\satop{d=1}{i^{(j+1)}_d>\>i^{(j)}_a}}^{\la^{(j+1)}}
\!\!(t^{(j)}_a\<\<-t^{(j+1)}_d)\,\prod_{b=a+1}^{\la^{(j)}}
\frac{t^{(j)}_a\<\<-t^{(j)}_b\<\<-h}{t^{(j)}_a\<\<-t^{(j)}_b}\,\biggr)\,.
\ee
In these formulas for a function $f(t_1\lc t_k)$ of some variables, we denote
\be
\Sym_{t_1\lc t_k}f(t_1\lc t_k)\,=\,
\sum_{\sigma\in S_k}f(t_{\sigma_1}\lc t_{\sigma_k}).
\ee

\begin{example}
Let $N=2$, $n=2$, $\bla=(1,1)$, $I=(\{1\},\{2\})$,
$J=(\{2\}, \{1\})$. Then
\bea
\WW_I(\TT;\zz)= -h\, (t^{(1)}_1\<\<-z_2),
\qquad
\WW_J(\TT;\zz)= -h\, (t^{(1)}_1\<\<-z_1-h).
\eea
\end{example}

\begin{example}
Let $N=2$, $n=3$, $\bla=(1,2)$, $I=(\{2\},\{1,3\})$. Then
\bea
\WW_I(\TT;\zz)= -h\, (t^{(1)}_1\<\<-z_1-h)\>(t^{(1)}_1\<\<-z_3).
\eea
\end{example}

\begin{example}
Let $N=2$, $n=3$, $\bla=(2,1)$, $I=(\{1,3\},\{2\})$. Then
\begin{align*}
\WW_I(\TT;\zz)\,=\,(-h)^2\,\Bigl( & (t^{(1)}_1\<\<-z_2)\>(t^{(1)}_1\<\<-z_3)\>
(t^{(1)}_2\<\<-z_1-h)(t^{(1)}_2\<\<-z_2-h)\,
\frac{t^{(1)}_1\<\<-t^{(1)}_2\<\<-h}{t^{(1)}_1\<\<-t^{(1)}_2}\>+{}
\\[3pt]
{}+{}\> & (t^{(1)}_2\<\<-z_2)\>(t^{(1)}_2\<\<-z_3)\>
(t^{(1)}_1\<\<-z_1-h)(t^{(1)}_1\<\<-z_2-h)\,
\frac{t^{(1)}_2\<\<-t^{(1)}_1\<\<-h}{t^{(1)}_2\<\<-t^{(1)}_1}\>\Bigr)\>.
\end{align*}
\end{example}

For a subset $A=\{a_1\lc a_j\}\subset\{1\lc n\}$, denote
$\zz_A=(z_{a_1}\lc z_{a_j})$. For $I\in\Il$, denote
$\zz_I=(\zz_{I_1}\lc\zz_{I_N})$. For
$f(\TT^{(1)}\lc\TT^{(N)})\in\C[\TT^{(1)}\lc\TT^{(N)}]^{S_{\la^{(1)}}\times\ldots\times S_{\la^{(N)}}}$,
we define $f(\zz_I)$ by substituting \,$\TT^{(j)}=(\zz_{I_1}\lc\zz_{I_j})$\,,
$\;j=1\lc N$.

\subsection{Weight functions $\WW_{\si,I}$}
\label{sec 3.2}

For $\si\in S_n$ and $I\in\Il$, we define
\beq
\label{WWsi}
\WW_{\si,I}(\TT;\zz) = \WW_{\si^{-1}(I)}(\TT;z_{\si(1)}\lc z_{\si(n)}),
\qquad
\UU_{\si,I}(\TT;\zz) = \UU_{\si^{-1}(I)}(\TT;z_{\si(1)}\lc z_{\si(n)}),
\eeq
where $\si^{-1}(I)=(\si^{-1}(I_1)\lc\si^{-1}(I_N))$.

\begin{example}
Let $N=2$, $n=2$, $\bla=(1,1)$, $I=(\{1\}, \{2\})$, $J=(\{2\},\{1\})$. Then
\vvn.2>
\begin{alignat*}2
\WW_{\id\<\>,\>I}(\TT;\zz)\,&{}=\,-\>h\>(t^{(1)}_1\<\<-z_2)\,, &
\WW_{\id\<\>,\>J}(\TT;\zz)\,&{}=\,-\>h\>(t^{(1)}_1\<\<-z_1\<-h)\,,
\\[2pt]
\WW_{s,\>I}(\TT;\zz)\, &{}=\,-\>h\>(t^{(1)}_1\<\<-z_2\<-h)\,, \qquad &
\WW_{s,\>J}(\TT;\zz)\, &{}=\,-\>h\>(t^{(1)}_1\<\<-z_1)\,,
\end{alignat*}
where \,$s$ \,is the transposition.
\end{example}

\subsection{Three-term relation}

\begin{lem}[{\cite[Lemma~3.6]{RTV}}]
\label{lem W si W}
For any $\si\in S_n$, \,$I\in \Il$, $\;i=1\lc n-1$, \,we have
\beq
\label{WW3}
\WW_{\sigma s_{\ii+1},I} =
\frac{z_{\sigma(i)}-z_{\sigma(i+1)}}{z_{\sigma(i)}-z_{\sigma(i+1)}+h}\WW_{\sigma,I} +
\frac h{z_{\sigma(i)}-z_{\sigma(i+1)}+h} \WW_{\sigma, s_{\sigma(i),\sigma(i+1)}(I)} ,
\eeq
where $s_{\ij}\in S_n$ is the transposition of \,$i$ and $j$.
\qed
\end{lem}

\subsection{Weight functions $W_I(\TT;\zz)$}

Let \,$\si_0\<\in S_n$ be the longest permutation,
\,$\si_0(i)=n+1-i$\,, \;$i=1\lc n$\,. For \,$I\in\Il$\,, denote
\beq
\label{WW}
W_I(\TT;\zz)\,=\,(-h)^{-\la^{\{1\}}}\>\WW_{\si_0, I}(\TT; \zz)\,,\qquad
U_I(\TT;\zz)\,=\,\UU_{\si_0, I}(\TT; \zz)\,.
\eeq
In other words, we have
\vvn.2>
\beq
\label{hWI--}
W_I(\TT;\zb)\,=\,
\Sym_{\>t^{(1)}_1\!\lc\,t^{(1)}_{\la^{(1)}}}\,\ldots\;
\Sym_{\>t^{(N-1)}_1\!\lc\,t^{(N-1)}_{\la^{(N-1)}}}
U_I(\TT;\zb)\,,
\vv-.4>
\eeq
\begin{align}
\label{UI}
& U_I(\TT;\zb)\,={}
\\[4pt]
\notag
&{}\!=\,\prod_{j=1}^{N-1}\,\prod_{a=1}^{\la^{(j)}}\,\biggl(
\prod_{\satop{c=1}{i^{(j+1)}_c\<<\>i^{(j)}_a}}^{\la^{(j+1)}}
\!\!(t^{(j)}_a\<\<-t^{(j+1)}_c)
\prod_{\satop{d=1}{i^{(j+1)}_d>\>i^{(j)}_a}}^{\la^{(j+1)}}
\!\!(t^{(j)}_a\<\<-t^{(j+1)}_d-h )\,\prod_{b=a+1}^{\la^{(j)}}
\frac{t^{(j)}_b\<\<-t^{(j)}_a\<\<-h}{t^{(j)}_b\<\<-t^{(j)}_a}\,\biggr)\,.
\kern-1em
\end{align}

\begin{example}
Let $N=2$, $n=2$, $\bla=(1,1)$, $I=(\{1\}, \{2\})$, $J=(\{2\},\{1\})$. Then
\vvn.1>
\be
W_I(\TT;\zz)\,=\,t^{(1)}_1\<\<-z_2-h\,, \qquad
W_J(\TT;\zz)\,=\,t^{(1)}_1\<\<-z_1\,.
\vv.5>
\ee
\end{example}

\subsection{Modification of the three-term relation}

For a function $f(z_1\lc z_n)$ and $i=1\lc n-1$, define the operator
\,$S_{\ii+1}$ by the formula
\vvn.1>
\beq
\label{Skk+1}
S_{\ii+1}\<\>f(z_1\lc z_n)\,=\,\frac{z_i\<-z_{i+1}\<-h}{z_i\<-z_{i+1}}\,
f(z_1\lc z_{i+1},z_i\lc z_n)\>+\>\frac h{z_i\<-z_{i+1}}\,f(z_1\lc z_n)\,.
\kern-2em
\vv.2>
\eeq
Lemma \ref{lem W si W} can be reformulated as follows.

\goodbreak
\begin{lem}
\label{c3t}
For any $I\in \Il$, $i=1\lc n-1$, we have
\vvn.3>
\begin{align}
\label{m3t}
& W_{s_{\ii+1}(I)}(\TT;\zz)\,=\,(S_{\ii+1}\>W_I)(\TT;\zz)
\\[4pt]
\notag
\!&{}=\,\frac{z_i\<-z_{i+1}\<-h}{z_i\<-z_{i+1}}\,
W_I(\TT;z_1\lc z_{i+1},z_i\lc z_n)\>+\>\frac{h}{z_i\<-z_{i+1}}\,
W_I(\TT; z_1\lc z_n)\,.
\\[-12pt]
\noalign{\qed}
\notag
\\[-33pt]
\notag
\end{align}
\end{lem}

\subsection{Shuffle properties}
\label{sShp}

Let $n,n_1,n_2\in \Z_{>0}$, $n=n_1+n_2$. Let $\bla^1, \bla^2, \bla \in\Z_{\geq 0}^N$,
$|\bla^1|=n_1$,
$|\bla^2|=n_2$, $\bla=\bla^1+\bla^2$. Let $I^1=(I^1_1\lc I^1_N)$
be a decomposition of the set $\{1\lc n_1\}$ into subsets such that
$|I^1_j|=\la^1_j$. Let $I^2=(I^2_1\lc I^2_N)$
be a decomposition of the set $\{n_1+1\lc n\}$ into subsets such that
$|I_j^2|=\la^2_j$.
Define the decomposition $I=(I_1\lc I_N)$ of the set
$\{1\lc n\}$ by the rule: $I_j=I^1_j\cup I^2_j$.

Consider the weight function $W_I$
of variables $(\TT^{(1)}\!\lc\TT^{(N)})$,
where $\TT^{(j)}=(t^{(j)}_1\!\lc t^{(j)}_{\la^{(j)}})$,
$\la^{(j)}=\la_1+\ldots+\la_j$, $j=1\lc N-1$,
and $\TT^{(N)}=(z_1\lc z_n)$.

Consider the weight function $W_{I^1}$ of variables
\,$(\TTT^{(1)}\!\lc\TTT^{(N)})$,
where \,$\TTT^{(j)}=(t^{(j)}_1\!\lc t^{(j)}_{(\bla^1)^{(1)}})$,
$(\bla^1)^{(j)}=\la^1_1+\ldots+\la_j^1$, $j=1\lc N-1$,
and $\TTT^{(N)}=(z_1\lc z_{n_1})$.
Consider the weight function $W_{I^2}$
of variables $\TTc=(\TTc^{(1)}\!\lc\TTc^{(N)})$,
where $\TTc^{(j)}=(t^{(j)}_{(\bla^1)^{(j)}+1}\lc t^{(j)}_{\la^{(j)}})$
for $j=1\lc N-1$, and $\TTc^{(N)}=(z_{n_1+1}\lc z_{n})$.
Denote $(\bla^2)^{(j)}=\la^2_1+\ldots+\la_j^2$, $j=1\lc N-1$.

\vsk.2>
Define the connection coefficient
\vvn.4>
\begin{align}
\label{ccoe}
C_{\bla^1\!,\<\>\bla^2}(\TT;\zz)\,&{}=\,\prod_{j=1}^{N-1}\Bigl[
\Bigl(\prod_{a=1}^{(\bla^1)^{(j)}}\prod_{b=(\bla^1)^{(j)}+1}^{\la^{(j)}}
\frac{t^{(j)}_b-t^{(j)}_a-h}{t^{(j)}_b-t^{(j)}_a}\>\Bigr)\times{}
\\[2pt]
\notag
&{}\times\,
\Bigl(\prod_{a=1}^{(\bla^1)^{(j)}}\prod_{c=(\bla^1)^{(j+1)}+1}^{\la^{(j+1)}}(t^{(j)}_a-t^{(j)}_c-h)\Bigr)
\Bigl(\prod_{a=(\bla^1)^{(j)}+1}^{\la^{(j)}}\prod_{c=1}^{(\bla^1)^{(j+1)}}(t^{(j)}_a-t^{(j)}_c)\Bigr)\Bigr]\,.
\kern-1.2em
\end{align}

\begin{lem}
\label{lem shuf}
We have
\vvn-.1>
\beq
W_I(\TT;\zz)\,=\,\frac{\Sym_{\>t^{(1)}_1\!\lc\,t^{(1)}_{\la^{(1)}}}\ldots\;
\Sym_{\>t^{(N-1)}_1\!\lc\,t^{(N-1)}_{\la^{(N-1)}}}\Wt_{I^1,I^2}(\TT;\zz)}
{\prod_{j=1}^{N-1} ((\bla^1)^{(j)})!((\bla^2)^{(j)})!}\;,
\kern-1em
\eeq
where
\vvn-.6>
\begin{align*}
\Wt_{I^1\!,\<\>I^2}(\TT;\zz)\,=\,C_{\bla^1\!,\<\>\bla^2}(\TT;\zz)\;
& W_{I^1}(\TTT^{(1)}\lc\TTT^{(N-1)};z_1\lc z_{n_1})
\\[1pt]
{}\times{} & W_{I^2}(\TTc^{(1)}\lc\TTc^{(N-1)}; z_{n_1+1}\lc z_n)\,.
\kern-1em
\\[-14pt]
\noalign{\qed}
\\[-34pt]
\end{align*}
\end{lem}

\subsection{Factorization}

Consider the $\frak{gl}_N$ weight function
$W_{\{1\lc n\},\emptyset\lc\emptyset}$ of variables $(\TT^{(1)}\lc\TT^{(N)})$,
where \,$\TT^{(j)}=(t^{(j)}_1\lc t^{(j)}_{n})$ \,for \,$j=1\lc N-1$, and
$\TT^{(N)}=(z_1\lc z_n)$.

\begin{lem}
\label{lem fa}
We have
\vvn-.4>
\beq
\label{factor}
W_{\{1\lc n\},\emptyset\lc\emptyset}(\TT^{(1)}\lc\TT^{(N-1)};\TT^{(N)})
\,=\>\prod_{j=1}^{N-1}\>
W^{\frak{gl}_2}_{\{1\lc n\},\emptyset}(\TT^{(j)};\TT^{(j+1)})\,,
\vv-.5>
\eeq
where \,$W^{\frak{gl}_2}_{\{1\lc n\},\emptyset}(\TT^{(j)};\TT^{(j+1)})$ is
\vvn.1>
the $\;\frak{gl}_2$ weight function assigned to the partition of the set
$\{1\lc n\}$ into two subsets \,$\{1\lc n\}$ and \;$\emptyset$.
\end{lem}
\begin{proof}
The function \,$W^{\frak{gl}_2}_{\{1\lc n\},\emptyset}(\TT^{(j)};\TT^{(j+1)})$
is symmetric in variables \,$(t^{(j+1)}_1\!\lc t^{(j+1)}_n)$ \,due to
the \,$\frak{gl}_2$ three-term relations of Lemma \ref{lem W si W}.
That symmetry and formula \Ref{hWI-} imply formula \Ref{factor}.
\end{proof}

\goodbreak
\subsection{Useful identities}

\begin{thm}
\label{thm si}
Given $k\in\Z_{\geq 0}$, consider variables
\,$t^{(0)}_1, t^{(k+1)}_1$ and \,$t^{(i)}_1, t^{(i)}_2$ for $i=1\lc k$.
Set
\be
F\,=\,(t^{(0)}_1\<\<-t^{(1)}_1)\>(t^{(k)}_1\<\<-t^{(k+1)}_1\<\<-h) -
(t^{(0)}_1\<\<-t^{(1)}_2\<\<-h)\>(t^{(k)}_2\<\<-t^{(k+1)}_1)\,,
\vvn.6>
\ee
\be
G\,=\,(t^{(0)}_1\<\<-t^{(1)}_2\<\<-h)\>(t^{(k)}_1\<\<-t^{(k+1)}_1\<\<-h) -
(t^{(0)}_1\<\<-t^{(1)}_1)\>(t^{(k)}_2\<\<-t^{(k+1)}_1)\,,
\vv.1>
\ee
and
\vvn-.2>
\be
H\,=\,\prod_{i=1}^{k-1}\bigl(
(t^{(i)}_1\<\<-t^{(i+1)}_2\<\<-h)\>(t^{(i)}_2\<\<-t^{(i+1)}_1)\bigr)\,
\prod_{i=1}^{k}\,\frac{t^{(i)}_2\<\<-t^{(i)}_1\<\<-h}{t^{(i)}_2\<\<-t^{(i)}_1}
\;.
\ee
Then
\beq
\label{si}
\Sym_{\>t^{(1)}_1\!\<,\>t^{(1)}_2}\,\ldots\;
\Sym_{\>t^{(k)}_1\!\<,\>t^{(k)}_2}(FH)\,=\,0
\eeq
and
\beq
\label{sisi}
\Sym_{\>t^{(1)}_1\!\<,\>t^{(1)}_2}\,\ldots\;
\Sym_{\>t^{(k)}_1\!\<,\>t^{(k)}_2}(G\<\>H)\,=\,0\,.
\eeq
\end{thm}
\begin{proof}
Formulae \Ref{si}, \Ref{sisi} are equivalent to
\vvn.3>
\be
\Sym_{\>t^{(1)}_1\!\<,\>t^{(1)}_2}\,\ldots\;
\Sym_{\>t^{(k)}_1\!\<,\>t^{(k)}_2}\bigl((F\pm G)\>H\bigr)\,=\,0\,.
\vv-.2>
\ee
Observe that
\vv-.1>
\be
F+G\,=\,(2\>t^{(0)}_1\<\<-t^{(1)}_1\<\<-t^{(1)}_2\<\<-h)\>
(t^{(k)}_1\<\<-t^{(k)}_2\<\<-h)\,,
\vvn.5>
\ee
\be
F-G\,=\,(t^{(1)}_1\<\<-t^{(1)}_2\<\<-h)\>
(t^{(k)}_1\<\<+t^{(k)}_2\<\<-2\>t^{(k+1)}_1\<\<-h)\,,
\vv-.4>
\ee
and
\vvn.1>
\be
\Sym_{\>s_1,s_2}\Bigl((s_1\<-u_2\<-h)\>(s_2\<-u_1)\,
\frac{s_2\<-s_1\<-h}{s_2-s_1}\>\Bigr)\,=\,
\Sym_{\>u_1,u_2}\Bigl((s_1\<-u_2\<-h)\>(s_2\<-u_1)\,
\frac{u_2\<-u_1\<-h}{s_2-s_1}\>\Bigr)\>.
\vv.1>
\ee
Hence the function
\vvn.2>
\be
W^{\frak{gl}_2}_{\{1,2\},\emptyset}(s_1,s_2, u_1,u_2)\,=\,
\Sym_{\>s_1,s_2}\Bigl((s_1\<-u_2\<-h)\>(s_2\<-u_1)\,
\frac{s_2\<-s_1\<-h}{s_2-s_1}\>\Bigr)
\vv.2>
\ee
is symmetric both in \,$s_1,s_2$ \,and \,$u_1,u_2$\,. Therefore,
\vvn.4>
\begin{align*}
& \Sym_{\>t^{(1)}_1\!\<,\>t^{(1)}_2}\,\ldots\;
\Sym_{\>t^{(k)}_1\!\<,\>t^{(k)}_2}\bigl((F+G)\>H\bigr)\,=\,
(2\>t^{(0)}_1\<\<-t^{(1)}_1\<\<-t^{(1)}_2\<\<-h)\>\times{}
\\[3pt]
&\;\,{}\times\>\Sym_{\>t^{(k)}_1\!\<,\>t^{(k)}_2}
\Bigl((t^{(k)}_1\<\<-t^{(k)}_2\<\<-h)\,
\frac{t^{(k)}_2\<\<-t^{(k)}_1\<\<-h}{t^{(k)}_2\<\<-t^{(k)}_1}\>\Bigr)\,
\prod_{i=1}^{k-1}\,W^{\frak{gl}_2}_{\{1,2\},\emptyset}
(t^{(i)}_1,t^{(i)}_2, t^{(i+1)}_1,t^{(i+1)}_2)\,=\,0\,.
\end{align*}
and
\begin{align*}
& \Sym_{\>t^{(1)}_1\!\<,\>t^{(1)}_2}\,\ldots\;
\Sym_{\>t^{(k)}_1\!\<,\>t^{(k)}_2}\bigl((F-G)\>H\bigr)\,=\,
(t^{(k)}_1\<\<+t^{(k)}_2\<\<-2\>t^{(k+1)}_1\<\<-h)
\>\times{}
\\[3pt]
&\;\,{}\times\>\Sym_{\>t^{(1)}_1\!\<,\>t^{(1)}_2}
\Bigl((t^{(1)}_1\<\<-t^{(1)}_2\<\<-h)\,
\frac{t^{(1)}_2\<\<-t^{(1)}_1\<\<-h}{t^{(1)}_2\<\<-t^{(1)}_1}\>\Bigr)\,
\prod_{i=1}^{k-1}\,W^{\frak{gl}_2}_{\{1,2\},\emptyset}
(t^{(i)}_1,t^{(i)}_2, t^{(i+1)}_1,t^{(i+1)}_2)\,=\,0\,.
\end{align*}
Theorem \ref{thm si} is proved.
\end{proof}

\section{Master function and discrete differentials}
\label{sMFDiDi}
\subsection{Master function}
\label{sMF}

Let \,$\pho(x)=\Ga(x/\ka)\,\Ga\bigl((h-x)/\ka\bigr)$\,.
Define the master function:
\vvn.3>
\begin{align}
\label{PHI}
\Phi_\bla(\TT;\zz;h;\qq)\,=\,
(e^{\<\>\pii\,(n\<\>-\<\la_N)}q_N)^{\>\sum_{a=1}^nz_a/\ka}\,
\prod_{i=1}^{N-1}\>\Bigl(\>e^{\<\>\pii\,(\la_{i+1}-\la_i)}\>
\frac{q_i}{q_{i+1}}\>\Bigr)^{\sum_{j=1}^{\la^{(i)}} t^{(i)}_j\!/\ka}\times{}&
\\[2pt]
\notag
{}\times\<{}\prod_{i=1}^{N-1}\,\prod_{a=1}^{\la^{(i)}}\,\biggl(\,
\prod_{\satop{b=1}{b\ne a}}^{\la^{(i)}}\,
\frac1{(t_a^{(i)}\<\<-t_b^{(i)}\<\<-h)\,\pho(t_a^{(i)}\<\<-t_b^{(i)})}\,
\prod_{c=1}^{\la^{(i+1)}}\<\pho(t_a^{(i)}\<\<-t_c^{(i+1)})\<\biggr)\,&.
\end{align}
It is a symmetric function of variables in each of the groups \;$\TT^{(i)}$,
\,$i=1\lc N-1$.

\subsection{Definition of discrete differentials}
\label{DefDi}
Consider the space $\mc S$ of functions of the form
$\Phi_\bla(\TT;\zz;h;\qq) f(\TT;\zz;h;\qq)$ where $f(\TT;\zz;h;\qq)$ is a rational function.
Consider the lattice $\ka\Z^{\la^{\{1\}}}$ whose coordinates are labeled by variables $t^{(i)}_j\in\TT$.
The shifts $t^{(i)}_j \mapsto t^{(i)}_j + \ka$ of any of the $\TT$-variables preserve the space $\mc S$ and extend to
an action of the lattice $\ka\Z^{\la^{\{1\}}}$ on $\mc S$. A {\it discrete
differential} is a finite sum of rational functions of the form
\vvn.5>
\beq
\label{41}
\frac{\Phi_\bla(\TT+w;\zz;h;\qq)}{\Phi_\bla(\TT;\zz;h;\qq)}\,f(\TT+\ww;\zz;h;\qq)
\>-f(\TT;\zz;h;\qq)\,,\kern-1.2em
\vv-.5>
\eeq
where \,$\ww\in \ka \Z^{\la^{\{1\}}}$.

\subsection{Special discrete differentials}
\label{sec sdd}
For integers ${1\leq\al<\bt\leq N}$, split
the variables \,$\TT=(\TT^{(1)}\lc\TT^{(N-1)})$\,,
$\;\TT^{(i)}=(t^{(i)}_1\lc t^{(i)}_{\la^{(i)}})$\,, into two groups
\,$\TT^{\{\albt\}}$ and \,$\TT_{\{\albt\}}$ \,as follows:
\vvn.2>
\be
\TT^{\{\albt\}} =\,(t^{(\al)}_{\la^{(\al)}}\>,\,t^{(\al+1)}_{\la^{(\al+1)}}\>
\lc\,t^{(\bt-1)}_{\la^{(\bt-1)}})
\vv.1>
\ee
and \,$\TT_{\{\albt\}}=(\TT^{(1)}_{\{\albt\}}\>\lc\TT^{(N-1)}_{\{\albt\}})$\,,
where
\,$\TT^{(i)}_{\{\albt\}}=(t^{(i)}_1\lc t^{(i)}_{\la^{(i)}-1})$
\,if $\;\al\leq i <\bt$ \,and \,$\TT^{(i)}_{\{\albt\}}=\TT^{(i)}$, otherwise.

\vsk.3>
For a rational function \,$g$ \,of $\TT_{\{\albt\}},\zz,\qq$\,, denote
\begin{align}
\label{dab}
d_{\TT^{\{\albt\}}}g\,:=\;
\frac{g(\TT_{\{\albt\}};\zz;h;\qq)}{q_\al-q_\bt}\,\Bigl(\>
& q_\bt\>(t_{\la^{(\al-1)}}^{(\al-1)}\<-t_{\la^{(\al)}}^{(\al)})\,
\prod_{i=\al}^{\bt-1}\,\prod_{a=1}^{\la^{(i-1)}-1}\!
(t_a^{(i-1)}-t_{\la^{(i)}}^{(i)})\,\times{}
\\[1pt]
\notag
{}\times\,
(t_{\la^{(\bt-1)}}^{(\bt-1)}\<-t_{\la^{(\bt)}}^{(\bt)}\<-h)\,
& \prod_{i=\al}^{\bt-1}\,\prod_{a=1}^{\la^{(i+1)}-1}
(t_{\la^{(i)}}^{(i)}-t_a^{(i+1)}-h)\,
\prod_{i=\al}^{\bt-1}\,\prod_{a=1}^{\la^{(i)}-1}\,
\frac {t_a^{(i)}\<\<-t_{\la^{(i)}}^{(i)}-h}{t_a^{(i)}\<\<-t_{\la^{(i)}}^{(i)}}
\,-{}\kern-2em
\\[3pt]
\notag
{}-\,{} & q_\al\>(t_{\la^{(\al-1)}}^{(\al-1)}\<-t_{\la^{(\al)}}^{(\al)}\<-h)\,
\prod_{i=\al}^{\bt-1}\,\prod_{a=1}^{\la^{(i-1)}-1}\!
(t_a^{(i-1)}\<\<-t_{\la^{(i)}}^{(i)}\<-h)\,\times{}
\\[2pt]
\notag
{}\times\,(t_{\la^{(\bt-1)}}^{(\bt-1)}\<-t_{\la^{(\bt)}}^{(\bt)}\bigr)\,
& \prod_{i=\al}^{\bt-1}\,\prod_{a=1}^{\la^{(i+1)}-1}\!
(t_{\la^{(i)}}^{(i)}-t_a^{(i+1)})\,
\prod_{i=\al}^{\bt-1}\,\prod_{a=1}^{\la^{(i)}-1}\,
\frac {t_{\la^{(i)}}^{(i)}-t_a^{(i)}-h}{t_{\la^{(i)}}^{(i)}-t_a^{(i)}}
\>\Bigr)\>.
\end{align}

\begin{lem}
\label{ldab}
The function $\;d_{\TT^{\{\albt\}}}g$ \,is a discrete differential.
\qed
\end{lem}
\begin{proof} Formula \Ref{dab} is an example of formula \Ref{41}, where
\vvn.3>
\begin{align*}
f(\TT;\zz;h;\qq)\,&{}=\,
g(\TT_{\{\albt\}};\zz;h;\qq)\,
(t_{ \la^{(\al-1)} }^{(\al-1)}\<-t_{\la^{(\al)}}^{(\al)})\,
\prod_{i=\al}^{\bt-1}\,\prod_{a=1}^{\la^{(i-1)}-1}\!
(t_a^{(i-1)}\<\<-t_{\la^{(i)}}^{(i)})\,\times{}
\\[2pt]
\notag
&{}\times\,(t_{\la^{(\bt-1)}}^{(\bt-1)}\<-t_{\la^{(\bt)}}^{(\bt)}\<-h)
\prod_{i=\al}^{\bt-1}\,\prod_{a=1}^{\la^{(i+1)}-1}\!
(t_{\la^{(i)}}^{(i)}\<-t_a^{(i+1)}\<-h)\,
\prod_{i=\al}^{\bt-1}\,\prod_{a=1}^{\la^{(i)}-1}\,
\frac{t_a^{(i)}\<\<-t_{\la^{(i)}}^{(i)}\<-h}
{t_a^{(i)}\<\<-t_{\la^{(i)}}^{(i)}}
\\[-14pt]
\end{align*}
and $\ww$ has coordinates \,$t^{(\al)}_{\la^{(\al)}}\>,
\,t^{(\al+1)}_{\la^{(\al+1)}}\>\lc\,t^{(\bt-1)}_{\la^{(\bt-1)}}$
\,equal to \,$\ka$ \,and other coordinates equal to zero.
\end{proof}

We rewrite $d_{\TT^{\{\albt\}}}g$ as
\be
d_{\TT^{\{\albt\}}}g\,=\,
\frac{q_\bt}{q_\al -q_\bt}\,\dti_{\TT^{\{\albt\}}}g
-\dch_{\TT^{\{\albt\}}}g,
\ee
where
\vvn-.4>
\begin{align}
\label{dabt}
\dti_{\TT^{\{\albt\}}}g\,:=\,g(\TT_{\{\albt\}};\zz;h;\qq)\,\Bigl(\>
& (t_{\la^{(\al-1)}}^{(\al-1)}\<-t_{\la^{(\al)}}^{(\al)})\,
\prod_{i=\al}^{\bt-1}\,\prod_{a=1}^{\la^{(i-1)}-1}\!
(t_a^{(i-1)}-t_{\la^{(i)}}^{(i)})\,\times{}
\\[1pt]
\notag
{}\times\,
(t_{\la^{(\bt-1)}}^{(\bt-1)}\<-t_{\la^{(\bt)}}^{(\bt)}\<-h)\,
& \prod_{i=\al}^{\bt-1}\,\prod_{a=1}^{\la^{(i+1)}-1}
(t_{\la^{(i)}}^{(i)}-t_a^{(i+1)}-h)\,
\prod_{i=\al}^{\bt-1}\,\prod_{a=1}^{\la^{(i)}-1}\,
\frac {t_a^{(i)}\<\<-t_{\la^{(i)}}^{(i)}-h}{t_a^{(i)}\<\<-t_{\la^{(i)}}^{(i)}}
\,-{}\kern-2em
\\[3pt]
\notag
{}-\,{} & (t_{\la^{(\al-1)}}^{(\al-1)}\<-t_{\la^{(\al)}}^{(\al)}\<-h)\,
\prod_{i=\al}^{\bt-1}\,\prod_{a=1}^{\la^{(i-1)}-1}\!
(t_a^{(i-1)}\<\<-t_{\la^{(i)}}^{(i)}\<-h)\,\times{}
\\[2pt]
\notag
{}\times\,(t_{\la^{(\bt-1)}}^{(\bt-1)}\<-t_{\la^{(\bt)}}^{(\bt)}\bigr)\,
& \prod_{i=\al}^{\bt-1}\,\prod_{a=1}^{\la^{(i+1)}-1}\!
(t_{\la^{(i)}}^{(i)}-t_a^{(i+1)})\,
\prod_{i=\al}^{\bt-1}\,\prod_{a=1}^{\la^{(i)}-1}\,
\frac {t_{\la^{(i)}}^{(i)}-t_a^{(i)}-h}{t_{\la^{(i)}}^{(i)}-t_a^{(i)}}
\>\Bigr)
\end{align}
and
\vvn-.6>
\begin{align}
\label{dabc}
\kern.6em\dch_{\TT^{\{\albt\}}}g\,&{}:=\,g(\TT_{\{\albt\}};\zz;h;\qq)\,
(t_{\la^{(\al-1)}}^{(\al-1)}\<-t_{\la^{(\al)}}^{(\al)}\<-h)\,
\prod_{i=\al}^{\bt-1}\,\prod_{a=1}^{\la^{(i-1)}-1}\!
(t_a^{(i-1)}\<\<-t_{\la^{(i)}}^{(i)}\<-h)\,\times{}
\\[2pt]
\notag
&{}\;\times\,(t_{\la^{(\bt-1)}}^{(\bt-1)}\<-t_{\la^{(\bt)}}^{(\bt)}\bigr)\,
\prod_{i=\al}^{\bt-1}\,\prod_{a=1}^{\la^{(i+1)}-1}\!
(t_{\la^{(i)}}^{(i)}-t_a^{(i+1)})\,
\prod_{i=\al}^{\bt-1}\,\prod_{a=1}^{\la^{(i)}-1}\,
\frac {t_{\la^{(i)}}^{(i)}-t_a^{(i)}-h}{t_{\la^{(i)}}^{(i)}-t_a^{(i)}}\;.
\end{align}
Denote
\beq
\label{ddab}
d_{\{\albt\}}g\,:=\,
\Sym_{\>t^{(1)}_1\!\lc\,t^{(1)}_{\la^{(1)}}}\,\ldots\;
\Sym_{\>t^{(N-1)}_1\!\lc\,t^{(N-1)}_{\la^{(N-1)}}}d_{\TT^{\{\albt\}}}g,
\vvn-.4>
\eeq
\beq
\label{tdab}
\dti_{\{\albt\}}g\,:=\,
\Sym_{\>t^{(1)}_1\!\lc\,t^{(1)}_{\la^{(1)}}}\,\ldots\;
\Sym_{\>t^{(N-1)}_1\!\lc\,t^{(N-1)}_{\la^{(N-1)}}}\dti_{\TT^{\{\albt\}}}g,
\vvn-.4>
\eeq
\beq
\label{cdab}
\dch_{\{\albt\}}g\,:=\,
\Sym_{\>t^{(1)}_1\!\lc\,t^{(1)}_{\la^{(1)}}}\,\ldots\;
\Sym_{\>t^{(N-1)}_1\!\lc\,t^{(N-1)}_{\la^{(N-1)}}}
\dch_{\TT^{\{\albt\}}}g\,.
\eeq
Then
\vvn-.4>
\beq
\label{dabg}
d_{\{\albt\}}g\,=\,\frac{q_\bt}{q_\al -q_\bt}\, \dti_{\{\albt\}}g
-\dch_{\{\albt\}}g.
\eeq

\begin{cor}
The function $\;d_{\{\albt\}}g$ \,is a discrete differential.
\qed
\end{cor}

\subsection{First key formula}
\label{sec fi}

Let \,$\bla\in\Z^N_{\geq 0}$\,, \,$|\bla|=n$.
For \,$\al,\bt=1\lc N$, \,$\al\ne\bt$\,, denote
\vvn.3>
\beq
\label{stri}
\bla_{\albt}\,=\,(\la_1\lc \la_\al-1\lc\la_\bt+1\lc\la_N)\,.\kern-1.4em
\vv.3>
\eeq
Notice that $|\>\bla_{\albt}\>|=|\>\bla\>|$\,.

\vsk.2>
Let \,$I=(I_1\lc I_N)\in\Il$, \,$I_k=(\ell_{k,1}\lc\ell_{k,\la_k})$,
\,$k=1\lc N$. For \,$\al\ne\bt$\,, \,$a=1\lc \la_\al$\,, \,$b=1\lc\la_\bt$\,,
denote
\vvn.3>
\beq
\label{abla}
(I)^{\<\>b}_{\btal}\,=\,
(I_1\lc I_\al\cup\{\ell_{\bt,b}\}\lc I_\bt-\{\ell_{\bt,b}\}\lc I_N)
\in\Ic_{\,\bla_{\albt}}\,.\kern-1.4em
\eeq

\vsk.2>
For \,$J\in\Ic_{\bla_{\albt}}$ \,and \,$b=1\lc \la_\bt+1$, we have
\,$(J)_{\btal}^{\<\>b}\in\Ic_\bla$\,. The function \,$U_J$ defined
by formula \Ref{hWI--} is a function of variables \,$\TT_{\albt}\>,\zz$\,.

\begin{thm}
\label{thm1}
We have
\vvn-.6>
\beq
\label{1id}
(\dti_{\{\albt\}} U_J)(\TT;\zz)\,=\,
-h\>\sum_{b=1}^{\la_\bt+1}\>W_{(J)_{\btal}^{\<\>b}}(\TT;\zz)\,.
\eeq
\end{thm}

Theorem \ref{thm1} is proved in Section \ref{prthm1}.

\subsection{Second key formula}
\label{seckey}
Let \,$I=(I_1\lc I_N)\in\Il$\,,
\,$I_k=(\ell_{k,1}\lc\ell_{k,\la_k})$\,. For \,$k_1, k_2=1\lc N$,
\,$k_1\ne k_2$, and \,$m_1=1\lc\la_{k_1}$\,, \,$m_2=1\lc\la_{k_2}$\,, define
the element \,$I_{k_1,k_2;m_1,m_2}=(\It_1\lc \It_N) \in \Il$
\,such that \,$\It_k = I_k$ \,if \,$k\ne k_1,k_2$\,, and
\vvn.4>
\beq
\label{Ikkmm}
\It_{k_1} = I_{k_1}\cup \{\ell_{k_2,m_2}\} - \{\ell_{k_1,m_1}\},\qquad
\It_{k_2}=I_{k_2}\cup \{\ell_{k_1,m_1}\} - \{\ell_{k_2,m_2}\}\,.
\vv.6>
\eeq

\begin{thm}
\label{2thm}
For $\;I\in\Il$ \,and $\;i=1\lc N-1$, we have
\vvn.4>
\begin{alignat}2
\label{44}
\Bigl(\,\sum_{j=1}^{\la^{(i)}}t^{(i)}_j
&{}-\sum_{j=1}^{\la^{(i-1)}}t^{(i-1)}_j-\sum_{a\in I_i}\,z_a \Bigr)\,W_I\,={}
\\
\notag
& {}=\,h\>\sum_{j=1}^{i-1}\,\sum_{m_1=1}^{\la_i}
\sum_{\satop{m_2=1}{\ell_{i,{m_1}}\<>\>\ell_{j,{m_2}}\!\!\!\!}}^{\la_j}
W_{I_{i,j;m_1,m_2}}\!
&{}-h\> \sum_{j=i+1}^N\,\sum_{m_1=1}^{\la_i}
\sum_{\satop{m_2=1}{\ell_{i,{m_1}}\<<\>\ell_{j,{m_2}}\!\!\!\!}}^{\la_j}
W_{I_{i,j;m_1,m_2}}\,+{}&
\\
\notag
&& \kern-4em
{}+\sum_{j=i+1}^N \sum_{a=1}^{\la_i} \dch_{\{i,j\}} U_{(I)_{\ij}^{\<\>a}}
- \sum_{j=1}^{i-1} \sum_{a=1}^{\la_j} \dch_{\{j,i\}}U_{(I)_{\ji}^{\<\>a}}
\,.&
\end{alignat}
\end{thm}

\vsk-.2>
Theorem \ref{2thm} is proved in Section \ref{prthm2}.

\section{Proof of Theorem \ref{thm1}}
\label{prthm1}

For $n=1$\,, Theorem \ref{thm1} is the following statement.

\begin{lem}
\label{lem n=1}
Let $n=1$.
For \,$1\le\ga\le n$\,, let \,$J^\ga=(J_1\lc J_N)$ \,be the decomposition
of the one-element set \,$\{1\}$\,, such that \,$J_\ga=\{1\}$ \,and
\,$J_j=\emptyset$ \,for \,$j\ne\ga$\,. Let \,$1\leq \al<\bt\leq N$. Then
\begin{alignat}2
\label{n1}
\dti_{\{\albt\}}\>U_{J^\ga}\, &{}=\,-hW_{J^\al}\,,\qquad &&\bt=\ga\,,
\\[4pt]
\label{n10}
\dti_{\{\albt\}}\>U_{J^\ga}\, &{}=\,0\,,\qquad &&\bt\ne\ga\,.
\end{alignat}
\end{lem}
\begin{proof}
For any $\ga$, we have $W_{J^\ga}=U_{J^\ga}$, and $U_{J^\ga}$ is the function of
$t^{(\ga)}_1\lc t^{(N-1)}_1, t^{(N)}_1=z_1$, which is identically equal to 1, see \Ref{hWI--}.

If \,$\bt=\ga$, then
\vvn.2>
\be
\dti_{\{\albt\}} U_{J^\ga}\,=\,\dti_{\{\al,\ga\}} U_{J^\ga}\,=\,
(t^{(\ga-1)}_1\<\<-t^{(\ga)}_1\<\<-h)-(t^{(\ga-1)}_1\<\<-t^{(\ga)}_1)\,=\,-h
\,=\,-h\>W_{J^\al}\,,
\vv.3>
\ee
which proves \Ref{n1}\,.

The proof of \Ref{n10} is by cases. If \,$\bt<\ga$, then \,$\dti_{\{\albt\}} U_{J^\ga}=(1-1)\cdot1=0$\,.
If \,$\ga<\al<\bt$, then \,$\dti_{\{\albt\}} U_{J^\ga}=0$ \,by identity \Ref{si}.
If \,$\al<\ga<\bt$, then \,$\dti_{\{\albt\}} U_{J^\ga}=0$ \,by identity \Ref{sisi}.
If \,$\al=\ga<\bt$, then \,$\dti_{\{\albt\}} U_{J^\ga}=0$ \,by the degeneration of
identity \Ref{si} as \,$t^{(0)}_1\!\to\infty$\,.
\end{proof}

For arbitrary $n$, Theorem \ref{thm1} follows by induction on \,$n$ \,from
the shuffle properties of weight functions in Lemma \ref{lem shuf}.
To avoid writing numerous indices we illustrate the reasoning by an example.

\vsk.3>
Let $N=3$, \,$n=2$, \,$J=(\emptyset,\{1,2\}, \emptyset)$, \,$\al=1$, \,$\bt=2$.
Then formula \Ref{1id} reads
\vvn.2>
\beq
\label{tie}
\dti_{\{1,2\}}U_J\,=\,
-hW_{(\{1\},\{2\}, \emptyset)}-hW_{(\{2\},\{1\}, \emptyset)}\,.
\eeq
Indeed, we have
\begin{align*}
& \begin{aligned}
\dti_{\{1,2\} }U_J\>=\,\Sym_{t^{(2)}_1, t^{(2)}_2}\Bigl(\<\<\bigl(
(t^{(1)}_1\<\<-t^{(2)}_1\<\<-h)\>(t^{(1)}_1\<\<-t^{(2)}_2\<\<-h)
-(t^{(1)}_1\<\<-t^{(2)}_1)\>(t^{(1)}_1\<\<-t^{(2)}_2)\bigr)\,&{}\times{}
\\
{}\times\,(t^{(2)}_1\<\<-z_2-h)\>(t^{(2)}_2\<\<-z_1)\,
\frac{t^{(2)}_2\<\<-t^{(2)}_1\<\<-h}{t^{(2)}_2\<\<-t^{(2)}_1}\>\Bigr)\>&{}={}
\end{aligned}
\\[4pt]
& \begin{aligned}
\phantom{\dti_{\{1,2\} }U_J}\>
=\,\Sym_{t^{(2)}_1, t^{(2)}_2}\Bigl(\<\<\bigl(
(t^{(1)}_1\<\<-t^{(2)}_1\<\<-h)\>(t^{(1)}_1\<\<-t^{(2)}_2\<\<-h)
-(t^{(1)}_1\<\<-t^{(2)}_1)\>(t^{(1)}_1\<\<-t^{(2)}_2\<\<-h) &\>{}+{}
\\[3pt]
{}+\>(t^{(1)}_1\<\<-t^{(2)}_1)\>(t^{(1)}_1\<\<-t^{(2)}_2\<\<-h)
-(t^{(1)}_1\<\<-t^{(2)}_1)(t^{(1)}_1\<\<-t^{(2)}_2)& \bigr)\,\times{}
\\
{}\times\,(t^{(2)}_1\<\<-z_2-h)\>(t^{(2)}_2\<\<-z_1)\,
\frac{t^{(2)}_2\<\<-t^{(2)}_1\<\<-h}{t^{(2)}_2\<\<-t^{(2)}_1}\,&\<\Bigr)\>.
\end{aligned}
\\[-24pt]
\end{align*}
This is the sum of four terms. The first two are
\vvn.3>
\begin{align*}
& \begin{aligned}
\Sym_{t^{(2)}_1, t^{(2)}_2}\Bigl(\<\<\bigl(
(t^{(1)}_1\<\<-t^{(2)}_1\<\<-h)\>(t^{(1)}_1\<\<-t^{(2)}_2\<\<-h)
-(t^{(1)}_1\<\<-t^{(2)}_1)\>(t^{(1)}_1\<\<-t^{(2)}_2\<\<-h)\bigr)\,&{}\times{}
\\
{}\times\,(t^{(2)}_1\<\<-z_2-h)\>(t^{(2)}_2\<\<-z_1)\,
\frac{t^{(2)}_2\<\<-t^{(2)}_1\<\<-h}{t^{(2)}_2\<\<-t^{(2)}_1}\>\Bigr)\>&{}={}
\end{aligned}
\\[4pt]
\noalign{\allowbreak}
&{}=\,-h\>\Sym_{t^{(2)}_1, t^{(2)}_2}\Bigl((t^{(1)}_1\<\<-t^{(2)}_2\<\<-h)
\>(t^{(2)}_1\<\<-z_2-h)\>(t^{(2)}_2\<\<-z_1)\,
\frac{t^{(2)}_2\<\<-t^{(2)}_1\<\<-h}{t^{(2)}_2\<\<-t^{(2)}_1}\>\Bigr)\,=\,
-h\>W_{\{1\},\{2\},\emptyset}\,,
\\[-20pt]
\end{align*}
the last two are
\vvn.2>
\begin{align*}
& \begin{aligned}
\Sym_{t^{(2)}_1, t^{(2)}_2}\Bigl(\<\<\bigl(
(t^{(1)}_1\<\<-t^{(2)}_1)\>(t^{(1)}_1\<\<-t^{(2)}_2\<\<-h)
-(t^{(1)}_1\<\<-t^{(2)}_1)\>(t^{(1)}_1\<\<-t^{(2)}_2)\bigr)\,&{}\times{}
\\
{}\times\,(t^{(2)}_1\<\<-z_2-h)\>(t^{(2)}_2\<\<-z_1)\,
\frac{t^{(2)}_2\<\<-t^{(2)}_1\<\<-h}{t^{(2)}_2\<\<-t^{(2)}_1}\>\Bigr)\>&{}={}
\end{aligned}
\\[4pt]
\noalign{\allowbreak}
&{}=\,-h\>\Sym_{t^{(2)}_1, t^{(2)}_2}\Bigl((t^{(1)}_1\<\<-t^{(2)}_2)
\>(t^{(2)}_1\<\<-z_2-h)\>(t^{(2)}_2\<\<-z_1)\,
\frac{t^{(2)}_2\<\<-t^{(2)}_1\<\<-h}{t^{(2)}_2\<\<-t^{(2)}_1}\>\Bigr)\,=\,
-h\>W_{\{2\},\{1\},\emptyset}\,,
\\[-20pt]
\end{align*}
and we get \Ref{tie}.

\vsk.2>
The treatment of these four terms is an inductive step from $n=1$ to $n=2$.
The analysis of the first two terms is the application of Theorem \ref{thm1}
for $n=1$ at the first point \,$z_1$. Namely, the factor
\,$\bigl((t^{(1)}_1\<\<-t^{(2)}_1\<\<-h)-(t^{(1)}_1\<\<-t^{(2)}_2)\bigr)$
\,corresponds to \,$\dti_{\{1,2\}}$ at \,$z_1$ and the product
\be
(t^{(1)}_1\<\<-t^{(2)}_2\<\<-h)\>(t^{(2)}_1\<\<-z_2-h)\>(t^{(2)}_2\<\<-z_1)\,
\frac{t^{(2)}_2\<\<-t^{(2)}_1\<\<-h}{t^{(2)}_2\<\<-t^{(2)}_1}
\vv.2>
\ee
is the connection coefficient between $W_{\{1\},\emptyset,\emptyset}$ sitting
at \,$z_1$ and $W_{\emptyset,\{2\},\emptyset}$ sitting at \,$z_2$, see
Lemma \ref{lem shuf}. And the analysis of the last two terms is the application
of Theorem \ref{thm1} for $n=1$ at the second point \,$z_2$. Namely, the factor
\,$\bigl((t^{(1)}_1\<\<-t^{(2)}_1\<\<-h)-(t^{(1)}_1\<\<-t^{(2)}_2)\bigr)$
\,corresponds to \,$\dti_{\{1,2\}}$ at \,$z_2$ and the product
\vvn-.6>
\be
(t^{(1)}_1\<\<-t^{(2)}_2)\>(t^{(2)}_1\<\<-z_2-h)\>(t^{(2)}_2\<\<-z_1)\,
\frac{t^{(2)}_2\<\<-t^{(2)}_1\<\<-h}{t^{(2)}_2\<\<-t^{(2)}_1}
\vv.2>
\ee
is the connection coefficient between $W_{\{2\},\emptyset,\emptyset}$ sitting
at \,$z_1$ and $W_{\emptyset,\{1\},\emptyset}$ sitting at \,$z_2$, see
Lemma \ref{lem shuf}.

\section{Proof of Theorem \ref{2thm}}
\label{prthm2}

\subsection{Proof of Theorem \ref{2thm} for $N=2$, \,$\bla=(n,0)$,
\,$I=(\{1\lc n\},\emptyset)$\,}

\begin{lem}
\label{l22}
We have
\vvn.2>
\beq
\label{2=2}
\sum_{l=1}^n\, (t^{(1)}_l - z_l)\,W^{\frak{gl}_2}_{\{1\lc n\},\emptyset}\,=\,
\sum_{a=1}^n\, \dch_{\{1,2\}}\, U^{\frak{gl}_2}_{\{1\lc a-1,a+1\lc n\},\{a\}}
\,.
\eeq
\end{lem}
\begin{proof}
We will prove formula \Ref{2=2} by induction on $n$. Denote
\be
\TT'=\<\>(t^{(1)}_1\lc t^{(1)}_{n-1})\,,\qquad
\TT''=\<\>(t^{(1)}_1\lc t^{(1)}_{n-2})\,,\qquad
\zz'\<=(z_1\lc z_{n-1})\,,
\ee
\be
A_n(\TT,\zz')\,=\,\prod_{b=1}^{n-1}\Bigl((t^{(1)}_n\<-z_b)\,
\frac{t_n^{(1)}\<-t_b^{(1)}\<-h}{t_n^{(1)}\<-t_b^{(1)}}\>\Bigr)\,,\qquad
B(\TT''\<\<,z)\,=\,\prod_{a=1}^{n-2}\,(t^{(1)}_a\<-z-h)\,.
\vv.3>
\ee
By formulae \Ref{dabc}\<\>, \Ref{cdab}\<\>, equality \Ref{2=2} reads as
follows:
\vvn-.2>
\begin{align}
\label{22}
\Sym_{t_1^{(1)}\lc t_n^{(1)}}\>\sum_{a=1}^n{}\Bigl(&(t^{(1)}_a\<-z_a)\,
U^{\frak{gl}_2}_{\{1\lc n\},\emptyset}(\TT,\zz)\,-{}
\\
\notag
{}-\,{} & (t^{(1)}_n\<-z_n)\,A_n(\TT,\zz')\,
U^{\frak{gl}_2}_{\{1\lc a-1,a+1\lc n\},\{a\}}(\TT'\<,\zz)\Bigr)
\>=\,0\,.
\kern-2em
\end{align}
For $n=1$\,, formula \Ref{22} is clearly true. For the induction step,
we explore formula \Ref{UI}\<\>. It implies that the summation term with
\,$a=n$ in formula \Ref{22} \,vanishes,
\beq
\label{UA}
U^{\frak{gl}_2}_{\{1\lc n\},\emptyset}(\TT,\zz)\,=\,
A_n(\TT,\zz)\,U^{\frak{gl}_2}_{\{1\lc n-1\},\emptyset}(\TT'\<,\zz')\,
B(\TT''\<\<,z_n)\,(t^{(1)}_{n-1}\<-z_n\<-h)\,,
\vv.1>
\eeq
and for \,$a<n$\,,
\vvn.1>
\begin{align}
\label{UU}
& U^{\frak{gl}_2}_{\{1\lc a-1,a+1\lc n\},\{a\}}(\TT'\<,\zz)\,={}
\\[4pt]
\notag
&{}\!=\,(t^{(1)}_{n-1}\<-z_{n-1})\,A_{n-1}(\TT'\<,\zz')\,
U^{\frak{gl}_2}_{\{1\lc a-1,a+1\lc n-1\},\{a\}}(\TT''\<,\zz')\,
B(\TT''\<\<,z_n)\,.
\end{align}
The last formula and the identity
\vvn.1>
\be
\Sym_{t^{(n)}_{n-1},\>t^{(n)}_n}\,(t^{(n)}_n\<-z_n)\,
\frac{t^{(n)}_n\<-t^{(n)}_{n-1}\<-h}{t^{(n)}_n\<-t^{(n)}_{n-1}}\,=\,
\Sym_{t^{(n)}_{n-1},\>t^{(n)}_n}\,(t^{(n)}_{n-1}\<-z_n\<-h)\,
\frac{t^{(n)}_n\<-t^{(n)}_{n-1}\<-h}{t^{(n)}_n\<-t^{(n)}_{n-1}}\;,
\vv-.1>
\ee
yield
\vvn-.2>
\begin{align}
\Sym_{t_1^{(1)}\lc t_n^{(1)}}\,(t^{(1)}_n\<-z_n)\, & A_n(\TT,\zz')\,
U^{\frak{gl}_2}_{\{1\lc a-1,a+1\lc n\},\{a\}}(\TT'\<,\zz)\,={}
\\[4pt]
\notag
{}=\,\Sym_{t_1^{(1)}\lc t_n^{(1)}}\, & (t^{(1)}_{n-1}\<-z_{n-1})\,
A_n(\TT,\zz')\,A_{n-1}(\TT'\<,\zz')\,\times{}
\\[4pt]
\notag
{}\times\<\>{}&\> U^{\frak{gl}_2}_{\{1\lc a-1,a+1\lc n-1\},\{a\}}(\TT''\<,\zz')\,
B(\TT''\<\<,z_n)\,(t^{(1)}_{n-1}\<-z_n\<-h)\,.
\\[-14pt]
\notag
\end{align}
Summarizing all observations, we see that formula \Ref{22}
follows from the equality
\vvn.2>
\begin{align}
\label{22a}
\Sym_{t_1^{(1)}\lc t_{n-1}^{(1)}} A_n(\TT, & \zz')\,
B(\TT''\<\<,z_n)\,(t^{(1)}_{n-1}\<-z_n\<-h)\,\times{}
\\[1pt]
\notag
{}\times\>\sum_{a=1}^{n-1}\Bigl({}& (t^{(1)}_a\<-z_a)
\,U^{\frak{gl}_2}_{\{1\lc n-1\},\emptyset}(\TT'\<,\zz')\,-{}
\\
\notag
{}-\,{} & (t^{(1)}_{n-1}\<-z_{n-1})\,A_{n-1}(\TT'\<,\zz')\,
U^{\frak{gl}_2}_{\{1\lc a-1,a+1\lc n-1\},\{a\}}(\TT''\<,\zz')\Bigr)\>=\,0\,,
\end{align}
with \;$t^{(1)}_n$ not involved in the symmetrization. Since the product
\vvn.3>
\be
A_n(\TT,\zz')\,B(\TT''\<\<,z_n)\,(t^{(1)}_{n-1}\<-z_n\<-h)
\vv.1>
\ee
is symmetric in \;$t^{(1)}_1\<\<\lc t^{(1)}_{n-1}$\,, formula \Ref{22a}
follows from the induction assumption
\vvn.3>
\begin{align}
\label{22'}
\Sym_{t_1^{(1)}\lc t_{n-1}^{(1)}}\>\sum_{a=1}^{n-1}{}\Bigl(&(t^{(1)}_a\<-z_a)
\,U^{\frak{gl}_2}_{\{1\lc n-1\},\emptyset}(\TT'\<,\zz')\>-{}
\\
\notag
{}-\,{} &(t^{(1)}_{n-1}\<-z_{n-1})\,A_{n-1}(\TT'\<,\zz')\,
U^{\frak{gl}_2}_{\{1\lc a-1,a+1\lc n-1\},\{a\}}(\TT''\<,\zz')\Bigr)\>=\,0\,.
\kern-2em
\\[-22pt]
\notag
\end{align}
Lemma \ref{l22} is proved.
\end{proof}

\subsection{Proof of Theorem \ref{2thm} for $N=2$ and \,$I=\Ima$}
For $N=2$, \,$\bla=(k,n-k)$\,, we denote
\,$\Ima\<=\bigl(\{n-k+1\lc n\}, \{1\lc n-k\}\bigr)$\,.
Then formula \Ref{44} becomes formula \Ref{22m} below.

\begin{lem}
\label{lem N2m}
We have
\begin{align}
\label{22m}
\sum_{l=1}^{k}\,(t^{(1)}_l - z_{n-k+l}) \,& W^{\frak{gl}_2}_{\{n-k+1\lc n\},
\{1\lc n-k\}}
\\[-3pt]
\notag
{}=\!\sum_{a=n-k+1}^n\!\dch_{\{1,2\}} \,&
U^{\frak{gl}_2}_{\{n-k+1\lc a-1,a+1\lc n\},\{1\lc n-k,a\}},
\end{align}
\end{lem}
\begin{proof}
Dividing both sides of the equation by
\,$\prod_{i=1}^k\prod_{a=1}^{n-k}(t^{(1)}_l\!-z_a)$\,,
turns formula \Ref{22m} into formula \Ref{2=2}.
\end{proof}

\subsection{Proof of Theorem \ref{2thm} for $i=1$, arbitrary $N$, and
\,$I=(\{1\lc n\},\emptyset\lc\emptyset)\,$}

\begin{prop}
\label{prop empty}
For $I=(\{1\lc n\},\emptyset\lc\emptyset)$\,, we have
\vvn.2>
\beq
\label{1ms}
\Bigl(\,\sum_{l=1}^n\,t^{(1)}_l -\sum_{a=1}^n\,z_a \Bigr)\,W_I\,=\,
\sum_{j=2}^n\,\sum_{a=1}^n\, \dch_{\{1,j\}}\, U_{(I)_{1,j}^{'a}}\,.
\eeq
\end{prop}
\begin{proof}
Formula \Ref{1ms} is equivalent to the formula
\vvn.2>
\be
\sum_{j=2}^n\,\sum_{l=1}^n\,(t^{(j-1)}_l\!- t^{(j)}_l)\,W_I\,=\,
\sum_{j=2}^n\,\sum_{a=1}^n\,\dch_{\{1,j\}}\,U_{(I)_{1,j}^{'a}}
\vv.2>
\ee
which follows from the next lemma.

\begin{lem}
\label{lem in}
For $j=2\lc N$ we have
\beq
\label{1j}
\sum_{l=1}^n\,(t^{(j-1)}_l\!- t^{(j)}_l)\,W_I\,=\,
\sum_{a=1}^n\,\dch_{\{1,j\}}\,U_{(I)_{1,j}^{'a}}
\eeq
\end{lem}
\begin{proof}
By Lemma \ref{lem fa}, the left-hand side of \Ref{1j} equals
\vvn.2>
\beq
\label{1jl}
\Bigl(\,\sum_{l=1}^n\, (t^{(j-1)}_l\!- t^{(j)}_l)\,
W^{\frak{gl}_2}_{\{1\lc n\},\emptyset}(\TT^{(j-1)};\TT^{(j)})\Bigr)\,
\prod_{\satop{i=2}{i\ne j}}^N\,
W^{\frak{gl}_2}_{\{1\lc n\},\emptyset}(\TT^{(i-1)};\TT^{(i)})\,.
\vv-.2>
\eeq
It is easy to see that the right hand side equals
\vvn.2>
\beq
\label{rhj}
\Bigl(\,\sum_{a=1}^n\,\dch_{\{j-1,j\}}\,
U_{(\{1\lc a-1,a+1\lc n\}, \{a\})} (\TT^{(j-1)};\TT^{(j)})\Bigr)\,
\prod_{\satop{i=2}{i\ne j}}^N\,
W^{\frak{gl}_2}_{\{1\lc n\},\emptyset}(\TT^{(i-1)};\TT^{(i)})\,.
\vv-.2>
\eeq
Hence, Lemma~\ref{lem in} follows from formula~\Ref{2=2}.
\end{proof}
Proposition \ref{prop empty} is proved.
\end{proof}

\subsection{Proof of Theorem \ref{2thm} for $i=1$, arbitrary $N$,
and \,$I=\Ima$}
For \,$\bla=(\la_1\lc\la_N)$\,, we denote
\,$\Ima\<=\bigl(\{n-\la_1+1\lc n\}\lc\{1\lc\la_N\}\bigr)$\,.
Then formula \Ref{44} takes the form
\beq
\label{i=1m}
\sum_{j=2}^N\,\sum_{l=1}^{\la_1}\,(t^{(j-1)}_{l+\la^{(j-1)}-\la_1}\!-
t^{(j)}_{l+\la^{(j)}-\la_1})\,W_{\Ima}\,=\,
\sum_{j=2}^N\,\sum_{a=1}^{\la_1}\,
\dch_{\{1,j\}} W_{(\Ima)_{1,j}^{'a}}\,.
\vv.3>
\eeq
The following lemma implies formula \Ref{i=1m}.

\begin{lem}
\label{lem i1}
For $j=2\lc N$, we have
\beq
\label{i1}
\sum_{l=1}^{\la_1}\,(t^{(j-1)}_{l+\la^{(j-1)}\!-\la_1} -
t^{(j)}_{l+\la^{(j)}-\la_1})\,W_{\Ima}\,=\,
\sum_{a=1}^{\la_1}\, \dch_{\{1,j\}}\, W_{(\Ima)_{1,j}^{'a}}\,.
\eeq
\end{lem}
\begin{proof}
The left-hand side of formula \Ref{i1} equals
\begin{align}
\label{lh1j}
\quad\Sym_{\>t^{(1)}_1\!\lc\,t^{(1)}_{\la^{(1)}}}\ldots\;
\Sym_{\>t^{(N-1)}_1\!\lc\,t^{(N-1)}_{\la^{(N-1)}}}\Bigl(\,
\sum_{l=1}^{\la_1}(t^{(j-1)}_{l+\la^{(j-1)}-\la_1}\!-
t^{(j)}_{l+\la^{(j)}-\la_1}) &
\\
\notag
{}\times\prod_{m=2}^N\,U^{\frak{gl}_2}_{\{\la_m+1\lc\la^{(m)}\},\{1\lc\la_m\}}
(\TT^{(m-1)};\TT^{(m)}) &{}\Bigr)\>,\kern-1em
\end{align}
while the right-hand side of \Ref{i1} equals by definition
\begin{align}
\label{rh1j}
\Sym_{\>t^{(1)}_1\!\lc\,t^{(1)}_{\la^{(1)}}}\ldots\;
\Sym_{\>t^{(N-1)}_1\!\lc\,t^{(N-1)}_{\la^{(N-1)}}}
\Bigl(\;\prod_{l=1}^{\la^{(j)}}(t^{(j-1)}_{\la^{(j-1)}}-t^{(j)}_{l})
\prod_{l=1}^{\la^{(j-1)}-1}\,\frac{t^{(j-1)}_{\la^{(j-1)}}-t^{(j-1)}_l\!-h}
{t^{(j-1)}_{\la^{(j-1)}}-t^{(j-1)}_{l}} &
\\
\notag
{}\!\times\sum_{b=1+\la^{(j)}-\la_1}^{\la^{(j)}}\,U^{\frak{gl}_2}
_{\{\la_j+1\lc\la^{(j)}-\la_1\lc b-1,b+1\lc\la^{(j)}\},\{1\lc\la_j, b\}}
(\TT^{(j-1)}\setminus\{t^{(j-1)}_{\la^{(j-1)}}\}\>,\,\TT^{(j)} &{})
\\
\notag
{}\times\>\prod_{\satop{m=2}{m\ne j}}^N\,U^{\frak{gl}_2}
_{\{\la_m+1\lc\la^{(m)}\},\{1\lc\la_m\}}(\TT^{(m-1)},\TT^{(m)} &{})\Bigr)\>.
\kern-1em
\end{align}
The equality of \Ref{lh1j} and \Ref{rh1j} follows from the following case of
formula \Ref{2=2}\>:
\vvn.3>
\begin{align*}
\sum_{l=1}^{\la_1}\,(t^{(j-1)}_{l+\la^{(j-1)}-\la_1}\!-
t^{(j)}_{l+\la^{(j)}-\la_1})\,W^{\frak{gl}_2}_{\{1\lc\la_1\},\emptyset}
(t^{(j-1)}_{1+\la^{(j-1)}-\la_1}\lc t^{(j-1)}_{\la^{(j-1)}};
t^{(j)}_{1+\la^{(j)}-\la_1}\lc t^{(j)}_{\la^{(j)}}) &
\\
{}=\,\sum_{a=1}^{\la_1}\,\dch_{\{1,2\}}\,
U^{\frak{gl}_2}_{\{1\lc a-1,a+1\lc\la_1\},\{a\}}
(t^{(j-1)}_{1+\la^{(j-1)}-\la_1}\lc t^{(j-1)}_{\la^{(j-1)}-1};
t^{(j)}_{1+\la^{(j)}-\la_1}\lc t^{(j)}_{\la^{(j)}}) &\,.
\\[-40pt]
\end{align*}
\vv>
\end{proof}

\subsection{Proof of Theorem \ref{2thm} for \,$i>1$, arbitrary $N$, \,and
\,$I=(\{1\lc n\},\emptyset\lc\emptyset)\,$}

For $i=2\lc N-1$\,, and \,$I=(\{1\lc n\},\emptyset\lc\emptyset)$\,,
Theorem \ref{2thm} says that
\vvn.3>
\be
\sum_{l=1}^n\,(t^{(i)}_l\<\<- t^{(i-1)}_l)\, W_I\,=\,
-\>\sum_{a=1}^n\,\dch_{\{1,i\}}\,U_{(I)_{1,j}^{'a}}\>,
\vv.1>
\ee
which is formula \Ref{1j}.

\subsection{Proof of Theorem \ref{2thm} for \,$i>1$, arbitrary \,$\bla$\,, and
\,$I=\Ima$}

To prove this case of Theorem \ref{2thm}, we introduce a partition
$I^{\<\>\max,\>j}=(I^{\<\>\max,\>j}_1\lc I^{\<\>\max,\>j}_j)$ of the set
$(1\lc \la^{(j)})$ by the rule
\beq
\label{Imax,j}
I^{\<\>\max,\>j}_a =\,\{\>i\ |\ \la^{(j)}\!-\la^{(a)}\< < i\leq
\la^{(j)}\!-\la^{(a-1)}\>\}\,,
\vv.1>
\eeq
so that $\;|I^{\<\>\max,\>j}_a|=\la_a$\,. For example,
\,$I^{\<\>\max,\>N}\<=\>\Ima$.

\vsk.3>
Formula \Ref{44} for \,$I=\Ima$ can be written as
\begin{align}
\label{44max}
&\sum_{l=1}^{j-1}\,\Bigl(\<\Bigl(\>\sum_{i\in I^{\<\>\max,\<\>j}_l}\!t^{(j)}_i
-\!\sum_{i\in I^{\<\>\max,\<\>j-1}_l}\!t^{(j-1)}_i\Bigr)\,W_{\Ima}\>+
\>\sum_{a=1}^{\la_l}\,\dch_{\{l,j\}}\,U_{(\Ima)_{l,j}^{'a}}\>\Bigr)
\\
\notag
\quad{}+{} & \sum_{l=j+1}^N \Bigl(\<\Bigl(\>\sum_{i\in I^{\<\>\max,\<\>l-1}_j}
\!t^{(l-1)}_j-\!\sum_{j\in I^{\<\>\max,\<\>l}_i}\!t^{(l)}_i\Bigr)\,
W_{\Ima}\>-\>\sum_{a=1}^{\la_l}\,
\dch_{\{j,l\}}\,U_{(\Ima)_{j,l}^{'a}}\>\Bigr)\,=\,0\,.\kern-1em
\end{align}
Formula \Ref{44max} follows from the next Proposition.

\begin{prop}
\label{prop il}
For $i=1,2\lc N-1$\,, and \,$j=i,i+1\lc N-1$\,, we have
\beq
\label{4m1}
\Bigl(\>\sum_{l\in I^{\<\>\max,\<\>j}_i}\!t^{(j)}_l\<-\!
\sum_{l\in I^{\<\>\max,\<\>j}_i}\!t^{(j+1)}_l\Bigr)\,W_{\Ima}\,=\,
\sum_{a=1}^{\la_i}\,\dch_{\{i,j+1\}}\,U_{(\Ima)_{i,j+1}^{'a}}\>.
\eeq
\end{prop}
\begin{proof}
For \,$i=1$\,, formula \Ref{4m1} follows from Lemma \ref{lem i1}.
For \,$i>1$, we prove formula \Ref{4m1} by induction on \,$\la^{(i-1)}$,
see Lemmas~\ref{ls1} and~\ref{ls2} below. If \,$\la^{(i-1)}\!=0$\,, that is,
\,$\la_j=0$ \,for all \,$j=1\lc i-1$\,, formula \Ref{4m1} follows from
Lemma~\ref{lem i1} by renaming variables.

\vsk.2>
We will indicate explicitly the dependence of the partitions \,$\Ima$,
\,$I^{\<\>\max,\<\>l}$ on \,$\bla$\,:
\vvn.2>
\be
\Ima_\bla=\>(\Ima_{\bla,1}\lc\Ima_{\bla,N})\,,\qquad
I^{\<\>\max,\<\>l}_\bla=\>(I^{\<\>\max,\<\>l}_{\bla,1}\lc
I^{\<\>\max,\<\>l}_{\bla,l})\,.
\vv.1>
\ee
We fix \,$i,j$ \,until the end of the proof of Proposition~\ref{prop il},
and omit the condition \,$|\bla|=n$\,.

\begin{lem}
\label{ls1}
Assume that formula \Ref{4m1} holds for \,$\bla=(0\lc 0,\la_k\lc\la_N)$
\,with \,$k\leq i$. Then formula \Ref{4m1} holds for
\,$\blat=(0\lc 0,1,\la_k\lc\la_N)$\,.
\end{lem}
\begin{proof}
Formula \Ref{4m1} for \,$\bla$ \,has the form
\vvn.3>
\begin{align}
\label{la}
\Bigl(\>\sum_{l\in I^{\<\>\max,\<\>j}_{\blai}}\!t^{(j)}_l-\!
\sum_{l\in I^{\<\>\max,\<\>j+1}_{\blai}}\!t^{(j+1)}_l\>\Bigr)\,
&\Sym_{\TT^{(k)}}\dots\,\Sym_{\TT^{(N-1)}}
\bigl(U_{\Ima_\bla}(\TT)\bigr)
\\[-5pt]
\notag
{}=\,\Sym_{\TT^{(k)}}\dots\,&\Sym_{\TT^{(N-1)}}
\Bigl(C_{\bla,\ij+1}\>\sum_{a=1}^{\la_i}\,
U_{(\Ima_{\bla})_{\ij+1}^{'a}}(\TT)\<\Bigr)\>,\kern-1em
\end{align}
where $C_{\bla,\>\ij+1}$ is the factor in the second and third lines of
definition \Ref{dabc}.

\vsk.2>
In addition to the variables \,$\TT=(\TT^{(k)}\lc\TT^{(N)})$ \,appearing
in formula \Ref{la}, formula \Ref{4m1} for \,$\blat$ \,contains
the new variables
$\;\TT_\new=\>
\bigl(\>t^{(k-1)}_1,\alb\,t^{(k)}_{1+\la^{(k)}}\>,\alb\,
t^{(k+1)}_{1+\la^{(k+1)}}\>\lc\,t^{(N-1)}_{1+\la^{(N-1)}}\>,\alb
\,t^{(N)}_{1+\la^{(N)}}\bigr)$\,, and has the form
\vvn.5>
\begin{align}
\label{tila}
\Bigl(\>\sum_{l\in I^{\<\>\max,\<\>j}_{\blat,i}}\!t^{(j)}_l-\!
\sum_{l\in I^{\<\>\max,\<\>j+1}_{\blat,i}}\!t^{(j+1)}_l\>\Bigr)\,
&\Sym_{\TTT^{(k)}}\dots\,\Sym_{\TTT^{(N-1)}}
\bigl(U_{\Ima_\blat}(\TTT)\bigr)
\\[-7pt]
\notag
{}=\,\Sym_{\TTT^{(k)}}\dots\,&\Sym_{\TTT^{(N-1)}}
\Bigl(C_{\blat,\ij+1}\>\sum_{a=1}^{\la_i}\,
U_{(\Ima_\blat)_{\ij+1}^{'a}}(\TTT)\<\Bigr)\>,\kern-1em
\end{align}
where \,$\TTT=\>\TT\cup\TT_\new=(\TTT^{(k-1)},\TTT^{(k)}\lc\TTT^{(N)})$\,.
It is easy to see from definition \Ref{UI} that
\vvn.3>
\beq
\label{U}
U_{\Ima_\blat}(\TTT)\,=\,U_{\Ima_{\bla}}(\TT)\,F(\TTT)\,,
\vv.2>
\eeq
where \,$F(\TTT)$ \,is the product of all factors appearing in \Ref{UI}
involving the interrelation of two variables at least one of those being from
\,$\TT_\new$\,. Moreover, \,$F(\TTT)$ \,is symmetric in the variables
\,$\TT^{(l)}$ \,for each \,$l=k\lc N$. Furthermore,
since \,$I^{\<\>\max,\<\>j}_{\blai}= I^{\<\>\max,\<\>j}_{\blat,i}$ \,and
\,$I^{\<\>\max,\<\>j+1}_{\blai}= I^{\<\>\max,\<\>j+1}_{\blat,i}$\,,
the first factors in the left-hand sides of formulas \Ref{la} and \Ref{tila}
coincide.

\vsk.2>
By all these observations, to get formula \Ref{tila} from \Ref{la},
we need to verify that
\vvn.2>
\beq
\label{ll}
\Sym_{\TTT^{(k)}}\dots\,\Sym_{\TTT^{(N-1)}} \Bigl(\>C_{\blat,\>\ij+1}\,
\sum_{a=1}^{\la_i}\,U_{(\Ima_\blat)_{i,j+1}^{'a}}\!-
F\>C_{\bla,\>\ij+1}\,
\sum_{a=1}^{\la_i}\,U_{(\Ima_{\bla})_{i,j+1}^{'a}}\>\Bigr)\>=\,0\,.
\vv.2>
\eeq
This equality follows from identity \Ref{si} for the variables
\,$t^{(i-1)}_{1+\la^{(i-1)}}\>,\alb\,t^{(i)}_{\la^{(i)}}\>,\alb
\,t^{(i)}_{1+\la^{(i)}}\>\lc\,t^{(j)}_{\la^{(j)}}\>,\alb
\,t^{(j)}_{1+\la^{(j)}}\>,\alb\,t^{(j+1)}_{1+\la^{(j+1)}}$\,.
Lemma \ref{ls1} is proved.
\end{proof}

\begin{example}
Let \,$N=5$\,, \,$\bla=(0,0,1,0,0)$\,, \,$\blat=(1,0,1,0,0)$\,.
For \,$i=3$\,, \,$j=3$\,, formulas \Ref{la} and \Ref{ll} take the form
$\;t^{(3)}_1\!-t^{(4)}_1\<=\>t^{(3)}_1\!-t^{(4)}_1$ and
\vvn.2>
\begin{align*}
\Sym_{t^{(3)}_1\!\<,\>t^{(3)}_2}\,\Sym_{t^{(4)}_1\!\<,\>t^{(4)}_2}
\Bigl( & (t^{(3)}_1\!-t^{(4)}_1)\>(t^{(2)}_1\!-t^{(3)}_1)\>
(t^{(3)}_1\!-t^{(4)}_2-h)\>(t^{(3)}_2\!-t^{(4)}_1)
\\
{}\times\,{} &(t^{(4)}_1\!-t^{(5)}_2\!-h)\>(t^{(4)}_2-t^{(5)}_1)\;
\frac{t^{(3)}_2\!-t^{(3)}_1\!-h}{t^{(3)}_2\!-t^{(3)}_1}\,\,
\frac{t^{(4)}_2\!-t^{(4)}_1\!-h}{t^{(4)}_2\!-t^{(4)}_1}\>\Bigr)
\\[6pt]
{}=\,\Sym_{t^{(3)}_1\!\<,\>t^{(3)}_2}\,\Sym_{t^{(4)}_1\!\<,\>t^{(4)}_2}
\Bigl( & (t^{(3)}_1\!-t^{(4)}_2\!-h)\>(t^{(2)}_2\!-t^{(3)}_1)\>
(t^{(3)}_2\!-t^{(4)}_2)\>(t^{(3)}_1\!-t^{(4)}_1)
\\
{}\times\,{} &(t^{(4)}_1\!-t^{(5)}_2\!-h)\>(t^{(4)}_2-t^{(5)}_1)\;
\frac{t^{(3)}_2\!-t^{(3)}_1\!-h}{t^{(3)}_2\!-t^{(3)}_1}\,\,
\frac{t^{(4)}_2\!-t^{(4)}_1\!-h}{t^{(4)}_2\!-t^{(4)}_1}\>\Bigr)\,,
\\[-16pt]
\end{align*}
respectively.
The last equality follows from identity \Ref{si} for the variables
$\;t^{(2)}_1\<,\alb\,t^{(3)}_1\<,\alb\,t^{(3)}_2\<,\alb\,t^{(4)}_2$.

\vsk.3>
For \,$i=3$\,, \,$j=4$\,, formulas \Ref{la} and \Ref{ll} take the form
$\;t^{(4)}_1\!-t^{(5)}_1\<=\>t^{(4)}_1\!-t^{(5)}_1$ and
\vvn.2>
\begin{align*}
\Sym_{t^{(3)}_1\!\<,\>t^{(3)}_2}\,\Sym_{t^{(4)}_1\!\<,\>t^{(4)}_2}
\Bigl( & (t^{(4)}_1\!-t^{(5)}_1)\>(t^{(2)}_1\!-t^{(3)}_1)\>
(t^{(3)}_1\!-t^{(4)}_2\!-h)\>(t^{(3)}_2\!-t^{(4)}_1)
\\
{}\times\,{} & (t^{(4)}_1\!-t^{(5)}_2-h)\>(t^{(4)}_2\!-t^{(5)}_1)\;
\frac{t^{(3)}_2\!-t^{(3)}_1\!-h}{t^{(3)}_2\!-t^{(3)}_1}\,\,
\frac{t^{(4)}_2\!-t^{(4)}_1\!-h}{t^{(4)}_2\!-t^{(4)}_1}\>\Bigr)
\\[6pt]
{}=\,\Sym_{t^{(3)}_1\!\<,\>t^{(3)}_2}\,\Sym_{t^{(4)}_1\!\<,\>t^{(4)}_2}
\Bigl( & (t^{(2)}_1\!-t^{(3)}_2\!-h)\>(t^{(2)}_2\!-t^{(3)}_1)\>
(t^{(3)}_1\!-t^{(4)}_2\!-h)\>(t^{(4)}_2\!-t^{(5)}_1)
\\
& \hphantom{h}\!\times\,(t^{(4)}_2\!-t^{(5)}_2)\>(t^{(4)}_1\!-t^{(5)}_1)\;
\frac{t^{(3)}_2\!-t^{(3)}_1\!-h}{t^{(3)}_2\!-t^{(3)}_1}\,\,
\frac{t^{(4)}_2\!-t^{(4)}_1\!-h}{t^{(4)}_2\!-t^{(4)}_1}\>\Bigr)\,,
\\[-16pt]
\end{align*}
respectively.
The last equality follows from identity \Ref{si} for the variables
$\;t^{(2)}_1\<,\alb\,t^{(3)}_1\<,\alb\,t^{(3)}_2\<,\alb\,t^{(4)}_1\<,\alb\,
t^{(4)}_2\<,\alb\,t^{(5)}_2$.
\end{example}

\begin{lem}
\label{ls2}
Assume that formula \Ref{4m1} holds for \,$\bla=(0\lc 0,\la_k\lc\la_N)$
\,with \,$k<i$ \,and \,$\la_k>0$\,. Then formula \Ref{4m1} holds for
\,$\blat=(0\lc 0,0,\la_k\<+1\lc\la_N)$\,.
\end{lem}
\begin{proof}
The proof is completely similar to that of Lemma~\ref{ls1}.
The only change is that the new variables are
$\;\TT_\new=\>\bigl(\>t^{(k)}_{1+\la^{(k)}}\>,\alb\,
t^{(k+1)}_{1+\la^{(k+1)}}\>\lc\,t^{(N-1)}_{1+\la^{(N-1)}}\>,\alb
\,t^{(N)}_{1+\la^{(N)}}\bigr)$\,.
\end{proof}

\begin{example}
Let \,$N=3$\,, \,$\bla=(1,1,0)$\,, \,$\blat=(2,1,0)$\,. For \,$i=2$\,,
\,$j=2$\,, formula \Ref{ll} proof follows from identity \Ref{si} for
the variables
$\;t^{(1)}_2\<,\alb\,t^{(2)}_2\<,\alb\,t^{(2)}_3\<,\alb\,t^{(3)}_3$.
\end{example}

Lemmas \ref{ls1} and \ref{ls2} yield Proposition \ref{prop il}.
\end{proof}
Theorem \ref{2thm} for \,$i>1$\,, arbitrary \,$N$, \,$\bla$\,, and
\,$I=\Ima$ \>is proved.

\subsection{Modification of the three-term relation}
For integers \,$\al,\bt$\,, $\;1\leq\al<\bt\leq N$, and
\,$\bla\in\Z^N_{\geq 0}$\,, \,$|\bla|=n$\,, \,recall the notations
$\;\TT^{\{\albt\}}$, \,$\TT_{\{\albt\}}$\,, \>$\bla_{\albt}$\,,
in Sections \ref{sec sdd} and \ref{sec fi}.

\begin{lem}
\label{dUW}
For any \,$1\leq \al<\bt \leq N$ and \,$I\in\Ic_{\bla_{\albt}}$\,,
\vvn.1>
we have $\;\dch_{\{\albt\}} U_I\>=\>c_{\albt}\,\dch_{\{\albt\}} W_I$\,,
where \;$\dsize c_{\albt}=\>
\frac{\prod_{i=\al}^{\bt-1}\la^{(i)}}{\prod_{i=1}^{N-1}\la^{(i)}!}\;$.
\end{lem}
\begin{proof}
Let
\vvn->
\begin{align}
\label{Gtz}
G_{\albt}(\TT,\zz)&{}\,=\,
(t_{\la^{(\al-1)}}^{(\al-1)}\<-t_{\la^{(\al)}}^{(\al)}\<-h)\,
\prod_{i=\al}^{\bt-1}\,\prod_{a=1}^{\la^{(i-1)}-1}\!
(t_a^{(i-1)}\<\<-t_{\la^{(i)}}^{(i)}\<-h)\,\times{}
\\[2pt]
\notag
&{}\;\times\,(t_{\la^{(\bt-1)}}^{(\bt-1)}\<-t_{\la^{(\bt)}}^{(\bt)}\bigr)\,
\prod_{i=\al}^{\bt-1}\,\prod_{a=1}^{\la^{(i+1)}-1}\!
(t_{\la^{(i)}}^{(i)}-t_a^{(i+1)})\,
\prod_{i=\al}^{\bt-1}\,\prod_{a=1}^{\la^{(i)}-1}\,
\frac {t_{\la^{(i)}}^{(i)}-t_a^{(i)}-h}{t_{\la^{(i)}}^{(i)}-t_a^{(i)}}
\\[-14pt]
\notag
\end{align}
be the product in the right-hand side of formula \Ref{dabc}.
Since \,$G_{\albt}(\TT,\zz)$ \,is symmetric in the variables
\,$\TT_{\{\albt\}}^{(i)}$ \,for every \,$i=1\lc N-1$\,, we can apply
the symmetrization in those variables to \,$U_I(\TT_{\{\albt\}},\zz)$
and divide the result by the order of the relevant product of the symmetric
groups \,before doing the overall symmetrization in formula \Ref{cdab} for
\,$\dch_{\{\albt\}} U_I$\,. This results in replacing
\,$U_I(\TT_{\{\albt\}},\zz)$ \,by \,$c_{\albt}\,W_I(\TT_{\{\albt\}},\zz)$\,,
see formula \Ref{hWI--}.
\end{proof}

Recall the operator \,$S_{\ii+1}$ acting on functions of \,$z_1\lc z_n$
\,given by formula \Ref{Skk+1}.

\begin{lem}
\label{3ctt}
For any $\;i=1\lc n-1$\,, $\;1\leq \al<\bt \leq N$, and
$\;I\in\Ic_{\bla_{\albt}}$\,, we have
\vvn.3>
\beq
\label{U3t}
S_{\ii+1}(\dch_{\{\albt\}} U_I)\,=\,\dch_{\{\albt\}} U_{s_{\ii+1}(I)}\,.
\vv.3>
\eeq
\end{lem}
\begin{proof}
The product \,$G_{\albt}(\TT,\zz)$\,, see~\Ref{Gtz}, is symmetric in
$z_1\lc z_n$. Hence
\vvn.3>
\begin{align*}
S_{\ii+1}(\dch_{\{\albt\}}U_I)\,=\,c_{\albt}\,S_{\ii+1}(\dch_{\{\albt\}}W_I)\,
\,=\,c_{\albt}\,\dch_{\{\albt\}}\>\bigl(S_{\ii+1}(W_I)\bigr) &
\\[3pt]
{}=\,c_{\albt}\,\dch_{\{\albt\}}W_{s_{\ii+1}(I)}\,=\,
\dch_{\{\albt\}} U_{s_{\ii+1}(I)} & \,.
\\[-24pt]
\end{align*}
by Lemmas \ref{dUW} and \ref{c3t}.
\end{proof}

\subsection{The end of the proof of Theorem \ref{2thm}}
\label{secsiI}
Given $\;l$, \,$1\leq l\leq N-1$\,, we add formulas \Ref{44} for
\,$i=1\lc l$\,. The result is
\begin{align}
\label{44f}
\Bigl(\,\sum_{j=1}^{\la^{(l)}}\,t^{(l)}_j-\sum_{i=1}^l\,
\sum_{a\in I_i}\,z_a \Bigr)\>W_I\>+\>h\>\sum_{i=1}^l\sum_{j=l+1}^N\,
\sum_{m_1=1}^{\la_i}\<
\sum_{\satop{m_2=1}{\ell_{i,{m_1}}\!<\>\ell_{j,{m_2}}\!\!\!\!}}^{\la_j}
W_{I_{\ij;\>m_1\<,m_2}} &
\\[2pt]
\notag
{}=\,\sum_{i=1}^l\sum_{j=l+1}^N\,\sum_{a=1}^{\la_i}\,
\dch_{\{\ij\}} U_{(I)_{\ij}^{\<\>a}} & \,,
\\[-14pt]
\notag
\end{align}
To finish the proof of Theorem \ref{2thm}, we need to prove formula \Ref{44f}
for any \,$I$ and any \,$i=1\lc N-1$.

\vsk.2>
For any permutation \,$\si$\,, denote by \,$|\<\>\si|$ \,the length of
\,$\si$\,.
For any \,$J,J'\in\Il$\,, define the permutation \,$\si_{\JJ'}$ as follows:
if \,$J_m=\>\{\>j_{m,1}\<\lsym<j_{m,\>\la_m}\}$\,,
\,$J'_m=\>\{\>j'_{m,1}\<\lsym<j'_{m,\>\la_m}\}$\,, then
\,$\si_{\JJ'}(j'_{m,\>l})=j_{m,\>l}$\,. Set \,$\si_J=\si_{J,\<\>\Ima}$\,.
The permutation \,$\si_J$ has the minimal length amongst all permutations
\,$\si$ \,such that \,$\si(\Ima)=J$\,.

\begin{lem}
\label{sis}
Assume that for \,$J\in\Il$ \,and a transposition \,$s_{\ii+1}$, we have
\,${|\<\>s_{\ii+1}\>\si_J\<\>|<|\<\>\si_J\<\>|}$\,. Then
\,$s_{\ii+1}\>\si_J=\>\si_{s_{\ii+1}(J)}$\,.
\qed
\end{lem}

We will prove formula \Ref{44f} by induction with respect to the length of
\,$\si_I$\,. For the base of induction \,$I=\Ima$, formula \Ref{44f}
is proved already.

\vsk.2>
Fix \,$I\in\Il$ \,and find \,$m$ \,such that
\,$|\<\>s_{\mm+1}\>\si_I\<\>|<|\<\>\si_I\<\>|$\,. Let \,$p,r$ \,be such that
\,$m\in I_p$ and \,$m+1\in I_r$. Since
$|\<\>s_{\mm+1}\>\si_I\<\>|<|\<\>\si_I\<\>|$\,, we have \,$p<r$\,.

\vsk.2>
Denote \,$\It=s_{\mm+1}(I)$\,. Then \,$\It_p\<=I_p\<-\{m\}\cup\{m+1\}$\,,
\,$\It_r\<=I_r\<-\{m+1\}\cup\{m\}$\,, and \,$\It_c\<=I_c$\,, otherwise.
And clearly, \,$I=s_{\mm+1}(\It)$\,.

\vsk.2>
Write formula \Ref{44f} for \,$\It$:
\begin{align}
\label{44t}
\Bigl(\,\sum_{j=1}^{\la^{(l)}}\,t^{(l)}_j-\sum_{i=1}^l\,
\sum_{a\in\It_i}\,z_a \Bigr)\>W_{\It}\>+\>h\>\sum_{i=1}^l\sum_{j=l+1}^N\,
\sum_{m_1=1}^{\la_i}\<
\sum_{\satop{m_2=1}{\ellt_{i,{m_1}}\!<\>\ellt_{j,{m_2}}\!\!\!\!}}^{\la_j}
W_{\It_{\ij;\>m_1\<,m_2}} &
\\[2pt]
\notag
{}=\,\sum_{i=1}^l\sum_{j=l+1}^N\,\sum_{a=1}^{\la_i}\,
\dch_{\{\ij\}} U_{(\It)_{\ij}^{\<\>a}} & \,,
\end{align}
where \,$\It_c=(\ellt_{c,1}\lc\ellt_{c,\>\la_c})$. We will show that
applying the operator \,$S_{\mm+1}$ to both sides of formula~\Ref{44t}
transforms it to formula~\Ref{44f} for \,$I$\,.

\vsk.2>
To compare the right-hand sides, observe that
$\;s_{\mm+1}\bigl((\It)_{\ij}^{\<\>a}\bigr)=(I)_{\ij}^{\<\>a}$\,.
Hence, Lemma~\ref{3ctt} yields
$\;S_{\mm+1}\bigl(\dch_{\{\ij\}} U_{(\It)_{\ij}^{\<\>a}}\bigr)=
\dch_{\{\ij\}} U_{(I)_{\ij}^{\<\>a}}$\,, that proves the desired assertion.

\vsk.2>
To compare the left-hand sides, observe first that
\vvn.3>
\be
s_{\mm+1}(\It_{\ij;\>m_1\<,m_2})=I_{\ij;\>s_{\mm+1}(m_1),\>s_{\mm+1}(m_2)}
\vvn-.2>
\ee
and
\vv-.3>
\be
S_{\mm+1}(W_{\It_{\ij;\>m_1\<,m_2}})\,=\,
W_{I_{\ij;\>s_{\mm+1}(m_1),\>s_{\mm+1}(m_2)}}
\vv.4>
\ee
by Lemma~\ref{c3t}. This proves the desired transformation of the second sum
in the left-hand side of~\Ref{44t} term by term provided \,$p>l$ \,or
\,$r\le l$\,. If \,$p\le l<r$\,, the matching between the terms of the second
sums in~\Ref{44t} and~\Ref{44f} is not perfect and the sum in~\Ref{44f}
contains one more term \,$h\>W_{I_{p,r;\>\mm+1}}$\,.

\vsk.2>
If \,$p>l$ \,or \,$r\le l$\,, the sum \,$\sum_{i=1}^l\sum_{a\in\It_i}z_a$
\,in formula~\Ref{44t} \,is symmetric in \,$z_m\>,z_{m+1}$ \,and equals
the sum \,$\sum_{i=1}^l\sum_{a\in I_i}z_a$ \,in formula~\Ref{44f}. Thus
\begin{align*}
S_{\mm+1}\Bigl(\<\Bigl(\,\sum_{j=1}^{\la^{(l)}}\,t^{(l)}_j-\sum_{i=1}^l\,
\sum_{a\in\It_i}\,z_a \Bigr)\>W_{\It}\Bigr)\,&{}=\,
\Bigl(\,\sum_{j=1}^{\la^{(l)}}\,t^{(l)}_j-\sum_{i=1}^l\,
\sum_{a\in I_i}\,z_a \Bigr)\>S_{\mm+1}(W_{\It})
\\
&{}=\,\Bigl(\,\sum_{j=1}^{\la^{(l)}}\,t^{(l)}_j-\sum_{i=1}^l\,
\sum_{a\in I_i}\,z_a \Bigr)\>W_I
\end{align*}
by Lemma~\ref{c3t}. If \,$p\le l<r$\,, then we have
\begin{align*}
S_{\mm+1}\Bigl(\<\Bigl(\,\sum_{j=1}^{\la^{(l)}}\,t^{(l)}_j-\sum_{i=1}^l\,
\sum_{a\in\It_i}\,z_a \Bigr)\>W_{\It}\Bigr)\,&{}=\,
\Bigl(\,\sum_{j=1}^{\la^{(l)}}\,t^{(l)}_j-\sum_{i=1}^l\,
\sum_{a\in I_i}\,z_a \Bigr)\>S_{\mm+1}(W_{\It})\>+\>h\>W_{\It}
\\
&{}=\,\Bigl(\,\sum_{j=1}^{\la^{(l)}}\,t^{(l)}_j-\sum_{i=1}^l\,
\sum_{a\in I_i}\,z_a \Bigr)\>W_I\>+\>h\>W_{I_{p,r;\>\mm+1}}\,,
\end{align*}
since \,$\It=I_{p,r;\>\mm+1}$\,. This shows that the operator \,$S_{\mm+1}$
transforms formula~\Ref{44t} to formula~\Ref{44f}.
This completes the induction step. Theorem \ref{2thm} is proved.

\begin{example}
Let \,$N=2$\,, \,$n=3$\,, \,$\bla=(2,1)$\,, \,$I=(\{1,3\},\{2\})$\,,
\,$\Ima\<=(\{2,3\},\{1\})$\,,~\,$\si_I=s_{1,2}$\,.
Formula \Ref{44f} is
\vvn.3>
\beq
\label{44fe}
(t^{(1)}_1\<\<+\<\>t^{(1)}_2\<\<-z_1\<-z_3)\,W_{\{1,3\},\{2\}}\>+\>
h\>W_{\{2,3\}\{1\}}\,=\,\dch_{\{1,2\}}\>
(U_{\{1\},\{2,3\}}+\>U_{\{3\},\{1,2\}})\,,
\vv.4>
\eeq
formula \Ref{44t} is
\vvn.3>
\beq
\label{44te}
(t^{(1)}_1\<\<+\<\>t^{(1)}_2\<\<-z_2\<-z_3)\,W_{\{2,3\},\{2\}}\,=\,
\dch_{\{1,2\}}\>(U_{\{2\},\{1,3\}}+\>U_{\{3\},\{1,2\}})\,.
\vv.3>
\eeq
and the operator \,$S_{1,2}$ \,transforms formula \Ref{44te}
to formula \Ref{44fe}.
\end{example}

\section{Corollary of Theorems \ref{thm1} and \ref{2thm}}
\label{Corr}

Let \,$\bla\in\Z^N_{\geq 0}$\,, \,$|\bla|=n$\,, and \,$I\in\Il$\,.
Recall the notations \,$(I)^{\<\>a}_{\albt}$\,, \,$I_{\ij;\>m_1\<,m_2}$\,,
\,see \Ref{abla}, \Ref{Ikkmm}, and the discrete differentials
\,$d_{\{\albt\}}\>g$\,, see~\Ref{ddab}. Define the discrete differential
\vvn.3>
\beq
\label{Di}
D_{\Ii}\>=\>\sum_{j=i+1}^N\>\sum_{a=1}^{\la_j}\,d_{\{\ji\}}\>U_{(I)_{\ji}^{\<\>a}}
\<-\>\sum_{j=1}^{i-1}\,\sum_{a=1}^{\la_i}\,d_{\{\ij\}} U_{(I)_{\ij}^{\<\>a}}\>.
\vv.3>
\eeq

\begin{cor}
\label{corth}
We have
\vvn.3>
\begin{align}
\label{Dic}
\biggl(\>\sum_{j=1}^{\la^{(i)}}\,t^{(i)}_j-\<
\sum_{j=1}^{\la^{(i-1)}}\,t^{(i-1)}_j-\sum_{a\in I_i}\, z_a\biggr)\>W_I\>+
\,h\>\biggl(\>
\sum_{j=1}^{i-1}\,\frac{q_i\>\la_j}{q_i\<-q_j}\>+
\sum_{j=i+1}^N\,\frac{q_j\>\la_i}{q_i\<-q_j}\>\biggr)\>W_I\>&{}+{}
\\[2pt]
\notag
{}+\,h\,\sum_{\satop{j=1}{j\ne i}}^N\,\sum_{m_1=1}^{\la_i}\>\biggl(
\sum_{\satop{m_2=1}{\ell_{i,{m_1}}\!<\>\ell_{j,{m_2}}\!\!\!\!}}^{\la_j}
W_{I_{i,j;m_1,m_2}}+\,\frac{q_j}{q_i\<-q_j}\,
\sum_{m_2=1}^{\la_j}W_{I_{\ij;\>m_1\<,m_2}}\biggr){}&\>{}=\,D_{\Ii}\,.\kern-1em
\\[-14pt]
\notag
\end{align}
\end{cor}

\begin{proof}
Theorems \ref{thm1} and \ref{2thm}, and formula \Ref{dabg} imply that
\begin{align*}
\Bigl(\,\sum_{j=1}^{\la^{(i)}}\,t^{(i)}_j-\<
\sum_{j=1}^{\la^{(i-1)}}\,t^{(i-1)}_j-\sum_{a\in I_i}\, z_a\Bigr)\>W_I\>
&{}-\>h\>\sum_{j=1}^{i-1}\>\sum_{m_1=1}^{\la_i}\!
\sum_{\satop{m_2=1}{\ell_{i,{m_1}}\!>\>\ell_{j,{m_2}}\!\!\!\!}}^{\la_j}\!\!
W_{I_{\ij;\>m_1\<,m_2}}
\\
\notag
{}+\>h\>\sum_{j=i+1}^N\>\sum_{m_1=1}^{\la_i}\<
\sum_{\satop{m_2=1}{\ell_{i,{m_1}}\!<\>\ell_{j,{m_2}}\!\!\!\!}}^{\la_j}
W_{I_{i,j;m_1,m_2}}\>&{}+\>h\>\sum_{j=1}^{i-1}\,\frac{q_i}{q_i\<-q_j}\,
\Bigl(\>\la_j\>W_I+\sum_{m_1=1}^{\la_i}\>\sum_{m_2=1}^{\la_j}
W_{I_{\ij;\>m_1\<,m_2}}\Bigr)\kern-1em
\\[-2pt]
\notag
{}+\>h\>\sum_{j=i+1}^N \!{}&\,\frac{q_j}{q_i\<-q_j}\,\Bigl(\>\la_i\>W_I+
\sum_{m_1=1}^{\la_i}\>\sum_{m_2=1}^{\la_j}W_{I_{\ij;\>m_1\<,m_2}}\Bigr)\,=\,
D_{\Ii}\,.\kern-1em
\\[-14pt]
\notag
\end{align*}
Formula \Ref{Dic} is obtained now by rearranging the terms in the left-hand
side of this equality.
\end{proof}

Recall the scalar master function \,$\Phi_\bla(\TT;\zz;h;\qq)$
\,given by \Ref{PHI}. Define
\vvn.3>
\beq
\label{Om}
\Om_\bla(\qq)\,=\,
\prod_{i=1}^{N-1}\prod_{j=i+1}^N (1-q_j/q_i)^{h\<\>\la_i/\<\ka}\>.
\vv.2>
\eeq
Introduce the \,$\Cnnl$-valued weight function
\vvn-.2>
\beq
\label{vvwf}
W_\bla(\TT;\zz;h)\,=\,\sum_{I\in\Il}\,W_I(\TT;\zz;h)\,v_I\,.
\vv.2>
\eeq
Recall the dynamical Hamiltonians \,$X_i(\zz;h;\qq)$ defined in \Ref{Xi}.

\begin{thm}
\label{thm main}
For every $\;i=1\lc N$, we have
\begin{align}
\label{ma id}
\Bigl(\ka\>q_i\>\frac{\der}{\der q_i}\>-X_i(\zz;h;\qq)\Bigr)\,{}
& \Om_\bla(\qq)\,\Phi_\bla(\TT;\zz;h;\qq)\>W_\bla(\TT;\zz)\,={}
\\[3pt]
\notag
{}=\,{}& \Om_\bla(\qq)\,\Phi_\bla(\TT;\zz;h;\qq)\>
\sum_{I\in\Il}\>D_{\Ii}(\TT;\zz;h;\qq)\>v_I\,.
\\[-20pt]
\notag
\end{align}
\end{thm}
\begin{proof}
The statement is equivalent to Corollary \ref{corth}.
\end{proof}

\section{Integral representations for solutions of dynamical equations}
\label{sirfs}
\subsection{Formal integrals}
Let \,$\bla\in\Z^N_{\geq 0}$\,, \,$|\bla|=n$\,, and \,$\ka\in\Cxs$.
Consider the space of functions of the form
\,$\Phi_\bla(\TT;\zz;h;\qq)\>f(\TT;\zz;h;\qq)$\,, where
\,$\Phi_\bla(\TT;\zz;h;\qq)$ \,is the master function~\Ref{PHI}, and
\,$f(\TT;\zz;h;\qq)$ \,is a polynomial in \,$\TT$ \,and holomorphic function
of \,$\zz,h,\qq$ on some domain \,$L\<\subset\<\C^n\!\times\C\times\C^N$.
Assume that we have a map \,$\Mc$ \,assigning to a function \,$\Phi_\bla f$
a function \,$\Mc(\Phi_\bla f)$ \,of variables \,$\zz,h,\qq$\,, holomorphic
on \,$L$\,, such that:
\begin{enumerate}
\item[(i)]
The map \,$\Mc$ \,is linear over the field of meromorphic on \,$L$ \,functions
in \,$\zz,\qq$,h\,,
\vvn.3>
\beq
\label{fi1}
\Mc\bigl(\Phi_\bla\>(g_1f_1+g_2f_2)\bigr)\,=\,
g_1\>\Mc(\Phi_\bla f_1)+g_2\>\Mc(\Phi_\bla f_2)
\vv.2>
\eeq
for any meromorphic functions \,$g_1,g_2$ \>of \,$\zz,h,\qq$\,, such that
\,$g_1f_1$ \,and \,$g_2f_2$ \,are holomorphic on \,$L$\,.

\vsk.3>
\item[(ii)] For any \,$i=1\lc N$, we have
\vvn.2>
\beq
\label{fi2}
\frac{\der}{\der q_i}\>\Mc(\Phi_\bla f)\,=\,
\Mc\left(\frac{\der}{\der q_i} (\Phi_\bla f)\right).
\eeq

\vsk.3>
\item[(iii)] If \,$f$ \,is a discrete differential of a polynomial
in \,$\TT$\,, then
\vvn.3>
\beq
\label{fi3}
\Mc(\Phi_\bla f)\,=\,0\,.
\eeq
\end{enumerate}

\vsk.2>
A map \,$\Mc$ \,is called a formal integral.
We have the following corollary of Theorem \ref{thm main}.

\begin{lem}
\label{lem foi}
If \,$\Mc$ \,is a formal integral, then the $\Cnnl$-valued
function
\vvn.3>
\be
F_{\Mc}(\zz;h;\qq)\,:=\,\Om_\bla\,\Mc(\Phi_\bla W_\bla)\,=\,
\Om_\bla\>\sum_{I\in\Il}\Mc(\Phi_\bla W_I)\>v_I
\ee
holomorphic on \,$L$\>, is a solution of the dynamical differential equations
\>\Ref{DEQ}.
\qed
\end{lem}

\subsection{Jackson integral}
\label{secJ}
Consider the space \,$\C^{\la^{\{1\}}}\!\<\<\times\C^n\!\times\C\times\C^N$
with coordinates \,$\TT,\zz,h,\qq$\,. The lattice \,$\ka\>\Z^{\la^{\{1\}}}\<$
naturally acts on this space by shifting the \,$\TT\<\>$-coordinates.

\vsk.2>
Let \,$J=(J_1\lc J_N)\in\Il$\,. Recall the notation
\;$\bigcup_{i=1}^{\,k}J_i=\>\{\>j^{(k)}_1\!\lsym<j^{(k)}_{\la^{(k)}}\}$\,.
Define \,$\Si_J\subset\C^{\la^{\{1\}}}\!\<\<\times\C^n\!\times\C\times\C^N$
by the equations:
\vvn.2>
\beq
\label{cy}
t^{(k)}_i=\>z_{j^{(k)}_i}, \qquad k=1\lc N-1\,,\quad i=1\lc\la^{(k)}\>,
\vv.1>
\eeq
and call it a {\it discrete cycle}.

\vsk.2>
For a function of \,$\TT$ \,and a point \,$\bss\<\in\C^{\la^{\{1\}}}\!$,
define $\;\Res_{\>\TT=\bss}$ to be the iterated residue,
\vvn.2>
\be
\Res_{\>\TT=\bss}\>=\,\Res_{t^{(1)}_1\<=\>s^{(1)}_1}\ldots\>
\Res_{t^{(1)}_{\la^{(1)}}\<=\>s^{(1)}_{\la^{(1)}}}\ldots\;
\Res_{t^{(N-1)}_1\<=\>s^{(N-1)}_1}\ldots\>
\Res_{t^{(N-1)}_{\la^{(N-1)}}\<=\>s^{(N-1)}_{\la^{(N-1)}}}\,.
\ee

\vsk.2>
Let \,$L'$ \>be the complement in \,$\C^n\!\times\C$ \,of the union of
the hyperplanes
\vvn.3>
\beq
\label{zzh}
h\>=\>m\<\>\ka\,,\qquad z_a\<-z_b\>=\>m\<\>\ka\,,\qquad
z_a\<-z_b\<+h\>=\>m\<\>\ka\,,
\vv.3>
\eeq
for all \,$a,b=1\lc n$\,, \,$a\ne b$\,, and all \,$m\in\Z$\,.
Let \,$L''\!\subset\C^N$ be the domain
\vvn.3>
\beq
\label{q/q}
|\>q_{i+1}/q_i|<1\,,\qquad i=1\lc N-1\,,
\vv.3>
\eeq
with additional cuts fixing a branch of $\;\log\>q_i$ \,for all \,$i=1\lc N$.
Set \,$L=L'\!\times L''\!\subset\<\C^n\!\times\C\times\C^N$.

\goodbreak
\vsk.2>
Let \,$f(\TT;\zz;h;\qq)$ be a polynomial in \,$\TT$ and a holomorphic function
of \,$\zz;h;\qq$ \,on \,$L$\,. For \,$(\zz;h;\qq)\in L$\,, define
\beq
\label{MSi}
\Mc_J(\Phi_\bla f)(\zz;h;\qq)\,=\!\sum_{\rr\<\in\Z^{\la^{\{1\}}}\!\!\!\!}
\Res_{\>\TT\>=\>\Si_J+\<\>\rr\ka\>}\bigl(\Phi_\bla(\TT;\zz;h;\qq)\>f(\TT;\zz;h;\qq)
\bigr)\,.
\eeq
This sum is called the {\it Jackson integral over the discrete cycle\/}
\,$\Si_J$\,.

\goodbreak
\begin{lem}
\label{Mji}
The map \,$\Mc_J$ is a formal integral.
\end{lem}
\begin{proof}
Each term of the sum in formula~\Ref{MSi} is a holomorphic function on \,$L'$.
\vvn.1>
Moreover, \,$\Res_{\>\TT\>=\>\Si_J+\<\>\rr\ka}
\bigl(\Phi_\bla(\TT;\zz;h;\qq)\>f(\TT;\zz;h;\qq)=0$
\,if \,$\rr\<\not\in\Z_{\leq 0}^{\la^{\{1\}}}$.
Hence, the sum over \,$\Z^{\la^{\{1\}}}\<$ reduces to the sum over
\,$\Z_{\leq 0}^{\la^{\{1\}}}\<$. The result is similar to a multidimensional
hypergeometric series multiplied by some fractional powers of \,$q_1\lc q_N$\,.
The obtained sum converges if \,$|\>q_{i+1}/q_i|<1$ \,for all \,$i=1\lc N-1$\,,
and gives a holomorphic function on \,$L$\,.
\vsk.2>
Properties \Ref{fi1}\,--\,\Ref{fi3} for the map \,$\Mc_J$ are clear.
Lemma \ref{Mji} is proved.
\end{proof}

\begin{lem}
\label{Mji+}
The function \,$\Mc_J(\Phi_\bla f)$ analytically continues to the hyperplanes
$\;h=m\<\>\ka$ \,for \,$m\in\Z_{>0}$\,.
\end{lem}
\begin{proof}
By the proof of Lemma \ref{Mji},
\vvn.3>
\beq
\label{MSih}
\Mc_J(\Phi_\bla f)(\zz;h;\qq)\,=\!\sum_{\rr\<\in\Z_{\leq0}^{\la^{\{1\}}}\!\!\!\!}
\Res_{\>\TT\>=\>\Si_J+\<\>\rr\ka\>}\bigl(\Phi_\bla(\TT;\zz;h;\qq)\>
f(\TT;\zz;h;\qq)\bigr)\,.\kern-1.6em
\vv-.2>
\eeq
for \,$(\zz;h;\qq)\in L$\,. By inspection, if $\;h\to m\<\>\ka$\,,
\,$m\in\Z_{>0}$\,, and \,$\rr\<\in\Z_{\leq0}^{\la^{\{1\}}}$, then
\vvn.4>
\be
\Res_{\>\TT\>=\>\Si_J+\<\>\rr\ka\>}\bigl(\Phi_\bla(\TT;\zz;h;\qq)\>
f(\TT;\zz;h;\qq)\bigr)\,\to\,
\Res_{\>\TT\>=\>\Si_J+\<\>\rr\ka\>}\bigl(\Phi_\bla(\TT;\zz;m\<\>\ka;\qq)\>
f(\TT;\zz;m\<\>\ka;\qq)\bigr)\,.
\vv.3>
\ee
Hence
\beq
\label{MSim}
\Mc_J(\Phi_\bla f)(\zz;h;\qq)\,\to\!
\sum_{\rr\<\in\Z_{\leq0}^{\la^{\{1\}}}\!\!\!\!}
\Res_{\>\TT\>=\>\Si_J+\<\>\rr\ka\>}\bigl(
\Phi_\bla(\TT;\zz;m\<\>\ka;\qq)\>f(\TT;\zz;m\<\>\ka;\qq)\bigr)\,,\kern-1.6em
\eeq
since the sum in the right-hand side converges if \,$|\>q_{i+1}/q_i|<1$
\,for all \,$i=1\lc N-1$\,.
\end{proof}

\begin{rem}
For \,${m\in\Z_{>0}}$\,, the sum
\vvn.1>
\,$\sum_{\rr\<\in\Z^{\la^{\{1\}}}}
\Res_{\>\TT\>=\>\Si_J+\<\>\rr\ka\>}\bigl(\Phi_\bla(\TT;\zz;m\<\>\ka;\qq)\>
f(\TT;\zz;m\<\>\ka;\qq)\bigr)$ diverges, and the function
\,$\Mc_J(\Phi_\bla f)(\zz;m\<\>\ka;\qq)$ \,is not given by formula \Ref{MSi}.
\end{rem}

\begin{example}
Let \,$N=2$\,, \,$\bla=(1,n-1)$\,, \,$J=(\{1\},\{2,3\lc n\})$\,. Then
\vvn.3>
\be
\Phi_\bla(\TT;\zz;h;\qq)\,=\,
(e^{\<\>\pii}\,q_2)^{\>\sum_{a=1}^nz_a\</\ka}
\Bigl(\>e^{\<\>\pii\,(n-2)}\>\frac{q_1}{q_2}\Bigr)^{t^{(1)}_1\!/\ka}
\,\prod_{a=1}^n\,\Ga\Bigl(\frac{t_1^{(1)}\<\<-z_a}\ka\Bigr)\,
\Ga\Bigl(\frac{z_a\<-t_1^{(1)}\<\<+h}\ka\Bigr)\,,
\ee
and
\beq
\label{MSe}
\Mc_{(\{1\},\{2,3\lc n\})}(\Phi_\bla\>f)\,=\,
\sum_{r\in\Z}\,\Res_{t^{(1)}_1=\>z_1+\>r\ka\>}(\Phi_\bla\>f)\,.
\eeq
Nonzero contributions to the sum in the right-hand side of \Ref{MSe}
come from the poles of \,$\Ga\bigl((t_1^{(1)}\<\<-z_1)/\ka\bigr)$\,.
Explicitly, the answer is
\vv.3>
\begin{align*}
\Mc_{(\{1\},\{2,3\lc n\})}( & \Phi_\bla\>f)\,=\,
(\>e^{\<\>\pii\,(n-1)}q_1)^{z_1\</\ka}\,
(e^{\<\>\pii}\,q_2)^{\>\sum_{a=2}^nz_a\</\ka}\>\times{}
\\[6pt]
{}\times\,{} &\ka\;\Ga\Bigl(\>\frac h\ka\>\Bigr)\,
\prod_{a=2}^n\,\Ga\Bigl(\frac{z_1\<-z_a}\ka\Bigr)\,
\Ga\Bigl(\frac{z_a\<-z_1\<+h}\ka\Bigr)\>\times{}
\\[2pt]
{}\times\,{}&\!\sum_{l=0}^\infty\,f(z_1\<-l\ka\>;\zz;h;\qq)\,
\Bigl(\>\frac{q_2}{q_1}\<\>\Bigr)^{\!l}\;
\prod_{j=0}^{l-1}\>\biggl(\>\frac{h+j\<\>\ka}{\ka+j\<\>\ka}\;
\prod_{a=2}^n\,\frac{z_a\<-z_1\<+h+j\<\>\ka}{z_a\<-z_1\<+\ka+j\<\>\ka}\>
\biggr)\>,
\\[-14pt]
\end{align*}
and the series converges if \,$|\>q_2/q_1|<1$\,.
\end{example}

\subsection{Solutions of dynamical equations}
Recall the \,$\Cnnl$-valued weight function
\,$W_\bla(\TT;\zz)$, given by \Ref{vvwf}. For \,$J\in\Il$\,, define
\vvn.3>
\beq
\label{mcF}
\Psi_J(\zz;h;\qq)\,=\,\Om_\bla(\qq)\,\Mc_J(\Phi_\bla W_\bla)(\zz;h;\qq)\,=\,
\Om_\bla(\qq)\,\sum_{I\in\Il}\,\Mc_J (\Phi_\bla W_I)(\zz;h;\qq)\,v_I\,.
\kern-1em
\eeq

\begin{thm}
\label{thm cy}
The function \,$\Psi_J(\zz;h;\qq)$ is a holomorphic
\vvn.1>
\,$\Cnnl$-valued function of $\;\zz,h,\qq$
\,on the domain $\;L\subset\<\C^n\!\times\C\times\C^N\<$ such that
\vvn.2>
\be
h\not\in\ka\>\Z_{\le0}\,,\qquad
z_a\<-z_b\<\not\in\ka\>\Z\,,\qquad z_a\<-z_b\<+h\not\in\ka\>\Z\,,
\vv.3>
\ee
for all \,$a,b=1\lc n$\,, \,$a\ne b$\,,
\vvn.1>
\be
|\>q_{i+1}/q_i|<1\,,\qquad i=1\lc N-1\,,
\vv.2>
\ee
and a branch of \,\,$\log\>q_i$ is fixed for each \,$i=1\lc N$. Furthermore,
\,$\Psi_J(\zz;h;\qq)$ \,is a solution of the dynamical differential equations
\Ref{DEQ}.
\end{thm}
\begin{proof}
The weight functions $W_I(\TT;\zz;h)$ are polynomials in \,$\TT,\zz,h$
\,and do not depend on \,$\qq$\,. Hence, Theorem \ref{thm cy} follows from
Lemmas \ref{Mji}, \ref{Mji+}, and \ref{lem foi}.
\end{proof}

\begin{thm}
\label{thm ba}
Under conditions of Theorem \ref{thm cy}, the collection of
\vvn.2>
\,$\Cnnl$-valued functions
\,$\bigl(\Psi_J(\zz;h;\qq)\bigr)_{J\in\Il}$ \,is a basis of solutions of
the dynamical equations \Ref{DEQ}.
\end{thm}
\begin{proof}
By formulas \Ref{MSih}, \Ref{mcF}, if \,$|\>q_{i+1}/q_i|\to 0$ \,for all
\,$i=1\lc N-1$\,, then
\vvn.4>
\begin{align}
\label{Fq0}
\Psi_J(\zz;h;\qq)\,\simeq\,{}&\Om_\bla(\qq)\,\Res_{\>\TT\>=\>\Si_J}
\bigl(\Phi_\bla(\TT;\zz;h;\qq)\bigr)\,W_\bla(\Si_J;\zz;h)\,={}
\\[8pt]
\notag
{}=\,{}&\Om_\bla(\qq)\,\Res_{\>\TT\>=\>\Si_J}
\bigl(\Phi_\bla(\TT;\zz;h;\qq)\bigr)\,
\sum_{I\in\Il}W_I(\Si_J;\zz;h)\,v_I\,.\kern-1.4em
\\[-18pt]
\notag
\end{align}
By \cite[Lemma 3.1]{RTV}, the matrix
\,$\bigl(W_I(\Si_J;\zz;h)\bigr)_{\IJ\in\Il}$ is triangular and the diagonal
entries \,$W_I(\Si_I;\zz;h)$ \,are nonzero if \,$h\ne 0$ \,and
\,$z_a\<-z_b\<\ne 0$\,, \,$z_a\<-z_b\<+h\ne 0$\,, for all \,$a,b=1\lc n$\,,
\,$a\ne b$\,. Hence the vectors \,$W_\bla(\Si_J)$\,, \,$J\in\Il$\,,
form a basis of \,$\Cnnl$ and the collection
\,$\bigl(\Psi_J(\zz;h;\qq)\bigr)_{J\in\Il}$ \,is a basis of solutions
of the dynamical equations \Ref{DEQ}.
\end{proof}

The functions \,$\Psi_J(\zz;h;\qq)$ \,were considered in \cite{TV1}.
It follows from \cite[Theorem 1.5.2]{TV1}, cf.~\cite{TV4}, that for every
\,$J\in\Il$\,, the function \,$\Psi_J(\zz;h;\qq)$ \,is a solution of the \qKZ/
equations \Ref{K_i}.

\begin{cor}
\label{cor KZdy}
The collection of \,$\Cnnl$-valued functions
\,$\bigl(\Psi_J(\zz;h;\qq)\bigr)_{J\in\Il}$ is a basis of
solutions of both the dynamical and \qKZ/ equations,
see \Ref{DEQ}, \Ref{K_i}, with values in \,$\Cnnl$\,.
\end{cor}

\begin{rem}
The functions $\Psi_J(\zz;h;\qq)$ \,are called the {\it multidimensional
\,$q$-hypergeometric solutions} of the dynamical equations.
In \cite{TV5}, we constructed another type of solutions of the dynamical
equations called the {\it multidimensional hypergeometric solutions}.
\end{rem}

\section{Equivariant quantum differential equations}
\label{sQde}

\subsection{Partial flag varieties}
\label{sec Partial flag varieties}

Let \,$\bla\in\Z^N_{\ge 0}$\,, \,$|\bla|=n$\,.
Consider the partial flag variety
\;$\Fla$ parametrizing chains of subspaces
\be
0\,=\,F_0\subset F_1\lsym\subset F_N =\,\C^n
\ee
with \;$\dim F_i/F_{i-1}=\la_i$, \;$i=1\lc N$.
Denote by \,$\tfl$ the cotangent bundle of \;$\Fla$ and
\be
\XX_n\>=\,\bigcup_{|\bla|=n}\tfl\,.
\ee

Let $u_1\lc u_n$ be the standard basis of $\C^n$. For \,$I\in\Il$\,,
let \,$x_I\in\Fla$ \,be the point corresponding to the coordinate flag
\,$F_1\<\lsym\subset F_N$\,, where \,$F_i$ \,is the span of the standard basis
vectors \;$u_j\in\C^n$ with \,$j\in I_1\<\lsym\cup I_i$\,. We embed \,$\Fla$
\,in \,$\tfl$ \,as the zero section and consider the points \,$x_I$ \,as points
of $\tfl$\,.

\subsection{Equivariant cohomology}
\label{sec:equiv}

Let \,$A\subset GL_n(\C)$ \,be the torus of diagonal matrices and
\,$T=A\times\Cxs$. The group \,$A$ \,acts on \;$\C^n\<$ and hence
on \,$\tfl$\,. Let the group \;$\Cxs\<$ act on \,$\tfl$ \,by multiplication
in each fiber. We denote by \,$-\>h$ \,its \,$\Cxs\<$-weight.

\vsk.2>
We consider the equivariant cohomology algebras $H^*_T(\tfl\>;\C)$ and
\vvn.2>
\be
H_T^*(\XX_n)\,=\>\bigoplus_{|\la|=n}\>H^*_T(\tfl\>;\C)\,.
\ee
Denote by \;$\GG_i=\{\ga_{i,1}\lc\ga_{i,\la_i}\}$ \,the set of the Chern
roots of the bundle over \,$\Fla$ \,with fiber \,$F_i/F_{i-1}$\,.
Let \;$\GG=(\GG_1\<\>\lsym;\GG_N)$\,. Denote by \,$\zb=\{\zzz\}$
\,the Chern roots corresponding to the factors of the torus \,$A$\,. Then
\beq
\label{Hrel}
H^*_T(\tfl)\,=\,\C[\GG]^{\>\Sla}\<\<\otimes\C[\zz]\otimes\C[h]\>\Big/\Bigl\bra
\,\prod_{i=1}^N\prod_{j=1}^{\la_i}\,(u-\ga_{\ij})\,=\,\prod_{a=1}^n\,(u-z_a)
\Bigr\ket\,.
\vv.2>
\eeq
The cohomology \,$H^*_T(\tfl)$ \,is a module over
\,$H^*_T({pt};\C)=\C[\zz]\otimes\C[h]$\,.

\vsk.2>
Notice that
\vvn-.3>
\beq
\label{ch cl}
\prod_{i=1}^{N-1}\prod_{j=i+1}^N\>\prod_{a=1}^{\la_i}\,
\prod_{b=1}^{\la_j}\,(\ga_{j,b}\<-\ga_{i,a})\,(\ga_{i,a}\<-\ga_{j,b}\<-h)
\vv.3>
\eeq
is the equivariant total Chern class of the tangent bundle of \,$\tfl$ \,and
\vvn.2>
\beq
\label{c1}
c_1(E_i)\,=\,\sum_{a=1}^{\la_i}\,\ga_{i,a}\,, \qquad i=1\lc N\>,
\eeq
is the equivariant first Chern class of the vector bundle \,$E_i$ \,over
\,$\tfl$ \,with fiber \,$F_i/F_{i-1}$\,.

\vsk.2>
For \,$i=1\lc N$, denote \,$\Tht_i\<=\{\tht_{i,1}\lc\tht_{i,\la^{(i)}}\}$
\,the Chern roots of the bundle \,$\FF_i$ \,over \,$\Fla$ \,with fiber
\,$F_i$\,. Let \,$\Thb=(\Tht_1\lc\Tht_{N})$\,. The relations
\vvn.2>
\beq
\label{rE}
\prod_{a=1}^{\la^{(i)}}\,(u-\tht_{i,\<\>a})\,=\,
\prod_{j=1}^i\,\prod_{k=1}^{\la_j}\,(u-\ga_{\jk})\,, \qquad i=1\lc N\>,
\vv.3>
\eeq
define the homomorphism
\vvn.3>
\be
\C[\Thb]^{S_{\la^{(1)}}\<\lsym\times S_{\la^{(N)}}}\!
\otimes \C[\zz]\otimes\C[h]\,\to\,H^*_T(\tfl)\,.
\vv.2>
\ee

\subsection{Stable envelope map}
Recall the weight functions \,$\WW_I$ \,defined in Sections \ref{dwf}.
Let \,$\WW_I(\Thb;\zz)\in H^*_T(\tfl)$ \,be the cohomology class
represented by the polynomial \,$\WW_{\id,I}(\TT;\zz)$ \,with the variables
$\;t^{(i)}_a\<$ replaced by $\;\tht_{i,\<\>a}$ \>for all \,$i=1\lc N-1$\,,
\,$a=1\lc\la^{(i)}$. Denote
\vvn.1>
\be
c_\bla(\Thb)\,=\, \prod_{i=1}^{N-1}\,\prod_{a=1}^{\la^{(i)}}\,
\prod_{b=1}^{\la^{(i)}}\,(\tht_{i,\<\>a}\<-\tht_{i,\<\>b}\<-h)
\in H^*_T(\tfl)\,.
\ee
Observe that $c_\bla(\Thb)$ is the equivariant Euler class of the bundle
\,$\bigoplus_{a=1}^{N-1}\>\Hom(\FF_a,\FF_a)$ \,if we make \,$\Cxs\!$ act on it
with weight \,$-\>h$\,.

\begin{thm}[{\cite[Theorem 4.1]{RTV}}]
For any \,$\bla$ \>and any \,$I\in\Il$\,, the cohomology class
$\WW_I(\Thb;\zz)\in H^*_T(\tfl)$ \,is divisible by \,$c_\bla(\Thb)$\,, that
is, there exists a unique element \,$\St_{\>I}\in H^*_T(\tfl)$ \,such that
\vvn-.3>
\beq
\label{StW}
[\WW_I(\Thb;\zz)]\,=\,c_\bla(\Thb)\cdot\>\St_{\>I}\,.
\vv-1.4>
\eeq
\vv.7>
\qed
\end{thm}

Define the {\it stable envelope map} by the rule
\vvn.3>
\beq
\label{stmap}
\St\>:\>\Cnn\<\<\otimes\C[\zz]\otimes\C[h]\,\to\,H_T^*(\XX_n)\,,\qquad
v_I\>\mapsto\>\St_{\>I}\,.\kern-1em
\vv.2>
\eeq

\begin{rem}
Stable envelope maps for Nakajima quiver varieties were introduced in
\cite{MO}. They were defined there geometrically in terms of the associated
torus action. The map $\;\St$ \,given by formula \Ref{stmap} is the stable
envelope map of \cite{MO} for the Nakajima quiver variety \,$\XX_n$\,,
described in terms of the Chern roots \,$\Thb,\zz,h$, see \cite{RTV}.
\end{rem}

\begin{rem}
After the substitution \,$h=1$\, the classes \,$\St_{\>I}\in H_T^*(\tfl)$
\,can be considered as elements of the equivariant cohomology
\,$H^*_{(\Cxs)^n}(\Fla)$ \,of the partial flag variety \,$\Fla$ (and not of
the cotangent bundle $\tfl$)\<\>. These new classes are the equivariant
Chern-Schwartz-MacPherson classes (CSM classes) of the corresponding Schubert
cells, see \cite{RV}.
\end{rem}

Let \,$\C(\zz;h)$ \,be the algebra of rational functions in \,$\zz,h$\,.
The map
\vvn.2>
\beq
\label{Stab}
\St\>:\>\Cnn\<\otimes\C(\zz;h)\,\to\,H_T^*(\XX_n)\,\ox\,\C(\zz;h)\,,\qquad
v_I\>\mapsto\>\St_{\>I}\,,\kern-1em
\vv.2>
\eeq
is an isomorphism of \,$\C(\zz;h)$-modules by \cite[Lemma 6.7]{RTV}.

\subsection{$H_T^*(\tfl)\<\>$-valued weight function}
Define the $H_T^*(\tfl)$-valued function \,$\Wh(\TT;\GG)$ as follows:
\beq
\label{Wh}
\Wh(\TT;\GG)\,=\,
\Sym_{\>t^{(1)}_1\!\lc\,t^{(1)}_{\la^{(1)}}}\,\ldots\;
\Sym_{\>t^{(N-1)}_1\!\lc\,t^{(N-1)}_{\la^{(N-1)}}}\Uh(\TT;\GG)\,,
\vv.4>
\eeq
\be
\Uh(\TT;\GG)\,=\,\prod_{j=1}^{N-1}\,\prod_{a=1}^{\la^{(j)}}\,\biggl(\;
\prod_{c=1}^{a-1}\;(t^{(j)}_a\<\<-t^{(j+1)}_c-h)
\prod_{d=a+1}^{\la^{(j+1)}}(t^{(j)}_a\<\<-t^{(j+1)}_d )\,
\prod_{b=a+1}^{\la^{(j)}}
\frac{t^{(j)}_a\<\<-t^{(j)}_b\<\<-h}{t^{(j)}_a\<\<-t^{(j)}_b}\,\biggr)\,.
\kern-2em
\vv.4>
\ee
where \,$(t^{(N)}_1\!\lc t^{(N)}_n)\,=\,(\ga_{1,1}\lc\ga_{1,\>\la_1},
\ga_{2,1}\lc\ga_{2,\>\la_2},\,\ldots\,,\ga_{N\<,\>1}\lc\ga_{N\<,\>\la_N})$\,,
\vvn.1>
cf.~formula \Ref{hWI-} for
\,$I=\Imi\<=\bigl(\{1\lc\la_1\}\lc\{n-\la_N\<+1\lc n\}\bigr)$\,.

\vsk.2>
\begin{example}
Let $N=2$, $\bla=(1,n-1)$.
Then \,$\Wh(\TT;\GG)=\prod_{a=1}^{n-1}\,(t^{(1)}_1\<\<-\ga_{2,a})$\,.\\[3pt]
Let $N=3$, $\bla=(1,1,1)$. Then
\begin{align*}
\Wh(\TT;\GG)\,&{}=\,(t^{(1)}_1\<\<-t^{(2)}_2)\>
(t^{(2)}_1\<\<-\ga_{2,1})\>(t^{(2)}_1\<\<-\ga_{3,1})\,
(t^{(2)}_2\<\<-\ga_{1,1}\<-h)\>(t^{(2)}_2\<\<-\ga_{3,1})\>
\frac{t^{(2)}_1\<\<-t^{(2)}_2\<\<-h}{t^{(2)}_1\<\<-t^{(2)}_2}\,+{}
\\[3pt]
&{}\>+\,
(t^{(1)}_1\<\<-t^{(2)}_1)\>
(t^{(2)}_2\<\<-\ga_{2,1})\>(t^{(2)}_2\<\<-\ga_{3,1})\,
(t^{(2)}_1\<\<-\ga_{1,1}\<-h)\>(t^{(2)}_1\<\<-\ga_{3,1})\>
\frac{t^{(2)}_2\<\<-t^{(2)}_1\<\<-h}{t^{(2)}_2\<\<-t^{(2)}_1}\;.
\\[-20pt]
\end{align*}
\end{example}

Define
\vvn-.7>
\beq
\label{QG}
Q(\GG)\,=\,\prod_{i=1}^{N-1}\prod_{j=i+1}^N\,\prod_{a=1}^{\la_i}\,
\prod_{b=1}^{\la_j}\,(\ga_{i,a}\<-\ga_{j,b}\<-h)\,\in H_T^*(\tfl)\,.\kern-1em
\vv.4>
\eeq

The image of the \,$\Cnnl$-valued weight function
\,$W_\bla(\TT;\zz)$\,, see~\Ref{vvwf}, is given by the next proposition.

\begin{prop}
\label{p111}
We have
\vvn.4>
\beq
\label{WSt}
\sum_{I\in\Il}\,W_I(\TT;\zz)\;\St_{\>\id,\>I}\,=\,Q(\GG)\,\Wh(\TT;\GG)\,.
\eeq
\end{prop}
\begin{proof}
Recall that \,$W_I(\TT;\zz)=
(-\>h)^{-\la^{\{1\}}}\>\WW_{\si_0\<\>,\>I}(\TT;\zz)$\,, see \Ref{WW}.
Recall the discrete cycle \,$\Si_J$ given by \Ref{cy}.
Let \,$\si^I\!\in S_n$ be a permutation such that \,$\si^I(\Imi)=I$\,.
Then formula \Ref{WSt} is equivalent to the following equality
\vvn.3>
\beq
\label{WWW}
\sum_{I\in\Il}\>\WW_{\si_0\<\>,\>I}(\TT;\zz)\;\WW_I(\Si_J;\zz)\,=\,
c_\bla(\Si_J)\,Q(\zz_J)\,\WW_{\si^J\!\<\<,\>J}(\TT;\zz)\,.\kern-1em
\vv.1>
\eeq
For the proof of formula \Ref{WWW}, consider the function
\vvn.3>
\be
Z(\TT;\TTT;\zz)\,=\,\sum_{I\in\Il}\>
\WW_{\si_0\<\>,\>I}(\TT;\zz)\;\WW_I(\TTT;\zz)\,.
\ee
Here \,$\TTT$ \,is an additional set of variables similar to \,$\TT$\,.
Then formula \Ref{WWW} reads
\vvn.4>
\beq
\label{WWZ}
Z(\TT;\Si_J;\zz)\,=\,
c_\bla(\Si_J)\,Q(\zz_J)\,\WW_{\si^J\!\<\<,\>J}(\TT;\zz)\,.
\vv.5>
\eeq
Three-term relations \Ref{WW3} imply that for any \,$\si\in S_n$\,, we have
\vvn.5>
\beq
\label{Hsi}
Z(\TT;\TTT;\zz)\,=\,\sum_{I\in\Il}\,
\WW_{\si,\>I}(\TT;\zz)\;\WW_{\si\si_0\<\>,\>I}(\TTT;\zz)\,.
\vv.2>
\eeq
By \cite[Lemma~3.2]{RTV}, we have
\,$\WW_{\si^J\!\si_0\<\>,\>I}(\Si_J,\zz)=c_\bla(z_J)\,Q(\zz_J)\,\dl_{\IJ}$\,.
Thus taking \,$\si=\si^J$, $\;\TTT=\Si_J$ \,in formula \Ref{Hsi}\,,
we get equality \Ref{WWZ}. Proposition \ref{p111} is proved.
\end{proof}

Define the cohomology classes
\vvn-.2>
\beq
\label{RG}
R(\GG)\,=\,\prod_{i=1}^{N-1}\prod_{j=i+1}^N\,\prod_{a=1}^{\la_i}\,
\prod_{b=1}^{\la_j}\,(\ga_{i,a}\<-\ga_{j,b})
\vv-.4>
\eeq
and
\vvn-.3>
\beq
\label{RGK}
R_I(\GG;\zz)\,=\,\prod_{i=1}^{N-1}\prod_{j=i+1}^N\,\prod_{a=1}^{\la_i}\>
\prod_{\,b\in I_j\!}\,(\ga_{i,a}\<-z_b)\,,\qquad I\<\in\Il\,.\kern-2em
\eeq
Notice that \,$R_I(\zz_J;\zz)=R(\zz_J)\,\dl_{\IJ}$\,.

\begin{prop}
\label{pdel}
For any \,$K\<\<\in\Il$\,, we have
\vvn.4>
\beq
\label{WDl}
\sum_{I\in\Il}\,W_I(\Si_K;\zz)\;\St_{\>\id,\>I}\,=\,
(-\>h)^{-\la^{\{1\}}}\<c_\bla(\Thb)\,R_K(\GG;\zz)\,Q(\GG)\,.\kern-1em
\vv.2>
\eeq
\end{prop}
\begin{proof}
Formula \Ref{WDl} is equivalent to the equality
\vvn.4>
\beq
\label{WWDl}
\sum_{I\in\Il}\,\WW_{\si_0\<\>,\>I}(\Si_K;\zz)\;\WW_I(\Si_J;\zz)\,=\,
\bigl(c_\bla(\Si_J)\bigr)^2\,R(\zz_J)\,Q(\zz_J)\,\dl_{\JK}\,.\kern-.6em
\vv.1>
\eeq
By \cite[Lemma~3.2]{RTV}, we have
\,$\WW_{\si^J\!,\>J}(\Si_K,\zz;h)=c_\bla(z_J)\,R(\zz_J)\,\dl_{\JK}$\,.
Thus taking $\;\TTT=\Si_K$ \,in formula \Ref{WWW}\,,
we get equality \Ref{WWDl}.

\vsk.2>
Formula \Ref{WWDl} also follows from \cite[Lemma~3.4]{RTV}.
Proposition \ref{pdel} is proved.
\end{proof}

\subsection{Quantum multiplication by divisors on \,$H^*_T(\tfl)$\,}
The quantum multiplication by divisors on \,$H^*_T(\tfl)$ is described in
\cite{MO}. The fundamental equivariant cohomology classes of divisors on $\tfl$
are linear combinations of \,$D_i\<=\gamma_{i,1}\<\lsym+\ga_{i,\la_i}$\,,
\vv.06>
\,$i=1\lc N$. The quantum multiplication
\vv.07>
\,$D_i\>*_\qqt\<\>:H^*_T(\tfl)\to H^*_T(\tfl)$ \,by the divisor \,$D_i$
\,depends on parameters \,$\qqt=(\qti_1\lc\qti_N)\in(\Cxs)^N$ and is given
in \cite[Theorem 10.2.1]{MO}.

\vsk.2>
The quantum connection \,$\naqla_{\bla\>,\<\>\qqt,\<\>\kat}$
\>on \,$H^*_T(\tfl)$ is defined by the formula
\vvn.1>
\be
\naqla_{\bla\>,\<\>\qqt,\<\>\kat\<\>,\<\>i}\,=\,
\kat\,\qti_i\frac\der{\der\qti_i}\>-\>D_i\>*_\qqt\,,\qquad i=1\lc N\>,
\vv.3>
\ee
where \,${\kat\in\Cxs}$ is the parameter of the connection, see \cite{BMO}.
The system of equations for flat sections of the quantum connection is called
the system of the {\it equivariant quantum differential equations}.

The isomorphism $\;\St\>$ \,allows us to compare the operators
\,$\nabla_{\bla\>,\<\>\qq,\<\>\ka\<\>,\<\>i}:=
\nabla_{\qq,\<\>\ka\<\>,\<\>i}\big|_{\Cnnl}$ \,of the dynamical connection
on \,$\Cnnl$\,, see~\Ref{nady}, \,and the operators
\,$\naqla_{\bla\>,\<\>\qqt,\<\>\kat\<\>,\<\>i}$ \,of the quantum connection
on \,$H^*_T(\tfl)$\,.

\vsk.2>
Recall the dynamical Hamiltonians \,$X_i(\zz;h;\qq)$\,, see \Ref{Xi}.
Define the modified dynamical Hamiltonians
\begin{align}
\label{Xl-i}
& X^-_{\blai}(\zz;h;\qq)\,={}
\\[3pt]
\notag
&\!{}=\,X_i(\zz;h;\qq)\big|_{\Cnnl}-\>
h\>\sum_{j=1}^{i-1}\>\frac{q_i}{q_i\<-q_j}\,\min\<\>(\la_i,\la_j)\>-\>
h\sum_{j=i+1}^N\>\frac{q_j}{q_i\<-q_j}\,\min\<\>(\la_i,\la_j)\,.
\end{align}
The modified dynamical connection on $\Cnnl$ is
\vvn.2>
\beq
\label{necon}
\nabla^-_{\bla\>,\<\>\qq,\<\>\ka\<\>,\<\>i}\,=\,
\ka\>q_i\frac{\der}{\der q_i} - X^-_{\blai}(\zz;h;\qq)\,,
\qquad i=1\lc N\>.
\vv.2>
\eeq
see \cite[Section 3.4]{GRTV}. Recall that \,$\hgrtv=-\>h$\,.

\begin{thm}[{\cite[Corollary 7.6]{RTV}}]
\label{qmh}
The isomorphism $\;\St$ \,identifies the operators
\vvn.06>
\,$D_i\>*_{\qqt}$ of quantum multiplication by \,$D_i$ on $H^*_T(\tfl)$
\vvn.1>
\,with the action of the modified dynamical Hamiltonians
$\;X^-_{\blai}(\zz;h;\qqt^{-1})$ \>on \,$\Cnnl$, where
$\;\qqt^{-1}\<\<=(\qti_1^{-1}\<\lc\qti_N^{-1})$\,.
Consequently, the differential operators
$\;\naqla_{\bla\>,\<\>\qqt,\<\>\kat\<\>,\<\>i}$ \>are identified with
the differential operators $\nabla_{\bla\>,\<\>\qqt^{-1}\!,-\kat\<\>,\<\>i}$\,.
\qed
\end{thm}

See also \cite[Theorem 7.5]{RTV}.

\vsk.3>
Set
\vvn->
\beq
\label{Omh}
\Omh_\bla(\qqt\>;\kat)\,=\,\prod_{i=1}^{N-1}\prod_{j=i+1}^N
(1-\qti_i/\qti_j)^{h\min(0,\>\la_j-\la_i)/\<\kat}\>.
\vv.2>
\eeq

Set \,$\la_{\{2\}}=\sum_{1\le i<j\le N}\>\la_i\>\la_j$\,.
For any \,$I\in\Il$\,, define
\vvn.3>
\beq
\label{Psh}
\Psh_I(\zz;h;\qqt\>;\kat)\,=\,\kat^{\>-\la^{\{1\}}\<-\<\>2\<\>\la_{\{2\}}}\,
(-1)^{\la^{\{1\}}\<+\>\la_{\{2\}}}\>
\bigl(\>\Ga(-\>h/\kat)\bigr)^{-\la^{\{1\}}}\,
\Pst_I(\zz;h;\qqt\>;\kat)\,,\kern-1em
\vvn.4>
\eeq
\be
\Pst_I(\zz;h;\qqt\>;\kat)\,=\,\Omh_\bla(\qqt;\kat)\<\<
\sum_{\rr\<\in\Z_{\ge0}^{\la^{\{1\}}}\!\!\!\!}
\Res_{\>\TT\>=\>\Si_I+\<\>\rr\kat\>}
\bigl(\Phi_\bla(\TT;\zz;\qqt^{-1};-\kat)\bigr)\>
Q(\GG)\,\Wh(\Si_I+\<\>\rr\kat;\GG)\,.\kern-1em
\vv.3>
\ee
For \,$\qqt\>,\kat$\, given, \,$\Psh_I(\zz;h;\qqt\>;\kat)$ \,belongs to
the extension of \,$H^*_T(\tfl)$ \,by holomorphic functions in \,$\zz,h$
\,on the domain $\;L'\<\subset\C^n\!\times\C$ \,such that
\vvn.2>
\beq
\label{domaih}
h\not\in\kat\>\Z_{\ge0}\,,\qquad
z_a\<-z_b\<\not\in\kat\>\Z\,,\qquad z_a\<-z_b\<+h\not\in\kat\>\Z\,,
\vv.1>
\eeq
for all \,$a,b=1\lc n$\,, \,$a\ne b$\,. Furthermore,
\,$\Psh_I(\zz;h;\qqt\>;\kat)$
\vvn.06>
\,is a holomorphic function of \,$\qqt$ on the domain $\;L''\<\<\subset\C^N\<$
such that \,$|\>\qti_i/\qti_{i+1}|<1$\,, \,$i=1\lc N-1$\,, and a branch of
$\;\log\>\qti_i$ \,is fixed for each \,$i=1\lc N$.

\begin{example}
Let \,$N=2$\,, \,$n=2$\,, \,$\bla=(1,1)$\,.
Recall the Gauss hypergeometric series
\be
_2F_1(a,b;c;x)\,=\,
\sum_{m=0}^\infty\,\frac{(a)_m\,(b)_m}{(c)_m}\;\frac{x^m}{m\<\>!}\;,
\ee
where \,$(u)_m=u\>(u-1)\ldots(u-m+1)$\,. Set
\be
F(z_1,z_2;h;\kat;x)\,=\>_2F_1\Bigl(-\>\frac h\kat\,,
\frac{z_1\<-z_2\<-h}\kat\,;1+\frac{z_1\<-z_2}\kat\,;x\Bigr)
\vv-.2>
\ee
and
\vvn.2>
\be
F'(z_1,z_2;h;\kat;x)\,=\,\frac{\der F}{\der x}\>(z_1,z_2;h;\kat;x)\,.
\vv-.2>
\ee
Then
\vvn.3>
\begin{align*}
& \Psh_{(\{1\},\{2\})}(z_1,z_2;h;\qti_1,\qti_2;\kat)\,=\,{}
\\[6pt]
&{}\!=\,\kat^{-2}\,(e^{-\pii}\,\qti_1)^{\<\>z_1\</\kat}\,
(e^{-\pii}\,\qti_2)^{\<\>z_2\</\kat}
\;\Ga\Bigl(\>\frac{z_2\<-z_1}\kat\<\>\Bigr)\,
\Ga\Bigl(\>\frac{z_1\<-z_2\<-h}\kat\<\>\Bigr)
\\[3pt]
& \hp{{}={}}\!\!\times\>(\ga_{1,1}\<-\ga_{2,1}\<-h)\>
\bigl((\ga_{2,1}\<-z_1)\,F(z_1,z_2;h;\kat;\qti_1/\qti_2)-
\kat\>(\qti_1/\qti_2)\,F'(z_1,z_2;h;\kat;\qti_1/\qti_2)\bigr)
\\[-14pt]
\end{align*}
and \;\,$\Psh_{(\{2\},\{1\})}(z_1,z_2;h;\qti_1,\qti_2;\kat)=
\Psh_{(\{1\},\{2\})}(z_2,z_1;h;\qti_1,\qti_2;\kat)$\,.
\end{example}

\begin{thm}
\label{mcor}
The collection of functions \,$\bigl(\Psh_I(\zz;h;\qqt\>;\kat)\bigr)_{I\in\Il}$
\vvn-.16>
is a basis of solutions of both the quantum differential equations
\,$\naqla_{\bla\>,\<\>\qqt,\<\>\kat\<\>,\<\>i}\,f\,=\,0$\,, \;$i=1\lc N$,
and the associated \qKZ/ difference equations.
\end{thm}
\begin{proof}
The statement follows from Theorems \ref{thm cy}, \ref{thm ba}, \ref{qmh},
and Proposition \ref{p111}, see Corollary \ref{cor KZdy}.
\end{proof}

\begin{rem}
The integral representations for solutions of the equivariant quantum
differential equations is a manifestation of a version of mirror symmetry.
The basis of solutions given by Theorem \ref{mcor} is an analog of Givental's
\,$J$-function.
\end{rem}

For \,$\qti_i/\qti_{i+1}\to0$ \,for all \,$i=1\lc N-1$\,, the leading term of
the asymptotics of $\,\,\Psh_I(\zz;h;\qqt\>;\kat)$ \>is given by taking
the residue at \;$\TT\>=\>\Si_I$\,.

\begin{thm}
\label{lead}
Assume that \,$\qti_i/\qti_{i+1}\to0$ \,for all \,$i=1\lc N-1$\,. Then
\vvn.3>
\begin{align}
\label{FHI}
\Psh_I(\zz;h;\qqt\>;\kat)\,&{}=\,
\prod_{i=1}^N\>\bigl(\>e^{\<\>\pii\,(\la_i-\<\>n)}\>\qti_i\bigr)
^{\>\sum_{a\in I_i}\<\<z_a/\kat}
\\[5pt]
\notag
&{}\times\,\prod_{i=1}^{N-1}\prod_{j=i+1}^N\,\prod_{a\in I_i}\,
\prod_{b\in I_j}\,\Ga\Bigl(1+\frac{z_b\<-z_a}\kat\Bigr)\,
\Ga\Bigl(1+\frac{z_a\<-z_b\<-h}\kat\Bigr)
\\[2pt]
\notag
&{}\times\,\biggl(\Dl_I+\!
\sum_{\satop{\mb\in\Z_{\ge 0}^{N\<-1}\!\!\!\!}{\mb\ne 0}}
\>\Psh_{I\<,\<\>\mb}(\zz;h;\kat)
\,\prod_{i=1}^{N\<-1}\>\Bigl(\>\frac{\qti_i}{\qti_{i+1}}\>\Bigr)^{\!m_i}
\biggr),
\\[-20pt]
\notag
\end{align}
where $\;\Dl_I(\GG,\zz)=R_I(\GG;\zz)/R(\zz_I)$ \,is the cohomology class
such that $\;\Dl_I(\zz_J;\zz)=\dl_{\IJ}$\,, and the classes
\;$\Psh_{I\<,\<\>\mb}(\zz;h;\kat)$ are rational functions in \,$\zz,h,\kat$\,,
\,regular on the domain $\;h\not\in\kat\>\Z_{\ge0}$\,,
$\;z_a\<-z_b\<\not\in\kat\>\Z$\,, $\;z_a\<-z_b\<+h\not\in\kat\>\Z$\,,
\;for all \,$a,b=1\lc n$\,, \,$a\ne b$\,.
\end{thm}
\begin{proof}
The statement follows from formula \Ref{Psh} and Propositions~\ref{p111},
\ref{pdel}.
\end{proof}

\begin{example} Let \,$N=2$\,, \,$n=2$\,, \,$\bla=(1,1)$\,.
As \,$\qti_1/\qti_2\to 0$\,, the leading term of the solution
\,$\Psh_{(\{1\},\{2\})}(z_1,z_2;h;\qti_1,\qti_2;\kat)$ \,is the cohomology
class
\vvn.2>
\be
(e^{-\pii}\,\qti_1)^{\<\>z_1\</\kat}\,
(e^{-\pii}\,\qti_2)^{\<\>z_2\</\kat}
\;\Ga\Bigl(1+\frac{z_2\<-z_1}\kat\<\>\Bigr)\,
\Ga\Bigl(1+\frac{z_1\<-z_2\<-h}\kat\<\>\Bigr)\,
\Dl_{(\{1\},\{2\})}\!
\vv.1>
\ee
and the leading term of the solution
\,$\Psh_{(\{2\},\{1\})}(z_1,z_2;h;\qti_1,\qti_2;\kat)$ \,is the cohomology
class
\vvn.3>
\be
\hp.(e^{-\pii}\,\qti_1)^{\<\>z_2\</\kat}\,
(e^{-\pii}\,\qti_2)^{\<\>z_1\</\kat}
\;\Ga\Bigl(1+\frac{z_1\<-z_2}\kat\<\>\Bigr)\,
\Ga\Bigl(1+\frac{z_2\<-z_1\<-h}\kat\<\>\Bigr)\,
\Dl_{(\{2\},\{1\})}\,.\!\!
\vv.1>
\ee
\end{example}

\section{Quantum Pieri rules}
\label{QuPr}

\subsection{Quantum equivariant cohomology algebra $\Hqtl$}
Let
\vvn.06>
\,$\qqt=(\qti_1\lc\qti_N)\in(\Cxs)^N$ have distinct coordinates.
\vv.06>
The {\it quantum equivariant cohomology algebra} \,$\Hqtl$ \,is the
\vv.06>
algebra generated by the operators \,$D_i\>*_\qqt:H^*_T(\tfl)\to H^*_T(\tfl)$
\,of quantum multiplication by the divisors \,$D_i$\,, \,$i=1\lc N$, see
details in \cite{MO, GRTV}. The algebra can be defined by generators and
relations as follows.

\vsk.2>
Introduce the variables $\gaq_{i,1}\lc\gaq_{i,\<\>\la_i}$ \,for
\,$i=1\lc N$. Set
\beq
\label{W}
\Wqt(u)\,=\,\det\>\biggl(\qti_i^{\>j-1}\;\prod_{k=1}^{\la_i}\,
\bigl(u-\gaq_{\ik}\<\<-h\>(i-j)\bigr)\<\biggr)_{i,j=1}^N\,.
\vv.1>
\eeq

\begin{thm}
\label{thm qca}
The quantum equivariant cohomology algebra \,$\Hqtl$ is isomorphic to the
algebra
\vvn-.7>
\beq
\label{HK}
\C[\GGq]^{\>\Sla}\<\<\otimes\C[\zz]\otimes\C[h]\>\Big/
\Bigl\bra\Wqt(u)\>=\<\!\prod_{1\le i<j\le N}\!(\qti_j\<-\qti_i)\,
\prod_{a=1}^n\,(u-z_a)\Bigr\ket
\eeq
where \;$\GGq=(\gaq_{1,\<\>1}\lc\gaq_{1,\<\la_1}\lc
\gaq_{N\<,\<\>1}\lc\gaq_{N\<,\<\la_N})$\,, and the correspondence is
\vvn.2>
\be
D_i\>*_\qqt\,\mapsto\,\Bigl[\,\sum_{k=1}^{\la_i}\,\gaq_{\ik}-\>
h\>\sum_{j=1}^{i-1}\>\frac{\qti_j}{\qti_j\<-\qti_i}\,\min\<\>(0,\la_j\<-\la_i)
\>-\>h\sum_{j=i+1}^N\>\frac{\qti_i}{\qti_j\<-\qti_i}\,\min\<\>(0,\la_j\<-\la_i)
\>\Bigr]\,.
\vv-.8>
\ee
\vv-.2>
\qed
\end{thm}

This theorem follows from \cite[Theorems~6.5, 7.10, and Lemma~6.10\>]{GRTV},
see also \cite{MTV2}. Notice that the parameters in this paper and in
\cite{GRTV}
\vv.06>
are related as follows: \,$h=-\>\hgrtv$, \,$\qti_i\<=q_i^{-1}\<$, \,$i=1\lc N$.

\begin{example}
Let \,$N=2$\,, \,$n=2$\,, \,$\bla=(1,1)$\,. Then
\,$D_i\>*_\qqt\<\>\mapsto\gaq_{i,\<\>1}$\,, \,$i=1,2$\,, and the relations are
\vvn-.3>
\beq
\label{erel}
\gaq_{1,\<\>1}\<+\gaq_{2,\<\>1}\>=\,z_1\<+z_2\,,\qquad
\gaq_{1,\<\>1}\>\gaq_{2,\<\>1}-\>
h\,\frac{\qti_1}{\qti_2\<-\qti_1}\,(\gaq_{1,\<\>1}-\gaq_{2,\<\>1}-h)\,=\,
z_1\<\>z_2\,.
\vv.1>
\eeq
\end{example}

It is easy to see that the algebra \,$\Hqtl$ \,does not change if all
\vvn.1>
\,$\qti_1\lc\qti_N$ are multiplied by the same number. In the limit
\vv.06>
$\qti_i/\qti_{i+1}\to 0$, \,$i=1\lc N-1$, the relations in \,$\Hqtl$
\,turn into the relations in \,$H^*_T(\tfl)$\,, see formula \Ref{Hrel}.

\subsection{Quantum equivariant Pieri rules}

Recall the weight functions \,$W_I(\TT;\zz)$\,,
\vvn.1>
see \Ref{hWI--}.
Introduce the variables $\Thq_i=\{\thq_{i,\<\>1}\lc\thq_{i,\<\la^{(i)}}\}$,
$\Thq=(\Thq_1\lc\Thq_N)$\,. Let \,$W_I(\Thq;\zz)$ \,be the polynomial
\,$W_I(\TT;\zz)$ \,with the variables $\;t^{(i)}_a\<$ replaced by
$\;\thq_{i,\<\>a}$ \>for all \,$i=1\lc N-1$\,, \,$a=1\lc\la^{(i)}$.
For any \,$m=1\lc N-1$\,, the relation
\vvn-.1>
\beq
\label{Well}
\det\>\biggl(\qti_i^{j-1}\,\prod_{k=1}^{\la_i}\,
\bigl(u-\gaq_{\ik}-h\>(i-j)\bigr)\<\biggr)_{i,j=1}^m\,=
\prod_{1\leq i<j\leq m}(\qti_j\<-\qti_i)\;
\prod_{a=1}^{\la^{(m)}}\,(u-\thq_{m,\<\>a})
\vv.1>
\eeq
allows us to express the elementary symmetric functions in the variables
\,$\Thq_m$ \>in terms of the elementary symmetric functions in the variables
\,$\GGq_i=(\gaq_{i,\<\>1}\lc\gaq_{i,\<\la_i})$ \,with \,$i=1\lc m$\,.
For example,
\vvn-.1>
\beq
\label{1sr sy}
\sum_{a=1}^{\la^{(m)}}\thq_{\ell,a}\,=\,
\sum_{i=1}^m\,\sum_{k=1}^{\la_i}\gaq_{\ik}-\>h\!
\sum_{1\leq i<j\leq m}\!\!(\la_i\<-\la_j)\,\frac{\qti_i}{\qti_i\<-\qti_j}\;,
\qquad m=1\lc N.\kern-2em
\eeq
Relations \Ref{Well} define a homomorphism
\vvn.3>
\be
\C[\Thq]^{\>S_{\la^{(1)}}\<\lsym\times\>S_{\la^{(N)}}}\<\<\otimes
\C[\zz]\otimes\C[h]\,\to\,\Hqtl\,.
\vv.2>
\ee
Let \,$\{W_I\}\in\Hqtl$ \,be the cohomology class represented by the image of
\,$W_I(\Thq;\zz)$\,.

\begin{thm}
\label{thm Pi}
For any \,$i=1\lc N$ and \,$I\in\Il$\,, the following relation in \,$\Hqtl$
holds:
\vvn-.3>
\begin{align}
\label{PR}
\biggl(\>\sum_{k=1}^{\la_i}\,\gaq_{\ik}
-\sum_{a\in I_i}\, z_a\biggr)\>\{W_I\}\,=\,h\>\la_i\>\biggl(\>
\sum_{j=1}^{i-1}\,\frac{\qti_j}{\qti_i\<-\qti_j}\>+
\sum_{j=i+1}^N\,\frac{\qti_i}{\qti_i\<-\qti_j}\>\biggr)\>\{W_I\}\>+{}&
\\[4pt]
\notag
{}-\,h\,\sum_{\satop{j=1}{j\ne i}}^N\,\sum_{m_1=1}^{\la_i}\>\biggl(
\sum_{\satop{m_2=1}{\ell_{i,{m_1}}\!<\>\ell_{j,{m_2}}\!\!\!\!}}^{\la_j}
\{W_{I_{i,j;m_1,m_2}}\}-\,\frac{\qti_i}{\qti_i\<-\qti_j}\,
\sum_{m_2=1}^{\la_j}\{W_{I_{\ij;\>m_1\<,m_2}}\}\biggr) &,
\kern-1em
\\[-16pt]
\notag
\end{align}
where \;$\ell_{i,{m_1}},\<\>\ell_{j,{m_2}},\<\>I_{\ij;\>m_1\<,m_2}$ are defined
in Section \ref{seckey}.
\end{thm}

Theorem \ref{thm Pi} is proved in Section \ref{proofs}.

\begin{example}
Let \,$N=2$\,, \,$n=2$\,, \,$\bla=(1,1)$\,. Then
\vvn.2>
\be
\{W_{(\{1\},\{2\})}\}\>=\>\gaq_{1,1}\<-z_2\<-h\,,\qquad
\{W_{(\{2\},\{1\})}\}=\gaq_{1,1}\<-z_1\,,
\vv.3>
\ee
and the quantum Pieri rules take the form
\vvn.1>
\begin{align}
\label{exPi}
& (\gaq_{1,1}\<-z_1)\>\{W_{(\{1\},\{2\})}\}\,=\,
h\,\frac{\qti_1}{\qti_1\<-\qti_2}\,
\bigl(\{W_{(\{1\},\{2\})}\}+\{W_{(\{2\},\{1\})}\}\bigr)
-h\>\{W_{(\{2\},\{1\})}\}\,,
\\[6pt]
\notag
& (\gaq_{1,1}\<-z_2)\>\{W_{(\{2\},\{1\})}\}\,=\,
h\,\frac{\qti_1}{\qti_1\<-\qti_2}\,
\bigl(\{W_{(\{2\},\{1\})}\}+\{W_{(\{1\},\{2\})}\}\bigr)\,.
\end{align}
These are the same relations as in formula \Ref{erel}.
\end{example}

\subsection{Bethe ansatz equations}
The Bethe ansatz equations is the following system of algebraic equations with
respect to the variables \,$\TT$\,:
\vvn.3>
\beq
\label{BAE}
\prod_{k=1}^{\la^{(i-1)}}\,
\frac{t^{(i-1)}_k\!-t^{(i)}_j\!-h}{t^{(i-1)}_k-t^{(i)}_j}\;
\prod_{k=1}^{\la^{(i+1)}}
\frac{t^{(i)}_j\!-t^{(i+1)}_k}{t^{(i)}_j\!-t^{(i+1)}_k\!-h}\;
\prod_{\satop{k=1}{k\ne j}}^{\la^{(i)}}\,
\frac{t^{(i)}_j\!-t^{(i)}_k\!-h}{t^{(i)}_j\!-t^{(i)}_k\!+h}\;=\,
\frac{q_{i+1}}{q_i}\;,
\vv-.2>
\eeq
for \,$i=1\lc N-1$\,, \,$j=1\lc\la^{(i)}$.
This system can be reformulated as the system of equations:
\vvn-.3>
\beq
\label{dcp}
\lim_{\ka\to 0}\frac{\Phi_\bla({}\dots,t^{(i)}_j\!+\ka,\dots{};\zz;h;\qq)}
{\Phi_\bla(\TT; \zz;h;\qq)}\;=\,1\,,\qquad
i=1\lc N-1\,,\quad j=1\lc\la^{(i)},\kern-2em
\vv.2>
\eeq
see \cite{TV1, MTV1}.

\begin{lem}
\label{lem bae}
For \,$I\in\Il$ and \;$i=1\lc N-1$\,, let \,$D_{\Ii}(\TT;\zz;h;\qq)$ be
the function defined in \Ref{Di}. Let $\TTc$ be a solution of the Bethe ansatz
equations \Ref{BAE}. Then \,$D_{\Ii}(\TTc;\zz;h;\qq)=0$ and the right-hand side
of formula \Ref{Dic} equals zero at \,$\TT=\TTc$.
\end{lem}
\begin{proof}
If \,$\TTc$ \,is a solution of equations \Ref{BAE}, then the second of
the two factors in the right-hand side of formula \Ref{dab} equals zero
at \,$\TT=\TTc$\,.
\end{proof}

\subsection{Proof of Theorem \ref{thm Pi}}
\label{proofs}

We have the following theorem.

\begin{thm}
\label{thm WB}
Let \,$\TTc$ be a solution of the Bethe ansatz equations \Ref{BAE}. Then there
exist unique polynomials \;$\prod_{k=1}^{\la_i}(u-\gac_{i,k})\in\C[u]$\,,
\,$i=1\lc N$, such that
\beq
\label{tWe}
\det\>\biggl(q_i^{m-j}\,\prod_{k=1}^{\la_i}\,
\bigl(u-\gac_{\ik}\<-h\>(i-j)\bigr)\<\biggr)_{\ij=1}^m=\,
\prod_{1\leq i<j\leq m}(q_i\<-q_j)\,\prod_{a=1}^{\la^{(m)}}\,(u-\tch^{(i)}_a)
\vv.1>
\eeq
for \,$i=1\lc N-1$\,, and
\beq
\label{tWen}
\det\>\biggl(q_i^{N-j}\,\prod_{k=1}^{\la_i}\,
\bigl(u-\gac_{\ik}-h\>(i-j)\bigr)\<\biggr)_{\ij=1}^N=
\,\prod_{1\leq i<j\leq N}(q_i\<-q_j)\,\prod_{a=1}^N\,(u-z_a).
\vv-1.1>
\eeq
\vv.5>
\qed
\end{thm}

This is \cite[Theorem 7.2]{MTV2}, which is \cite[Proposition 7.6]{MV2},
which in its turn is a generalization of \cite[Lemma 4.8]{MV1}.

\begin{proof}[Proof of Theorem \ref{thm Pi}]
Formula \Ref{PR} is obtained from formula \Ref{Dic} by several substitutions.
First take \,$q_i\<=\qti_i^{-1}$ for all \,$i=1\lc N$, substitute the variables
\,$t^{(i)}_j$ by \,$\thq_{\ij}$\,, and replace the term $D_{\Ii}$ by zero.
\vv.06>
Then write symmetric functions in the variables \,$\Thq_m$ \>via symmetric
functions in the variables \,$\GGq_i$\,, \,$i=1\lc m$\,.
\vv.06>
As a result, the expression
\,$\sum_{j=1}^{\la^{(i)}}t^{(i)}_j\!-\sum_{j=1}^{\la^{(i-1)}}\!t^{(i-1)}_j$
\vv.06>
\>becomes \,$\sum_{j=1}^{\la_i}\gaq_{\ij}-
h\>\sum_{j=1}^{i-1}(\la_i-\la_j)\,\qti_j\<\>/(\qti_i\<-\qti_j)$ \,according
to formula \Ref{1sr sy}.

\vsk.2>
Lemma \ref{lem bae} and Theorem \ref{thm WB} mean that formula \Ref{PR} holds
for those values of \,$\GGq_1\lc\GGq_{N-1}$ that come from solutions $\TTc$
of the Bethe ansatz equations \Ref{BAE}. By \cite[Theorem 7.3]{MTV2} of
completeness of the Bethe ansatz, such values of \,$\GGq_1\lc\GGq_{N-1}$ form
a Zariski open subset of all values of \,$\GGq_1\lc\GGq_{N-1}$ satisfying
defining relations of the algebra \,$\Hqtl$\,, see \Ref{HK}.
This proves Theorem \ref{thm Pi}.
\end{proof}

\subsection{Limit \,$\qti_i/\qti_{i+1}\to 0$\,, \,$i=1\lc N-1$\,,
and CSM classes of Schubert cells}

In the limit \,$\qti_i/\qti_{i+1}\to 0$ \,for all \,$i=1\lc N-1$\,,
the algebra \,$\Hqtl$ \,turns into the algebra \,$H^*_T(\tfl)$ \,and
the classes $\{W_I\}\in\Hqtl$ become the classes \,$[W_I]\in H^*_T(\tfl)$\,.
Then formula \Ref{PR} takes the form
\begin{align}
\label{PRl}
\Bigl(\>{}&\sum_{k=1}^{\la_i}\ga_{\ik}\>-\sum_{a\in I_i}z_a\Bigr)\>[W_I]\,={}
\\[1pt]
\notag
&{}=\,h\>\sum_{j=1}^{i-1}\,\sum_{m_1=1}^{\la_i}\!
\sum_{\satop{m_2=1}{\ell_{i,{m_1}}\<>\>\ell_{j,{m_2}}\!\!\!\!}}^{\la_j}
[W_{I_{i,j;m_1,m_2}}] -\>h\sum_{j=i+1}^N\,\sum_{m_1=1}^{\la_i}\!
\sum_{\satop{m_2=1}{\ell_{i,{m_1}}\<<\>\ell_{j,{m_2}}\!\!\!\!}}^{\la_j}
[W_{I_{i,j;m_1,m_2}}]\,.
\\[-14pt]
\notag
\end{align}
In particular, identities in \Ref{exPi} turn into the identities
\vvn.3>
\beq
\label{exPi0}
(\ga_{1,1}\<-z_1)\,[W_{(\{1\},\{2\})}]\,=\,-\>h\,[W_{(\{2\},\{1\})}]\,,
\qquad
(\ga_{1,1}\<-z_2)\,[W_{(\{2\},\{1\})}]\,=\,0\,.
\vv.3>
\eeq

\begin{rem}
After the substitution \,$h=1$\,, the classes \,$[W_I]\in H^*_T(\tfl)$
\,can be considered as elements of the equivariant cohomology
\,$H^*_{(\Cxs)^n}(\Fla)$\,. By \cite{RV} these new classes \,$[W_I]_{h=1}$
\,are proportional to the CSM classes \,$\ka_I$ \>of the corresponding Schubert
cells with the coefficient of proportionality independent of the index \,$I$\,.
Hence formula \Ref{PRl} induces the equivariant Pieri rules for the equivariant
CSM classes:
\vvn.3>
\begin{align}
\label{PRlC}
\Bigl(\>{}&\sum_{k=1}^{\la_i}\ga_{\ik}\>-\sum_{a\in I_i}z_a\Bigr)\>\ka_I\,={}
\\[1pt]
\notag
&{}=\,h\>\sum_{j=1}^{i-1}\,\sum_{m_1=1}^{\la_i}\!
\sum_{\satop{m_2=1}{\ell_{i,{m_1}}\<>\>\ell_{j,{m_2}}\!\!\!\!}}^{\la_j}
\ka_{I_{i,j;m_1,m_2}}\<-\>h\sum_{j=i+1}^N\,\sum_{m_1=1}^{\la_i}\!
\sum_{\satop{m_2=1}{\ell_{i,{m_1}}\<<\>\ell_{j,{m_2}}\!\!\!\!}}^{\la_j}
\ka_{I_{i,j;m_1,m_2}}\>,
\\[-14pt]
\notag
\end{align}
see detailed definitions of the CSM classes in \cite{RV}.
\end{rem}

\section{Solutions of quantum differential equations and equivariant
\,$K\<$-theory}
\label{seckth}
\subsection{Solutions and equivariant \,$K\<$-theory}
\label{equivkth}

Introduce more variables: \,$y=e^{\<\>2\<\>\pii\,h/\kat}$,
\;$\tdd^{\>(i)}_j\!=e^{\<\>2\<\>\pii\;t^{(i)}_j\!\</\<\kat}$,
\,$\zdd_i\<\<=e^{\<\>2\<\>\pii\,z_i/\kat}$,
\,$\gmd_{\ij}\<=e^{\<\>2\<\>\pii\,\gm_{i,j}/\kat}$, etc.
\vvn-.1>
We will use the acute superscript also for the corresponding collections
\vvn.1>
of those variables like \;$\GGd,\>\ttd\>,\zzd$\,. We will write
\;$\GGd^{\pm1}\<,\alb\>\ttd^{\pm1}\<,\alb\zzd^{\pm1}$ for the collections
extended by the inverse variables, for instance,
\,$\zzd^{\pm1}\<=(\zdd_1^{\pm1}\lc\zdd_n^{\pm1})$\,.

\vsk.2>
Let \,$P$ \,be a Laurent polynomial in the variables \;$\ttd,\zzd,y$\,,
symmetric in \,$\tdd^{\>(i)}_1\!\lc\tdd^{\>(i)}_{\la^{(i)}}$ \,for each
\,$i=1\lc N-1$\,. Define
\beq
\label{PPI}
\Psh_P(\zz;h;\qqt\>;\kat)\,=\,
\sum_{I\in\Il}\,P(\Sid_I,\zzd,y)\,\Psh_I(\zz;h;\qqt\>;\kat)\,,
\vv-.3>
\eeq
where \,$\Psh_I(\zz;h;\qqt\>;\kat)$ \,are given by \Ref{Psh}.
\begin{lem}
\label{PsiPsol}
The function \,$\Psh_P(\zz;h;\qqt\>;\kat)$ \,is a solution of
both the quantum differential equations
\,${\naqla_{\bla\>,\<\>\qqt,\<\>\kat\<\>,\<\>i}\,f=0}$\,, \;$i=1\lc N$,
and the associated \qKZ/ difference equations.
\end{lem}
\begin{proof}
The statement follows from Theorem~\ref{mcor}.
\end{proof}

\begin{lem}
\label{PsiPhom}
For a Laurent polynomial \,$P$ in \,$\ttd,\zzd,y$ \>symmetric in
\vvn-.06>
\,$\tdd^{\>(i)}_1\!\lc\tdd^{\>(i)}_{\la^{(i)}}$ \,for each \,$i=1\lc N-1$\,,
the function \,$\Psh_P(\zz;h;\qqt\>;\kat)$ \,is holomorphic in \,$\qqt$ on
\vvn.06>
the domain $\;L''\<\<\subset\C^N\<$ such that \,$|\>\qti_i/\qti_{i+1}|<1$\,,
\,$i=1\lc N-1$\,, and a branch of $\;\log\>\qti_i$ is fixed for each
\,$i=1\lc N$, and \,$\Psh_P(\zz;h;\qqt\>;\kat)$ is holomorphic in \,$\zz,h$
\,on the domain $\;L'''\<\subset\C^n\!\times\C$ \,such that
\vvn.4>
\beq
\label{domain}
h\not\in\kat\>\Z_{\ge0}\,,\kern2.6em z_a\<-z_b\<+h\not\in\kat\>\Z\,,\kern1.4em
a,b=1\lc n\,,\quad a\ne b\,.\kern-2em
\vv.1>
\eeq
\end{lem}
\begin{proof}
By the properties of \,$\Psh_I(\zz;h;\qqt\>;\kat)$\,, see~\Ref{domaih},
we need only to show that the function \,$\Psh_P(\zz;h;\qqt\>;\kat)$ is regular
at the hyperplanes \,$z_a\<-z_b\in\kat\>\Z$\,. This will be done in
Section~\ref{secpsiph} below.
\end{proof}

Consider the equivariant \,$K\<\<$-theory algebra, see
\cite[Section 2.3]{RTV2}, \cite[Section 4.4]{RTV3},
\vvn.3>
\beq
\label{Krel}
K_T(\tfl)\,=\,\C[\GGd^{\pm1}]^{\>\Sla}\<\<\otimes
\C[\zzd^{\pm1}]\otimes\C[y^{\pm1}]\>\Big/\Bigl\bra
\,\prod_{i=1}^N\prod_{j=1}^{\la_i}\,(u-\gmd_{\ij})\,=\,
\prod_{a=1}^n\,(u-\zdd_a)\Bigr\ket\,,\kern-1.6em
\vv.1>
\eeq
cf.~\Ref{Hrel}.
Introduce the variables \,$\thd_{i,\<\>\la^{(i)}}$\,, \,$i=1\lc N$.
The relations
\vvn.1>
\beq
\label{kre}
\prod_{a=1}^{\la^{(i)}}\,(u-\thd_{i,\<\>a})\,=\,
\prod_{j=1}^i\,\prod_{k=1}^{\la_j}\,(u-\gmd_{\jk})\,, \qquad i=1\lc N\>,
\vv.2>
\eeq
define the epimorphism
\,$\C[\Thd^{\pm1}]^{S_{\la^{(1)}}\<\lsym\times S_{\la^{(N)}}}\!
\otimes \C[\zzd^{\pm1}]\otimes\C[y^{\pm1}]\,\to\,K_T(\tfl)$\,.
\vvn.1>
Thus the assignment \,$P\mapsto\Psh_P$ \,defines a map from \,$K_T(\tfl)$ \,to
\vvn.06>
the space of solutions of the quantum differential equations and the associated
\vvn.1>
\qKZ/ difference equations with values in \,$H^*_T(\tfl)$ extended by functions
in \,$\zz,h,\qqt$ \,holomorphic in the domain \,$L'''\!\times L''$. We evaluate
below the determinant of this map.

\vsk.3>
The cohomology algebra \,$H^*_T(\tfl)$ \,is a free module over
\,$H^*_T({pt};\C)=\C[\zz]\otimes\C[h]$\,, with a basis given
by the classes of Schubert polynomials
\vvn.4>
\beq
\label{YG}
Y_I(\GG)\,=\,A_{\si^I}(\ga_{1,1}\lc\ga_{1,\>\la_1},
\ga_{2,1}\lc\ga_{2,\>\la_2},\,\ldots\,,\ga_{N\<,\>1}\lc\ga_{N\<,\>\la_N})\,,
\qquad I\in\Il\,.\kern-1em
\vv.4>
\eeq
Similarly, the algebra \,$K_T(\tfl)$ \,is a free module over
\,$K_T({pt};\C)=\C[\zzd^{\pm1}]\otimes\C[y^{\pm1}]$\,, with
a basis given by the classes of Schubert polynomials
\vvn.4>
\beq
\label{YhG}
\Yh_I(\GGd)\,=\,A_{\si^I}(\gmd_{1,1}\lc\gmd_{1,\>\la_1},
\gmd_{2,1}\lc\gmd_{2,\>\la_2},\,\ldots\,,\gmd_{N\<,\>1}\lc\gmd_{N\<,\>\la_N})\,,
\qquad I\in\Il\,.\kern-1em
\vv.4>
\eeq
Both assertions are clear from Proposition~\ref{fAsil}.

\vsk.2>
Expand solutions of the quantum differential equation using those Schubert
bases:
\vvn.4>
\beq
\label{PshY}
\Psh_{\Yh_I}(\GG;\zz;h;\qqt\>;\kat)\,=\,\sum_{J\in\Il\!}\,
\Pshb_{\IJ}(\zz;h;\qqt\>;\kat)\,Y_J(\GG)\,.
\vv-.2>
\eeq

\begin{thm}
\label{detPYY}
Let \,$n\ge 2$\,. We have
\vvn.3>
\begin{align}
\label{detPsh}
\det\bigl(\Pshb_{\IJ}(\zz;h;\qqt\>;\kat)\bigr)_{\IJ\in\Il}\<=\,
\prod_{i=1}^{N-1}\prod_{j=i+1}^N(1-\qti_i/\qti_j)
^{h\min(\la_i,\la_j)/\<\kat}\;\prod_{i=1}^N\,
\qti_i^{\,d^{(1)}_{\bla,i}\sum_{a=1}^n z_a\</\<\kat}\,\times{}&
\\[3pt]
\notag
{}\times\,\prod_{a=1}^n\,\prod_{\satop{b=1}{b\ne a}}^n\,
\Bigl(\<\>2\<\>\pi\<\>\sqrt{-1}\,\,
\Gm\Bigl(1+\frac{z_a\<-z_b\<-h}\kat\>\Bigr)
\Bigr)^{\!d^{(2)}_\bla}\!,\; &
\\[-24pt]
\notag
\end{align}
where
\vvn-.3>
\beq
\label{dla12}
d^{(1)}_{\bla\<\>,\<\>i}\>=\,
\frac{\la_i\>(n-1)\<\>!}{\la_1\<\>!\ldots\la_N\<\>!}\;,\qquad
d^{(2)}_\bla\>=\,\frac{2\>(n-2)\<\>!}{\la_1\<\>!\ldots\la_N\<\>!}\,
\sum_{i=1}^{N-1}\sum_{j=i+1}^N\la_i\>\la_j\,.\kern-2em
\vv.1>
\eeq
\end{thm}
\begin{proof}
By Lemma \ref{PsiPsol}, the left-hand side of~\Ref{detPsh} solves
the differential equations
\vv.3>
\be
\Bigl(\<\>\kat\>\qti_i\frac{\der}{\der\qti_i}\>-\>
\tr\bigl(X^-_{\blai}(\zz;-h;\qqt^{-1})\bigr)\<\<\Bigr)\>
\det\<\>\bigl(\Pshb_{\IJ}(\zz;h;\qqt\>;\kat)\bigr)_{\IJ\in\Il}\<=\,0\,,
\qquad i=1\lc N\>,\kern-2em
\vv.1>
\ee
where \,$X^-_{\blai}$ \,are the modified dynamical Hamiltonians \Ref{Xl-i}.
\vv.06>
Thus \;$\det\<\>(\Pshb_{\IJ})$ \,equals the first two products in
the right-hand side of \Ref{detPsh} multiplied by a factor that does not
depend on \,$\qqt$\,. The remaining factor is found by taking the limit
\,$\qti_i/\qti_{i+1}\to0$ \,for all \,$i=1\lc N-1$\,, and applying
Theorem~\ref{lead} and Proposition~\ref{detA}.
\end{proof}

\begin{cor}
\label{basisK}
The collection of functions
\,$\bigl(\Psh_{\Yh_I}(\zz;h;\qqt\>;\kat)\bigr)_{I\in\Il}$
\vvn-.16>
is a basis of solutions of both the quantum differential equations
\,$\naqla_{\bla\>,\<\>\qqt,\<\>\kat\<\>,\<\>i}\,f\,=\,0$\,, \;$i=1\lc N$,
and the associated \qKZ/ difference equations.
\qed
\end{cor}

\subsection{End of proof of Lemma~\ref{PsiPhom}}
\label{secpsiph}

It is enough to show the regularity of \,$\Psh_P(\zz;h;\qqt\>;\kat)$
\,at the hyperplanes \,$z_a\<-z_b\in\kat\>\Z$ \,assuming that
\,$h/\kat$ is real negative and sufficiently large.

\vsk.2>
For a number \,$A$\,, let \,$C(A)\subset\C$ \,be a parabola with the following
parametrization:
\vvn.3>
\beq
\label{CA}
C(A)\,=\,\{\,\kat\>\bigl(A+s^2\<\<-s\>\sqrt{\<-1}\>\bigr)\ \,|\ \,s\in\R\,\}\,.
\vv.2>
\eeq
Given \,$\zz,\>\kat$\,, take \>$A$ \,such that all the points \,$\zzz$ \,are
inside \,$C(A+N-2)$\,. Suppose \,$h/\kat$ is a sufficiently large negative real
so that all the points \,$z_1\<+h\lc z_n\<+h$ \,are outside \,$C(A)$\,. Set
\vvn-.7>
\beq
\label{Ccz}
\Cc_\bla(\zz)=\bigl(C(A)\bigr)^{\times\<\>\la^{(1)}}\!\<\<\lsym\times
\bigl(C(A+N-2)\bigr)^{\times\<\>\la^{(N\<-1)}}
\vv.2>
\eeq
indicating the dependence on \,$\bla$ \,and \,$\zz$ \,explicitly. The integral
\Ref{PshP} below does not depend on a particular choice of \,$A$\,.

\goodbreak
\vsk.2>

\begin{lem}
\label{intlem}
For a Laurent polynomial \,$P$ in \,$\ttd,\zzd,y$ \>symmetric in
\vvn-.06>
\,$\tdd^{\>(i)}_1\!\lc\tdd^{\>(i)}_{\la^{(i)}}$ \,for each \,$i=1\lc N-1$\,,
we have
\vvn.3>
\beq
\label{PshP}
\Psh_P(\zz;h;\qqt\>;\kat)\,=\,
\kat^{\>-\la^{\{1\}}\<-\<\>2\<\>\la_{\{2\}}}\,
(-1)^{\la^{\{1\}}\<+\>\la_{\{2\}}}\>
\bigl(\>2\<\>\pi\<\>\sqrt{-1}\;\Ga(-\>h/\kat)\bigr)^{-\la^{\{1\}}}\,
\Pst_P(\zz;h;\qqt\>;\kat)\,,
\vvn.2>
\eeq
\be
\Pst_P(\zz;h;\qqt\>;\kat)\,=\,
\frac{\Omh_\bla(\qqt;\kat)}{\la^{(1)}\<\>!\dots\la^{(N\<-1)}\<\>!}\;
\int_{\<\>\Cc_\bla(\zz)}\!\!P(\ttd;\zzd;y)\,
\Phi_\bla(\TT;\zz;\qqt^{-1};-\kat)\,
Q(\GG)\,\Wh(\TT\<\>;\GG)\;d^{\>\la^{\{1\}}}\<\TT\,.
\vv.4>
\ee
\end{lem}
\begin{proof}
The integral converges provided \,$|\>\qti_i/\qti_{i+1}|<1$ \,for all
\vvn.06>
\,$i=1\lc N-1$\,, and a branch of $\;\log\>\qti_i$ is fixed for each
\,$i=1\lc N$\,. Evaluate the integral by residues in the following way:
replace \,$\Cc_\bla(\zz)$ \,by
\,$\bigl(C(A+B)\bigr)^{\times\<\>\la^{(1)}}\!\<\<\lsym\times
\bigl(C(A+B+N-2)\bigr)^{\times\<\>\la^{(N\<-1)}}$,
\vvn.1>
where \,$B\<\in\<\R_{\ge0}$\,, and send \,$B$ to infinity.
Then by~\Ref{Psh}, the resulting series yields formula \Ref{PPI}.
\end{proof}

The integrand in formula~\Ref{PshP} is regular at the hyperplanes
\,$z_a\<-z_b\in\kat\>\Z$, and so does the function
\,$\Psh_P(\zz;h;\qqt\>;\kat)$\,. Lemma~\ref{PsiPhom} is proved.
\qed

\subsection{The homogeneous case \,$\zz=0$\,}
\label{homogen}

The quantum differential equations
\,$\naqla_{\bla\>,\<\>\qqt,\<\>\kat\<\>,\<\>i}\,f\,=\,0$
\,depend on \,$\zz$ \,as a parameter and are well defined at \,$\zz=0$\,.

\vsk.2>
For any Laurent polynomial \,$P$ \,in \;$\ttd,y$\,, symmetric in
\,$\tdd^{\>(i)}_1\!\lc\tdd^{\>(i)}_{\la^{(i)}}$ \,for each \,$i=1\lc N-1$\,,
the function \,$\Psh_P(0\<\>;h;\qti;\kat)$ is a solution of the quantum
differential equations
\,$\naqla_{\bla\>,\<\>\qqt,\<\>\kat\<\>,\<\>i\<\>,\>\zz=0}\,f\,=\,0$\,,
\;$i=1\lc N$, see Lemma~\ref{PsiPsol}.

\begin{lem}
\label{PsiPhom0}
The function \,$\Psh_P(0\<\>;h;\qti;\kat)$ \,is holomorphic in \,$\qqt,h$
\vvn.1>
provided \,$|\>\qti_i/\qti_{i+1}|<1$\,, \,$i=1\lc N-1$\,, a branch of
$\;\log\>\qti_i$ \,is fixed for each \,$i=1\lc N$, and
\,$h\not\in\kat\>\Z_{\ge0}$\,.
\end{lem}
\begin{proof}
By Lemma~\ref{PsiPhom}, we need only to show that \,$\Psh_P(0\<\>;h;\qti;\kat)$
\,is regular if \,$h\in\kat\>\Z_{<0}$\,. We will prove that
\,$\Psh_P(0\<\>;h;\qti;\kat)$ \,is regular if \,$h/\kat\in\R_{<0}$\,.

\vsk.2>
If \,$h/\kat$ \,is a sufficiently large negative real, write
\,$\Psh_P(0\<\>;h;\qti;\kat)$ \,by formula~\Ref{PshP}. Then one can replace
the integration contour \,$\Cc_\bla(0)$ \,by the contour
\vvn.1>
\be
\Cc'_\bla(h,\kat)\,=\,
\bigl(C\bigl((N\<\<-1)\>\eps\bigr)\bigr)^{\times\<\>\la^{(1)}}\!\<\<
\times
\bigl(C\bigl((N\<\<-2)\>\eps\bigr)\bigr)^{\times\<\>\la^{(2)}}\!\<\<
\lsym\times
\bigl(C(\eps)\bigr)^{\times\<\>\la^{(N\<-1)}},
\vv.2>
\ee
where \,$\eps=h/(N\<\>\kat)$\,, without changing the integral.
With the integration over \,$\Cc'_\bla(h,\kat)$\,, it is clear that
\,$\Psh_P(0\<\>;h;\qti;\kat)$ continues to a function regular for all negative
real \,$h/\kat$\,.
\end{proof}

Consider the algebras
\vvn.2>
\be
H^*_\Cxx(\tfl)\,=\,H^*_T(\tfl)/\bra\zz=0\>\ket\,,\qquad
K_\Cxx(\tfl)\,=\,K_T(\tfl)/\bra\zzd=(1\lc 1)\<\>\ket\,.
\vv.2>
\ee
The algebra \,$H^*_\Cxx(\tfl)$ \,is a free module over \,$\C[h]$ \,and
\vvn.1>
the algebra \,$K_\Cxx(\tfl)$ \,is a free module over \,$\C[y^{\pm1}]$\,,
with bases given by the respective classes of Schubert polynomials,
see \Ref{YG}, \Ref{YhG}.

\vsk.2>
Expand solutions of the quantum differential equation at \,$\zz=0$ \,using
those Schubert bases:
\beq
\label{PshY0}
\Psh_{\Yh_I}(\GG;\>\<0\<\>;h;\qqt\>;\kat)\,=\,\sum_{J\in\Il\!}\,
\Pshb_{\IJ}(0\<\>;h;\qqt\>;\kat)\,Y_J(\GG)\,.
\vv.2>
\eeq
Let \,$d_\bla=\<\>n\<\>!\<\>/(\la_1\<\>!\ldots\la_N\<\>!\<\>)$\,.
Formula~\Ref{detPsh} at \,$\zz=0$ \,takes the form
\vvn.4>
\begin{align}
\label{detPsh0}
& \det\bigl(\Pshb_{\IJ}(0\<\>;h;\qqt\>;\kat)\bigr)_{\IJ\in\Il}\<={}
\\[2pt]
&\;{}=\,\Bigl(\<\>2\<\>\pi\<\>\sqrt{-1}\,\,\Gm\Bigl(1-\frac h\kat\>\Bigr)\Bigr)
^{d_\bla\sum_{1\le i<j\le N}\la_i\>\la_j}\,
\prod_{i=1}^{N-1}\prod_{j=i+1}^N(1-\qti_i/\qti_j)^{h\min(\la_i,\la_j)/\<\kat}.
\notag
\end{align}

\begin{cor}
\label{basisK1}
The collection of functions
\,$\bigl(\Psh_{\Yh_I}(0\<\>;h;\qqt\>;\kat)\bigr)_{I\in\Il}$
\vvn-.16>
is a basis of solutions of the quantum differential equations
\,$\naqla_{\bla\>,\<\>\qqt,\<\>\kat\<\>,\<\>i\<\>,\>\zz=0}\,f\,=\,0$\,,
\;$i=1\lc N$.
\qed
\end{cor}

\subsection{The limit \,$h\to\infty$\,}
\label{hinfty}
Suppose that \,$\qti_i/\qti_{i+1}=(-\<\>h)^{-\la_i\<-\la_{i+1}}p_i/p_{i+1}$\,,
\,$i=1\lc N-1$\,, and \,$\qti_N\<=p_N$\,, where \,$p_1\lc p_N$ \,are new
variables. The limit \,$h\to\infty$ \,keeping \,$p_1\lc p_N$ \,fixed
corresponds to replacing the cotangent bundle \,$\tfl$ \>by the partial flag
variety \,$\Fla$ \>itself, the algebras \,$H^*_T(\tfl)$\,, \,$K_T(\tfl)$ \,by
the respective algebras \,$H^*_A(\Fla)$\,, \,$K_A(\Fla)$\,, where \,$A\subset
GL_n(\C)$ \,is the torus of diagonal matrices, and the equivariant quantum
differential equations for $\tfl$ by the analogous equations for $\Fla$.
We will discuss this limit in detail in a separate paper making here only
a few remarks.

\vsk.2>
We identify \,$H^*_A(\Fla)$ with the subalgebra in \,$H^*_T(\tfl)$ \,of
\,$h$-independent elements, and \,$K_A(\Fla)$ with the subalgebra in
\,$K_T(\tfl)$ \,of \,$y$-independent elements.

\vsk.2>
The discussion of the limit \,$h\to\infty$ \,is based on Stirling's formula
\vvn.5>
\beq
\label{Stir}
\frac{\Gm(\al-h/\kat)}{\Gm(\bt-h/\kat)}\;\sim\,(-\<\>h/\kat)^{\al-\bt}\>,
\qquad h\to\infty\,.\kern-2em
\vv.5>
\eeq
For \,$\Fla$\,, \,we have the following counterparts of the master function
\vvn.5>
\begin{align}
\label{PHIo}
\Phi^\cirs_\bla(\TT;\zz;\pp\<\>;\kat)\,=\,
(e^{\<\>\pii\,(\la_N\<-\<\>n)}p_N)^{\>\sum_{a=1}^nz_a/\kat}\,
\prod_{i=1}^{N-1}\>\Bigl(\>
\frac{e^{\<\>\pii\,(\la_i\<-\la_{i+1})}}{\kat^{\>\la_i+\la_{i+1}}}\>
\frac{p_i}{p_{i+1}}\>\Bigr)^{\sum_{j=1}^{\la^{(i)}} t^{(i)}_j\!/\kat}\times{}&
\\[2pt]
\notag
{}\times\<{}\prod_{i=1}^{N-1}\,\prod_{a=1}^{\la^{(i)}}\,\biggl(\,
\prod_{\satop{b=1}{b\ne a}}^{\la^{(i)}}\,
\frac1{\Gm\bigl((t_b^{(i)}\<\<-t_a^{(i)})/\kat\bigr)}\,
\prod_{c=1}^{\la^{(i+1)}}\<\Gm\bigl((t_c^{(i+1)}\<\<-t_a^{(i)})/\kat\bigr)
\<\biggr)\,&,
\\[-21pt]
\notag
\end{align}
the weight function
\vvn-.3>
\beq
\label{Who}
\Who(\TT,\GG)\,=\,\prod_{i=1}^{N-1}\prod_{j=i+1}^N\,
\prod_{a=1}^{\la^{(i)}}\>\prod_{b=\la^{(i)}\<+1}^{\la^{(i+1)}}\!
(t^{(i)}_a\<\<-\gm_{j,\<\>b})\,,
\vv.4>
\eeq
and solutions of the quantum differential equations
\vvn.5>
\beq
\label{Psho}
\Psh^\cirs_I(\zz;\pp\<\>;\kat)\,=\,
(-\>\kat)^{\>-\la^{\{1\}}\<-\<\>\la_{\{2\}}}\;
\kat^{\>\sum_{i=1}^{N-1}\sum_{a\in I^i}(\la_i+\<\>\la_{i+1})\>z_a/\kat}
\;\Pst^\cirs_I(\zz;\pp\<\>;\kat)\,,\kern-1em
\vvn.3>
\eeq
\be
\Pst^\cirs_I(\zz;\pp\<\>;\kat)\,=
\sum_{\rr\<\in\Z_{\ge0}^{\la^{\{1\}}}\!\!\!\!}
\Res_{\>\TT\>=\>\Si_I+\<\>\rr\kat\>}
\bigl(\Phi^\cirs_\bla(\TT;\zz;\pp\<\>;\kat)\bigr)\>
\Who(\Si_I+\<\>\rr\kat;\GG)\,,\kern-1em
\ee
where \;$I^{\>i}\<=\bigcup_{\>j=1}^{\,i}I_j$ \,and
\,$\la_{\{2\}}=\sum_{1\le i<j\le N}\>\la_i\>\la_j$\,.
The series converges and defines a holomorphic function
\,$\Psh^\cirs_P(\zz;\pp\<\>;\kat)$ \,of \,$\zz,\pp$ on the domain in
\vvn.1>
$\;\C^n\!\times\C^N\<$ such that a branch of $\;\log\>p_i$ \,is fixed for each
\,$i=1\lc N$, \,and \,$z_a\<-z_b\<\not\in\kat\>\Z$ \,for all \,$a,b=1\lc n$\,,
\,$a\ne b$\,.

\vsk.3>
Set \,$\la^{\{2\}}=\,\sum_{i=1}^{N-1}\,\bigl(\la^{(i)}\bigr)^2$.
As \,$h\to\infty$\,, we have
\,$(-\<\>h)^{\la^{\{1\}}\<-\la^{\{2\}}}\>\Wh(\TT,\GG)\>\to\,\Who(\TT,\GG)$\,,
\vvn.7>
\be
\frac{(-\<\>h)^{\la^{\{2\}}\<-\<\>\la^{\{1\}}}\>
(-\<\>h/\kat)^{(n\<\>-\la_N)\sum_{a=1}^nz_a/\kat}}
{\bigl(\<\>\Gm(-\<\>h/\kat)\bigr)^{\la^{\{1\}}\<+\<\>\la_{\{2\}}}}\;
\Phi_\bla(\TT;\zz;h;\qqt^{-1};-\kat)\,\to\,
\Phi^\cirs_\bla(\TT;\zz;\pp\<\>;\kat)\,,
\vv-.1>
\ee
and
\vvn.3>
\be
\frac{\kat^{\>\sum_{i=1}^{N-1}\sum_{a\in I^i}(\la_i+\<\>\la_{i+1})\>z_a/\kat}
\>(-\<\>h/\kat)^{(n\<\>-\la_N)\sum_{a=1}^nz_a/\kat}}
{\bigl(\<\>\Gm(1-\<\>h/\kat)\bigr)^{\>\la_{\{2\}}}}\;
\Psh_I(\zz;h;\qqt\>;\kat)\,\to\,
\Psh^\cirs_I(\zz;\pp\<\>;\kat)\,.
\vv.8>
\ee
If \,$p_i/p_{i+1}\to0$ \,for all \,$i=1\lc N-1$\,, then similarly to \Ref{FHI},
\vvn.4>
\begin{align}
\label{FHIo}
\Psh^\cirs_I(\zz;\pp\<\>;\kat)\,=\,
\prod_{i=1}^N\>\bigl(\>e^{\<\>\pii\,(\la_i-\<\>n)}\>p_i\bigr)
^{\>\sum_{a\in I_i}\<\<z_a/\kat}
\,\prod_{i=1}^{N-1}\prod_{j=i+1}^N\,\prod_{a\in I_i}\,
\prod_{b\in I_j}\,\Ga\Bigl(1+\frac{z_b\<-z_a}\kat\Bigr)\times{}\! &
\\[2pt]
\notag
{}\times\,\biggl(\Dl_I+\!
\sum_{\satop{\mb\in\Z_{\ge 0}^{N\<-1}\!\!\!\!}{\mb\ne 0}}
\>\Psh^\cirs_{I\<,\<\>\mb}(\zz;\kat)
\,\prod_{i=1}^{N\<-1}\>\Bigl(\>\frac{p_i}{p_{i+1}}\>\Bigr)^{\!m_i}
\biggr)&,
\\[-15pt]
\notag
\end{align}
where $\;\Dl_I(\GG,\zz)=R_I(\GG;\zz)/R(\zz_I)$ \,is the cohomology class
such that $\;\Dl_I(\zz_J;\zz)=\dl_{\IJ}$\,, and the classes
\;$\Psh^\cirs_{I\<,\<\>\mb}(\zz;\kat)$ are rational functions
in \,$\zz,\kat$\,, regular if \,$z_a\<-z_b\<\not\in\kat\>\Z$
\,for all \,$a,b=1\lc n$\,, \,$a\ne b$\,.

\vsk.2>
Recall the contour $\;\Cc_\bla(\zz)$\,, see~\Ref{Ccz}.
Given a Laurent polynomial \,$P$ \,in the variables \;$\ttd,\zzd$\,,
symmetric in \,$\tdd^{\>(i)}_1\!\lc\tdd^{\>(i)}_{\la^{(i)}}$ \,for each
\,$i=1\lc N-1$\,, define
\vvn.3>
\beq
\label{PshPo}
\Psh^\cirs_P(\zz;\pp\<\>;\kat)\,=\,
\frac{(-\>\kat)^{\>-\la^{\{1\}}\<-\<\>\la_{\{2\}}}\;
\kat^{\>\sum_{i=1}^{N-1}\sum_{a\in I^i}(\la_i+\<\>\la_{i+1})\>z_a/\kat}}
{\bigl(2\<\>\pi\<\>\sqrt{-1}\,\bigr)^{\la^{\{1\}}}}\;
\Pst^\cirs_P(\zz;\pp\<\>;\kat)\,,
\vvn.2>
\eeq
\be
\Pst^\cirs_P(\zz;\pp\<\>;\kat)\,=\,
\frac1{\la^{(1)}\<\>!\dots\la^{(N\<-1)}\<\>!}\;
\int_{\<\>\Cc_\bla(\zz)}\!\!P(\ttd;\zzd)\,
\Phi^\cirs_\bla(\TT;\zz;\pp\<\>;\kat)\bigr)\,
\Wh^\cirs(\TT\<\>;\GG)\;d^{\>\la^{\{1\}}}\<\TT\,,\kern-1.4em
\vv.2>
\ee
cf.~\Ref{Psho}. The integral converges and defines a holomorphic function
\vvn.06>
\,$\Psh^\cirs_P(\zz;\pp\<\>;\kat)$ \,of \,$\zz,\pp$ on the domain in
$\;\C^n\!\times\C^N\<$ such that a branch of $\;\log\>p_i$ \,is fixed for each
\,$i=1\lc N$. Furthermore,
\beq
\label{PPIo}
\Psh^\cirs_P(\zz;\pp\<\>;\kat)\,=\,
\sum_{I\in\Il}\,P(\Sid_I,\zzd)\,\Psh^\cirs_I(\zz;\pp\<\>;\kat)\,,
\eeq
and the assignment \,$P\mapsto\Psh^\cirs_P$ \,defines a map from \,$K_A(\Fla)$
\,to the space of solutions of the quantum differential equations with values
in \,$H^*_A(\Fla)$\,.

\vsk.2>
Consider the classes \,$Y_I(\GG)\>,\,\Yh_I(\GGd)$ \,given by \Ref{YG},
\Ref{YhG}, and write
\vvn.4>
\beq
\label{PshYo}
\Psh^\cirs_{\Yh_I}(\GG;\zz;\pp\<\>;\kat)\,=\,
\sum_{J\in\Il\!}\,\Pshb^\cirs_{\IJ}(\zz;\pp\<\>;\kat)\,Y_J(\GG)\,,
\eeq
cf.~\Ref{PshY}. Taking the limit \,$h\to\infty$ \,in formula \Ref{detPsh}
yields
\vvn.1>
\beq
\label{detPsho}
\det\bigl(\Pshb^\cirs_{\IJ}(\zz;\pp\<\>;\kat)\bigr)_{\IJ\in\Il}\<=\,
\bigl(\<\>2\<\>\pi\<\>\sqrt{-1}\,\bigr)
^{d_\bla\sum_{1\le i<j\le N}\la_i\>\la_j}\,
\prod_{i=1}^N\,p_i^{\,d^{(1)}_{\bla,i}\sum_{a=1}^n z_a\</\<\kat},
\vv-.2>
\eeq
where
\vvn-.3>
\be
d_\bla\>=\,\frac{n\<\>!}{\la_1\<\>!\ldots\la_N\<\>!}\;,\qquad
d^{(1)}_{\bla\<\>,\<\>i}\>=\,
\frac{\la_i\>(n-1)\<\>!}{\la_1\<\>!\ldots\la_N\<\>!}\;.
\vv.4>
\ee
Therefore, the collection of functions
\,$\bigl(\Psh^\cirs_{\Yh_I}(\zz;\pp\<\>;\kat)\bigr)_{I\in\Il}$ is a basis
of solutions of both the quantum differential equations with values in
\,$H^*_A(\Fla)$\,.

\appendix
\section{Basics on Schubert polynomials}
\label{Sch}

For references regarding Schubert polynomials, see for example \cite{L,M}.

\vsk.2>
Let \,$D_1\lc D_{n-1}$ \,be the divided difference operators acting on
functions of \,$\xxx$\,:
\vvn.4>
\be
\Dx_if(\xxx)\,=\,
\frac{f(x_1\lc x_n)-f(x_1\lc x_{i+1},x_i\lc x_n)}{x_i\<-x_{i+1}}\:,
\vv.3>
\ee
cf.~\Ref{Skk+1}. They satisfy the nil-Coxeter algebra relations,
\vvn.4>
\beq
\label{nilCx}
(\Dx_i)^2=\>0\,,\qquad \Dx_i\Dx_{i+1}\Dx_i=\>\Dx_{i+1}\<\>\Dx_i\Dx_{i+1}\,,
\qquad \Dx_i\Dx_j=\>\Dx_j\Dx_i\,,\quad |i-j|>1\,.
\vv.3>
\eeq
Given \,$\si\in S_n$ \,with a reduced decomposition
\,$\si=s_{i_1\<,\<\>i_1+1}\ldots s_{i_j\<,\<\>i_j+1}$\,, define
\,$\Dx_\si\<=\<\>\Dx_{i_1}\<\ldots\Dx_{i_j}$\,. For instance,
\,$\Dx_{\id}$ is the identity operator and \,$\Dx_{s_{\ii+1}}\!=\Dx_i$.
Due to relations \Ref{nilCx}, the operator \,$\Dx_\si$ \,does not depend
on the choice of a reduced decomposition. Moreover,
\vvn.4>
\be
\Dx_\si\>\Dx_\tau=\<\>\Dx_{\si\tau}\,,\quad
\mathrm{if}\ \;|\<\>\si\<\>|+|\<\>\tau\<\>|=|\<\>\si\tau\<\>|\,,\qquad
\Dx_\si\>\Dx_\tau=\<\>0\,,\quad\mathrm{otherwise}\,.\kern-1em
\vv.2>
\ee
Here \,$|\<\>\si\<\>|$ is the length of \,$\si$.
Denote \,$\xx_\si=(x_{\si(1)}\lc x_{\si(n)})$\,.
Let \,$\si_0$ \,be the longest permutation, \,$\si_0(i)=n+1-i$\,,
\;$i=1\lc n$\,. Then
\be
\Dx_{\si_0}\<\>f(\xx)\,=\<\prod_{1\le i<j\le n}\!(x_i\<-x_j)^{-1}
\,\sum_{\si\in S_n\!}\,(-1)^\si\<\>f(\xx_\si)\,.
\ee

\vsk.2>
The Schubert polynomials \,$A_\si(\xx)$\,, \,$\si\in S_n$\,, are defined
by the rule
\vvn.4>
\beq
\label{Schx}
A_\si(\xx)\,=\,
\Dx_{\si^{-1}\si_0}(x_1^{n-1}\>x_2^{n-2}\!\ldots\<\>x_{n-1})\,.
\vv.1>
\eeq
In particular, \,$A_{\si_0}\!=x_1^{n-1}\>x_2^{n-2}\!\ldots\<\>x_{n-1}$ \,and
\,$A_{\id}\<=1$\,.

\begin{prop}
\label{orth}
For any \,$\si,\tau\in S_n$\,,
\vvn.2>
\beq
\label{Asit}
\Dx_{\si_0}\bigl(A_\si(\xx)\>A_{\tau\si_0}(\xx_{\si_0})\bigr)\,=\,
(-1)^{\si\si_0}\>\dl_{\si\<,\<\>\tau}\,.
\vv-1.2>
\eeq
\vv-.2>
\qed
\end{prop}

\begin{prop}
\label{Cauchy}
Cauchy formula holds,
\vvn-.4>
\beq
\label{Cxy}
\sum_{\si\in S_n\!}\,(-1)^\si A_\si(\xx)\,A_{\si\si_0}(\yy)\,=\,
\prod_{i=1}^{n-1}\,\prod_{j=1}^{n-i}\,(y_i\<-x_j)\,.
\vv-1.4>
\eeq
\vv-.2>
\qed
\end{prop}

\vsk.1>
For any \,$f\in\C[\xx]$ \,and \,$\si\in S_n$\,, define
\,$f_\si\in\C[\xx]^{S_n}$ by the rule
\vvn.3>
\beq
\label{fsi}
f_\si(\xx)\,=\,(-1)^{\si\si_0}\<\>
\Dx_{\si_0}\bigl(f(\xx)\>A_{\si\si_0}(\xx_{\si_0})\bigr)\,.
\vv.3>
\eeq

\begin{prop}
\label{free}
For any \,$f\in\C[\xx]$\,,
\vvn.1>
\beq
\label{fAsi}
f(\xx)\,=\,\sum_{\si\in S_n\!}\,f_\si(\xx)\,A_\si(\xx)\,.
\vv-1.4>
\eeq
\vv-.2>
\qed
\end{prop}

Thus \,$\C[\xx]$ \,is a free module over \,$\C[\xx]^{S_n}$ of rank \,$n\<\>!$
with a basis given by Schubert polynomials.

\goodbreak
\vsk.2>
Recall the notation from Section~\ref{Nts}, and \,$\Imi,\>\Ima\!\in\Il$\,,
\vvn.4>
\be
\Imi\>=\,\bigl(\{1\lc\la_1\}\lc\{n-\la_N\<+1\lc n\}\bigr)\,,
\vv.1>
\ee
\be
\Ima\>=\bigl(\{n-\la_1+1\lc n\}\lc\{1\lc\la_N\}\bigr)\,.
\vv.4>
\ee
For \,$I=(I_1\lc I_N)\in\Il$\,,
\,$I_j=\{\<\>i_{j,1}\<\<\lsym<i_{j,\<\>\la_j}\}$\,, define the permutations
\,$\si^I$,
\vvn.3>
\be
\si^I(k)\>=\>i_{\jk-\<\la^{(j-1)}}\,,\qquad k\in\Imi_j,\quad j=1\lc N\>,
\kern-2em
\vv.4>
\ee
and \,$\si_I\<=\si^I(\si^{\Ima})^{-1}$\>. Then \,$\si^I(\Imi)=\si_I(\Ima)=I$\,.

\vsk.3>
Let \,$\Sla\!\subset S_n$ \>be the isotropy subgroup
\vv.5>
of \,$\Imi$.

\begin{lem}
\label{Asil}
For any \,$I\in\Il$\,, we have \,$A_{\si^I}(\xx)\in\C[\xx]^{\Sla}$.
\vv.6>
\qed
\end{lem}

For example,
\vv.4>
\,$A_{\si^{\Ima}}(\xx)=\prod_{\<\>a=1}^{N\<-1}\>
\prod_{\>i\in\Imi_a}\>x_i^{N\<-a}$\>.

\begin{prop}
\label{orthI}
For any \,$I,\<J\<\in\Il$\,,
\vvn.3>
\beq
\label{AsIt}
\Dx_{\si^{\Ima}}\bigl(A_{\si^I}(\xx)\>A_{\si_J}(\xx_{\si_0})\bigr)\,=\,
(-1)^{\si_I}\,\dl_{\IJ}\,.
\vv-1.2>
\eeq
\vv-.2>
\qed
\end{prop}

\begin{prop}
\label{CauchI}
We have,
\vvn.4>
\beq
\label{CxyI}
\sum_{I\in\Il\!}\,(-1)^{\si^I}\!A_{\si^I}(\xx)\,A_{\si_I}(\yy_{\si_0})\,=\!
\prod_{1\le a<b\le N}\>\prod_{i\in\Imi_a\!}\>\prod_{j\in\Imi_b\!}
(y_j\<-x_i)\,.
\vv-1.2>
\eeq
\vv-.2>
\qed
\end{prop}

\begin{prop}
\label{fAsil}
For any \,$f\in\C[\xx]^{\Sla}$, \,we have
\vvn.2>
\beq
\label{fAsI}
f(\xx)\,=\,\sum_{I\in\Il\!}\,f_{\si^I}(\xx)\,A_{\si^I}(\xx)\,,
\vv.1>
\eeq
that is, in formula \Ref{fAsi}, \,$f_\si=0$ \,unless \,$\si=\si^I$ for some
\,$I\in\Il$\,, \>and
\vvn.5>
\beq
\label{fsI}
f_{\si^I}(\xx)\,=\,(-1)^{\si_I}\<\>
\Dx_{\si^{\Ima}}\bigl(f(\xx)\>A_{\si_I}(\xx_{\si_0})\bigr)\,.
\vv-1.2>
\eeq
\vv-.2>
\qed
\end{prop}

Define
\vvn-.4>
\be
R_\bla(\xx)\,=\!\prod_{1\le a<b\le N}\>
\prod_{i\in\Imi_a\!}\>\prod_{j\in\Imi_b\!}(x_i\<-x_j)\,.
\vv.2>
\ee

\begin{prop}
\label{DIma}
For any \,$f\in\C[\xx]^{\Sla}$, \,we have
\vvn.2>
\beq
\label{Dimaf}
\Dx_{\si^{\Ima}}f(\xx)\,=\,
\sum_{I\in\Il\!}\,\frac{f(\xx_{\si^I})}{R_\bla(\xx_{\si^I})}\;.
\vv-1.4>
\eeq
\vv-.2>
\qed
\end{prop}

\begin{prop}
\label{detA}
Let \,$n\ge 2$\,. Then
\vvn.1>
\beq
\label{detAIJ}
\det\bigl(A_{\si^I}(\xx_{\si^J})\bigr)_{\IJ\in\Il}=\<
\prod_{1\le i<j\le n}\!(x_j\<-x_i)^{m_\bla}\>,
\vvn-.1>
\eeq
where
\vv-.3>
\be
m_\bla\>=\,\frac{2\>(n-2)\<\>!}{\la_1\<\>!\ldots\la_N\<\>!}\,
\sum_{i=1}^{N-1}\sum_{j=i+1}^N\la_i\>\la_j\,.\kern-2em
%m_\bla\>=\,\frac{(n-2)\<\>!}{\la_1\<\>!\ldots\la_N\<\>!}\,
%\Bigl(n^2\<-\sum_{i=1}^n\,\la_i^2\>\Bigr)\,.
\vv-1.4>
\ee
\vv-.2>
\qed
\end{prop}

\section{Leading Terms of Solutions and Gamma Conjecture}
\label{app A}

The formula \Ref{FHI} for the asymptotics of solutions
\,$\bigl(\Psh_I(\zz;h;\qqt\>;\kat)\bigr)_{I\in\Il}$ to the joint system of the
quantum differential equations and associated qKZ difference equations reminds
the statement of the gamma conjecture, see \cite{D1,D2,KKP,GGI,GI,GZ}.

The gamma conjecture \cite{D2,GGI} is a conjecture relating the quantum
cohomology of a Fano manifold $X$ with its topology. The quantum cohomology of $X$
defines a flat quantum connection over $\Cxs$ in the direction of first Chern class
$c_1(X)$. The connection has a regular singular point at $t=0$ and an irregular singular point
at $t=\infty$. The connection has a distinguished (multivalued) flat section $J_X(t)$
defined by Givental in \cite{G1} and called
the J-function. Under certain assumptions, the limit of the J-function:
\vvn.3>
\be
A_X\>:=\>
\lim_{t\to\infty}\frac{J_X(t)}{\langle [\on{pt}], J_X(t)\rangle} \in H^*(x)
\vv.4>
\ee
exists and defines the {\it principal asymptotic class} $A_X$ of $X$.
The gamma conjecture says that $A_X$ equals the gamma class $\hat\Ga_X$ of the
tangent bundle of $X$.

The gamma class of a holomorphic
vector bundle $E$ over a topological space $X$ is the multiplicative characteristic class, in the sense
of Hirzebruch, associated to the power series expansion $\Ga(1 + x) = 1 -\ga x
+\frac{\ga^2+\zeta(2)}2x^2+\dots $
of the gamma function at 1, where $\ga$ is the Euler constant and $\zeta(2)$ is the value at
2 of the
zeta function. In other words, the gamma class is the function that associates to a holomorphic
bundle $E$ over $X$ the cohomology class $\hat\Ga(E) = \prod_i\Ga(1+\tau_i)\in H^*(X;\R)$, where the total
Chern class
of $E$ has the formal factorization $c(E)=\prod_i(1+\tau_i)$ with the Chern roots $\tau_i$ of degree 2.
If $E$ is the tangent bundle
of $X$, we write $\hat\Ga_X$ for $\hat\Ga(E)$.
Its terms of degree $\leq 3$ are given by the formula
\vvn.3>
\begin{align}
\hat\Ga(E)\,=\,{}&
1 -\ga c_1 + \big(-\zeta(2) c_2 + \frac{\zeta(2)+\ga^2}2c_1^2\big)
\\
&\!\!{}+\big(-\zeta(3)c_3 +(\zeta(3) + \ga\zeta(2))c_1c_2 -\frac{2\zeta(3)+
3\ga\zeta(2)+\ga^3}6c_1^3\big) + \ldots{}\,,
\notag
\\[-24pt]
\notag
\end{align}
see \cite{GZ}.

\vsk.3>
Consider the equivariant gamma class of \,$\tfl$\,,
\vvn.4>
\be
\Gmh_\tfls\>=\,\prod_{i=1}^{N-1}\prod_{j=i+1}^N\,\prod_{a=1}^{\la_i}\,
\prod_{b=1}^{\la_j}\,\Ga(1+\ga_{j,b}\<-\ga_{i,a})\,
\Ga(1+\ga_{i,a}\<-\ga_{j,b}\<-h)\,.
\vv.1>
\ee
cf.~\Ref{ch cl}, and the equivariant first Chern classes
\vvn.1>
\,$c_1(E_i)=\sum_{a=1}^{\la_i}\,\ga_{i,a}$\,, \,$i=1\lc N$,
of the vector bundles \,$E_i$ \,over
\,$\tfl$ \,with fibers \,$F_i/F_{i-1}$\,,
see~\Ref{c1}. Theorem~\ref{lead} can be reformulated as follows.

\begin{thm}
[Gamma theorem for $\tfl$]
\label{thm gamma}
For $\ka=1$, the leading term of the asymptotics of the \,$q$-hypergeometric
solutions \,$\bigl(\Psh_I(\zz;h;\qqt\>;\kat)\bigr)_{I\in\Il}$ in \Ref{FHI} is
the product of the equivariant gamma class of $\;\tfl$ and the exponentials of
the equivariant first Chern classes of the associated vector bundles
\,$E_1,\dots,E_N$:
\vvn.1>
\beq
\label{FHE}
\hat \Ga_{T^*\Fla}\prod_{l=1}^N\,
\bigl(\>e^{\<\>\pii\,(\la_i-\<\>n)}\>\qti_i\bigr)^{c_1(E_i)}.
\vv->
\eeq
\qed
\end{thm}

Similarly formula \Ref{FHIo} can be reformulated as follows.

\begin{thm}
[Gamma theorem for \,$\Fla$\,]
\label{thm gamma fl}
For \,$\ka=1$\,, the leading term of the asymptotics of
the \,$q$-hypergeometric solutions
\,$\bigl(\Psh^\cirs_I(\zz;\pp\<\>;\kat)\bigr)_{I\in\Il}$ \,in~\Ref{Psho} is
the product of the equivariant gamma class of $\;\tfl$ and the exponentials
of the equivariant first Chern classes of the associated vector bundles
\,$E_1,\dots,E_N$:
\vvn.1>
\beq
\label{FHEo}
\hat\Ga_{\Fla}\prod_{l=1}^N\,
\bigl(\>e^{\<\>\pii\,(\la_i-\<\>n)}\>p_i\bigr)^{c_1(E_i)}.
\vv->
\eeq
\qed
\end{thm}

\begin{example}
Let \,$N=2$\,, \,$n=2$\,, \,$\bla=(1,1)$\,.
For $\kat=1$, the leading term of the asymptotics of the \,$q$-hypergeometric
solutions for \,$T^*P^1$ is the class
\vvn.4>
\be
(e^{-\pii}\,\qti_1)^{\<\>\ga_{1,1}}\,
(e^{-\pii}\,\qti_2)^{\<\>\ga_{2,1}}
\;\Ga(1+\ga_{2,1}\<-\ga_{1,1})\,\Ga(1+\ga_{1,1}\<-\ga_{2,1}\<-h)
\vv.5>
\ee
and
the leading term of the asymptotics of the \,$q$-hypergeometric solutions
for \,$P^1$ is the class
\vvn.4>
\be
(e^{-\pii}\,p_1)^{\<\>\ga_{1,1}}\,(e^{-\pii}\,p_2)^{\<\>\ga_{2,1}}
\;\Ga(1+\ga_{2,1}\<-\ga_{1,1})\,.
\vv.5>
\ee
\end{example}

\end{document}